\documentclass[11pt, a4paper]{amsart}  

\newif\ifpersonal
\newif\ifarxiv
\RequirePackage[l2tabu, orthodox]{nag} 
\allowdisplaybreaks[1]

\ifpersonal
\newcommand*{\personal}[1]{\textcolor{blue}{(Personal: #1)}}
\newcommand*{\todo}[1]{\textcolor{red}{(Todo: #1)}}
\else
\newcommand*{\personal}[1]{}
\newcommand*{\todo}[1]{}
\fi

\usepackage{amsmath}
 \usepackage{microtype,fixltx2e,lmodern} 
\usepackage{enumerate,comment,braket,xspace,tikz-cd,csquotes} 
\usepackage{amsxtra}
\usepackage{amscd}
\usepackage{amsthm}
\usepackage{amsfonts}
\usepackage{amssymb}
\usepackage{comment}
\usepackage{eucal}
\usepackage[all]{xy}
\usepackage{graphicx,psfrag}
\usepackage{hyperref}
\usepackage{color}
\usepackage{enumitem}
\usepackage[english]{babel}
\usepackage[utf8]{inputenc}
\usepackage{mathtools}
\usepackage{epstopdf}
\usepackage{mdframed}

\DeclareMathOperator{\NE}{NE}

\DeclareMathOperator{\vir}{vir}%

\newcounter{dummy} \numberwithin{dummy}{subsection}

\newtheorem{prop}[dummy]{Proposition}

\newtheorem{cor}[dummy]{Corollary}
\newtheorem{lem}[dummy]{Lemma}
\newtheorem{defn}[dummy]{Definition}

\newtheorem{thm}[dummy]{Theorem}

\numberwithin{equation}{subsection}

\theoremstyle{remark}
\newtheorem{remark}[dummy]{Remark}
\newtheorem{example}[dummy]{Example}

\newtheorem{construction}[dummy]{Construction}

\newcommand{\C}{\mathcal{C}}
\newcommand{\T}{\mathcal{T}}

\newcommand{\Sp}{\mathrm{Sp}}

\newcommand{\D}{\mathcal{D}}
\newcommand{\V}{\mathcal{V}}

\newcommand{\iCat}{\mathrm{Cat}_{\infty}}
\newcommand{\Fin}{\mathrm{Fin}_*}

\newcommand{\nfin}{\langle n \rangle}
\newcommand{\mfin}{\langle m \rangle}

\newcommand{\twofin}{\langle 2 \rangle}
\newcommand{\onefin}{\langle 1 \rangle}

\newcommand{\M}{\mathrm{M}}

\newcommand{\Op}{\mathcal{O}}

\newcommand{\Opmonoidal}{\mathcal{O}^{\otimes}}

\newcommand{\Cmonoidal}{\mathcal{C}^{\otimes}}

\newcommand{\Dmonoidal}{\mathcal{D}^{\otimes}}

\newcommand{\Spaces}{\mathcal{S}}

\newcommand{\Prl}{\mathcal{P}r^L}

\newcommand{\Prlstable}{\mathcal{P}r^L_{\mathrm{Stb}}}

\newcommand{\derivedkmonoidal}{\mathrm{D}(k)^{\otimes}}
\newcommand{\derivedk}{\mathrm{D}(k)}
\newcommand{\inftyone}{(\infty,1)}

\DeclareMathOperator{\CAlg}{\mathrm{CAlg}}
\DeclareMathOperator{\Mod}{\mathrm{Mod}}
\DeclareMathOperator{\Fun}{\mathrm{Fun}}
\DeclareMathOperator{\Tw}{\mathrm{Tw}}
\DeclareMathOperator{\Ext}{\mathrm{Ext}}
\DeclareMathOperator{\R}{\mathbb{R}}
\DeclareMathOperator{\rhom}{\mathbb{R}\mathrm{Hom}}
\DeclareMathOperator{\Map}{\mathrm{Map}}
\DeclareMathOperator{\Env}{\mathrm{Env}}
\DeclareMathOperator{\Nerve}{\mathrm{N}}
\DeclareMathOperator{\Qcoh}{\mathrm{Qcoh}}
\DeclareMathOperator{\Perf}{\mathrm{Perf}}
\DeclareMathOperator{\Coh}{\mathrm{Coh}^{\mathrm{b}}}
\DeclareMathOperator{\Sh}{\mathrm{Sh}}
\DeclareMathOperator{\Ktheory}{\mathrm{K}}
\DeclareMathOperator{\Gtheory}{\mathrm{G}}

\newcommand{\Mnprestable}{\mathrm{M}^{\mathrm{pre}}_{0}}
\newcommand{\Mnstable}{\overline{\mathcal{M}}_{0,n}}
\newcommand{\Mstablemonoidal}{\overline{\mathcal{M}}^{\otimes}}
\newcommand{\Mgnstable}{\overline{\mathcal{M}}_{g,n}}
\newcommand{\MnK}{\mathfrak{M}_{0,n,\beta}}
\newcommand{\MKO}{\mathfrak{M}^{\otimes}}
\newcommand{\MnKO}{\mathfrak{M}(n,\beta)}

\newcommand{\Mmonoidal}{\mathrm{M}^{\otimes}}
\newcommand{\BO}{\mathrm{B}\mathcal{O}}
\newcommand{\dst}{\mathrm{dst}_k}
\newcommand{\dgcont}{\mathrm{DGCat}^{\mathrm{cont}}}
\newcommand{\dgcontk}{\mathrm{DGCat}_k^{\mathrm{cont}}}
\newcommand{\dgcontkcompact}{\mathrm{DGCat}_k^{\mathrm{cont}, \mathrm{c}}}
\newcommand{\dgcontcompact}{\mathrm{DGCat}^{\mathrm{cont}, \mathrm{c}}}

\newcommand{\stablemaps}{\mathbb{R}\overline{\mathcal{M}}_{0,n}(X,\beta)}
\newcommand{\stablemapstwo}{\mathbb{R}\overline{\mathcal{M}}_{0,2}(X,\beta)}
\newcommand{\intcart}{\int_{\mathrm{Cart}}}
\newcommand{\intcocart}{\int_{\mathrm{coCart}}}
\newcommand{\virtualsheaf}{\mathcal{O}^{\mathrm{vir}}}

\begin{document}

\title{Brane Actions, Categorification of Gromov-Witten theory and Quantum $\Ktheory$-theory}
\author{Etienne Mann}

\address{Etienne Mann, Universit\'e d'Angers, Département de mathématiques
Bâtiment I
Faculté des Sciences
2 Boulevard Lavoisier
F-49045 Angers cedex 01
France }
\email{etienne.mann@univ-angers.fr }

\author{Marco Robalo}

\address{Marco Robalo, Sorbonne  Université. Université Pierre  et  Marie  Curie,
Institut Mathématiques de Jussieu Paris Rive Gauche, CNRS, Case 247, 4, place Jussieu,
75252 Paris Cedex 05, France }
\email{marco.robalo@imj-prg.fr}

  \thanks{E.M is supported by the grant of the Agence Nationale de la
    Recherche ``New symmetries on Gromov-Witten theories'' ANR- 09-JCJC-0104-01. and "SYmétrie miroir et SIngularités irrégulières provenant de la PHysique "ANR-13-IS01-0001-01/02}
    
  \thanks{M. R was supported by a Postdoctoral Fellowship of the Fondation Sciences Mathematiques de Paris}

\begin{abstract}
Let $X$ be a smooth projective variety. Using the idea of brane actions discovered by To\"en, we construct a lax associative action of the operad of stable curves of genus zero on the variety $X$ seen as an object in correspondences in derived stacks. This action encodes the Gromov-Witten theory of $X$ in purely geometrical terms and induces an action on the derived category $\Qcoh(X)$ which allows us to recover the  Quantum $\Ktheory$-theory of Givental-Lee. 
\end{abstract}


\personal{
\begin{center}
LIST OF TODOs TEST
\end{center}
\vspace{1cm}
Solve  PROP \ref{prop-compatiblePerfCoh}\\
1) the fact these are maps of stacks and not derived schemes, which are not representable by derived schemes
2)Why proper quasi-smooth (why of finite presentation?) pullbacks preserve bounded coherent? Meaning, why are they of finite tor amplitude.\\
\vspace{1cm}
THM \ref{thmbranes}\\
1) Re-check the proof specially the verification it is a cocartesian fibration\\
\vspace{1cm}
H-descent: I think the proof is correct because we are only describing the lax structure and for this we only need the structure sheaf. As the maps  in the diagram are closed immersions, therefore proper, then they preserve bounded coherent.\\ 
1) 
\newpage
}

\maketitle
\tableofcontents

\section{Introduction}

Gromov-Witten invariants were introduced by Kontsevich and Manin in algebraic geometry in \cite{KMgwqceg, MR1363062}. The foundations were then completed by Behrend,
Fantechi and Manin in \cite{Behrend-Manin-stack-stable-mapGWI-1996},
\cite{MR1437495} and \cite{MR1431140}. In symplectic
geometry, the definition is due to Y. Ruan and G. Tian in \cite{RTmqc}, \cite{MR1390655} and
\cite{MR1483992}.  Mathematicians developed some 
techniques to compute them: via a localization
formula proved by Graber and Pandharipande in
\cite{MR1666787}, via a degeneration formula proved
by J. Li in \cite{MR1938113} and another one called quantum Lefschetz proved by Coates-Givental
\cite{Givental-Coates-2007-QRR} and Tseng \cite{tseng_orbifold_2010}.

 These invariants can be encoded using different
mathematical structures: quantum products, cohomological field theories (Kontsevich-Manin in
\cite{KMgwqceg}), Frobenius manifolds (Dubrovin in \cite{Dtft}), Lagrangian cones and Quantum
$D$-modules (Givental \cite{MR2115767}), variations of non-commutative Hodge structures (Iritani
\cite{Iritani-2009-Integral-structure-QH} and Kontsevich, Katzarkov and Pantev in
\cite{Katzarkov-Pantev-Kontsevich-ncVHS}) and so on, and used to express different aspects of mirror symmetry. Another important aspect concerns the 
study of the functoriality of Gromov-Witten invariants via crepant resolutions or flop transitions in
terms of these structures (see \cite{MR2234886}, \cite{MR2360646}, \cite{CCIT-wall-crossingI},
\cite{CCIT-computingGW}, \cite{Bryan-Graber-2009}, \cite{MR2683208}, \cite{2013arXiv1309.4438B},
\cite{2014arXiv1407.2571B}, \cite{2014arXiv1410.0024C}, etc).\\

The goal of this project is to study a suggestion of Manin and To\"en:\\

 \emph{Can the Gromov-Witten invariants of $X$ be detected at the level of the derived category $\Qcoh(X)$?}.\\

We first recall the classical construction of these invariants. Let $X$ be a smooth projective variety (or orbifold). The basic ingredient to define GW-invariants is the moduli stack of stable maps to $X$ with a fixed degree $\beta \in H_{2}(X,\mathbb{Z})$, $\overline{\mathcal{M}}_{g,n}(X,\beta)$. The evaluation at the marked
points gives maps of stacks $ev_i :\overline{\mathcal{M}}_{g,n}(X,\beta) \to X$ and
 forgetting the morphism and stabilizing the curve gives a map
$p:\overline{\mathcal{M}}_{g,n}(X,\beta) \to \overline{\mathcal{M}}_{g,n}$. 

To construct the invariants, we integrate over ``the fundamental class'' of the moduli stack
$\overline{\mathcal{M}}_{g,n}(X,\beta)$. For this integration to be possible, we need this moduli stack to be
proper, which was proved by Behrend-Manin \cite{Behrend-Manin-stack-stable-mapGWI-1996} and some form of smoothness. In general, the stack $\overline{\mathcal{M}}_{g,n}(X,\beta)$ is not smooth and has many components with different dimensions. Nevertheless and thanks to a theorem of Kontsevich  \cite{MR1363062}, it is quasi-smooth - in the sense that locally it looks like the intersection of two smooth sub-schemes of smooth scheme. In genus zero however this stack is known to be smooth under some assumptions on the geometry of $X$, for instance, when $X$ is the projective space or a Grassmaniann, or more generally when $X$ is convex, \textit{i.e.}, if for any map $f:\mathbb{P}^1\to X$, the group $\mathrm{H}^1(\mathbb{P}^1, f^*(\mathrm{T}_X))$ vanishes. See \cite{MR1492534}. 

Behrend-Fantechi then defined in \cite{MR1437495} a
``virtual fundamental class'', denoted by $[\overline{\mathcal{M}}_{g,n}(X,\beta)]^{\vir}$, which is a cycle in the Chow ring of
$\overline{\mathcal{M}}_{g,n}(X,\beta)$ that plays the role of the usual fundamental class. Finally, this allows us to define the maps
that encoded GW-invariants as 

\begin{align}
  \label{eq:2}
 I_{g,n,\beta}^{X}: H^{*}(X)^{\otimes n}& \to H^{*}(\overline{\mathcal{M}}_{g,n}) \\
(\alpha_{1}\otimes\ldots \otimes \alpha_{n}) &\mapsto
\mathrm{Stb}_{*}\left(\left[\overline{\mathcal{M}}_{g,n}(X,\beta)\right]^{\vir}\cup (\cup_i ev^{*}_i(\alpha_{i})) \nonumber
\right)
\end{align}

\noindent and we set

\begin{align}
  \label{eq:2-2}
 I_{g,n}^{X}:=\sum_\beta\, I_{g,n,\beta}^{X}
\end{align}

This collection of maps verifies some particular compatibilities as $n$ and $\beta$ vary, summarized in the notion of cohomological field theory \cite[\S 6]{KMgwqceg}. \\
 
 Latter, in \cite{MR2040281,MR1786492}, Lee and Givental defined morphisms

 \begin{align}
 \label{eq:6}
 \mathrm{K}_{g,n,\beta}^{X}: \mathrm{K}(X)^{\otimes n}& \longrightarrow \mathrm{K}(\overline{\mathcal{M}}_{g,n}) \\
 (\gamma_{1}\otimes\ldots \otimes \gamma_{n}) &\longmapsto
  \mathrm{Stb}_{*}\left(\mathcal{O}^{\mathrm{vir}}_{\overline{\mathcal{M}}_{g,n}(X,\beta)}\otimes (\otimes_i \mathrm{ev}^{*}_i(\gamma_{i})) \right)\nonumber
 \end{align}
 
 \noindent where $\mathcal{O}^{\mathrm{vir}}_{\overline{\mathcal{M}}_{g,n}(X,\beta)}$ is an element in
 $\mathrm{G}_0(\overline{\mathcal{M}}_{g,n}(X,\beta))$ which is called the virtual structure sheaf and which
 plays a similar role to that of the virtual fundamental class. The main result of \cite{MR2040281} is
 that these morphisms satisfy the axioms of a $\mathrm{K}$-field theory. Notice that these axioms are very
 similar to the one of the cohomological field theory except the splitting axiom which is more
 elaborated as it was first explained by Givental in \cite{MR1786492}.\\

\subsection{Main results}
Our main goal in this paper is to provide a first answer to the question of Manin and To\"en and construct a system of Gromov-Witten invariants at the level of the derived category of $X$. Denote by $\Qcoh$, the derived category of quasi-coherent sheaves.
In this paper, we define dg-functors
\begin{align}
 \label{eq:7}
\mathrm{QC}_{0,n,\beta}^{X}: \Qcoh(X)^{\otimes n}& \to \Qcoh(\overline{\mathcal{M}}_{0,n}) \\
(E_{1}\otimes\ldots \otimes E_{n}) &\mapsto
\mathrm{Stb}_{*}\left(\mathcal{O}_{\mathbb{R}\overline{\mathcal{M}}_{0,n}(X,\beta)}\otimes (\otimes_i \widetilde{\mathrm{ev}}^{*}_i(E_{i})) \nonumber
\right)
 \end{align}
 where $\mathcal{O}_{\mathbb{R}\overline{\mathcal{M}}_{0,n}(X,\beta) }$ is the structure sheaf of
 the derived stack
 $\mathbb{R}\overline{\mathcal{M}}_{0,n}(X,\beta)$ which is the natural derived enhancement of the
 stack $\overline{\mathcal{M}}_{0,n}(X,\beta)$. Notice that the map
 $\mathrm{Stb}:\mathbb{R}\overline{\mathcal{M}}_{0,n}(X,\beta) \to \overline{\mathcal{M}}_{0,n}$
 and $\widetilde{\mathrm{ev}}_{i}:\mathbb{R}\overline{\mathcal{M}}_{0,n}(X,\beta) \to X$ are related with
 the classical stabilization maps  (resp. $\mathrm{ev}_{i}$) by the natural closed immersion
 $\overline{\mathcal{M}}_{0,n}(X,\beta)\hookrightarrow\mathbb{R}\overline{\mathcal{M}}_{0,n}(X,\beta)$.\\

Our main theorem says that these dg-functors satisfy the axioms of a system of Gromov-Witten invariants and that when passing to $\mathrm{K}$-theory we recover the invariants of Givental-Lee:

\begin{thm}\label{thm:Coh,field,theory}(See Prop. \ref{laxactiongwcategorification}, 
\ref{prop-compatiblePerfCoh}, \ref{prop-splitting} and  Cor. \ref{cor-comparisonwithLee}) Let $X$ be a smooth projective variety over $\mathbb{C}$.
  \begin{enumerate}
  \item The dg-functors $\mathrm{QC}_{0,n,\beta}^{X}$ satisfy the axioms of a $\Qcoh$-field theory i.e., the
fundamental class, the mapping to a point and the splitting axioms similar to the one in $\mathrm{K}$-theory.
  \item The dg-functors and the axioms on $\mathrm{QC}_{0,n,\beta}^{X}$ restrict also to $\Perf(X)$, the derived category perfect complexes \footnote{Recall that for $X$ smooth, the inclusion of $\Perf(X)$ inside $\Coh(X)$ - the derived category of bounded coherent sheaves - is an equivalence}.
    \item By applying the $\mathrm{K}$-theory functor we recover the K-theoretic GW classes of Givental-Lee of \cite{MR2040281,MR1786492} and its splitting principle.
  \end{enumerate}
\end{thm}

\medskip
 
In order to explain the strategy to prove this theorem, let us start by recalling that the collection of homology groups
$\{H_{*}(\overline{\mathcal{M}}_{g,n}),g,n \in \mathbb{N}\}$ forms a
modular operad \cite[\S 6.3]{MR1601666}. By definition, a cohomological field
theory in the sense discussed above, is an algebra over this operad (see \cite[\S
2.25]{MR1601666}). In this case, the maps $I_{g,n,\beta}^{X}$, written in the more suggestive form

 \begin{equation}
   \label{eq:3}
   H_{*}(\overline{\mathcal{M}}_{g,n}) \otimes H^{*}(X)^{\otimes (n-1)} \to H^{*}(X)
 \end{equation}

 \noindent are only expressing the action of $\{H_{*}(\overline{\mathcal{M}}_{g,n}),g,n \in \mathbb{N}\}$ and the conditions to which they are submitted are then controlled by the rules of operation of gluing curves along marked points $\overline{\mathcal{M}}_{g,n}\times \overline{\mathcal{M}}_{g',m}\to \overline{\mathcal{M}}_{g+g',n+m-2}$.
 
This operadic viewpoint is at the heart of this paper as our original goal was exactly to study the existence of this action before passing to cohomology. In fact, the definition of the $I_{g,n,\beta}$ evokes the diagrams of stacks
 
 \begin{align}
   \label{eq:4}
    \xymatrix{ & \overline{\mathcal{M}}_{g,n}(X,\beta) \ar[dl]_{(p,ev_{1}, \ldots ,ev_{n-1})} \ar[dr]^{ev_{n}} & \\
      \overline{\mathcal{M}}_{g,n} \times X^{n-1} & & X}
 \end{align}
 
\noindent and as explained in the pioneering works of Kapranov-Getzler \cite{MR1601666} the operadic structure on the homology groups $H_{*}(\overline{\mathcal{M}}_{g,n})$ is induced by the fact that the family of stacks $\Mgnstable$ forms itself an operad in stacks $\Mstablemonoidal$, with composition given by gluing curves along marked points. One can hope to investigate if the diagrams (\ref{eq:4}) seen as morphisms

$$
\xymatrix{\overline{\mathcal{M}}_{g,n} \times X^{n-1}\ar@{-->}[r]& X}
$$
 
\noindent in the category of correspondences in stacks can themselves be seen as part of an action of $\Mstablemonoidal$ and if the action in (\ref{eq:3}) is induced by this new one after passing to cohomology. There is however an immediate problem that appears if we restrict ourselves to work in the setting of usual stacks: the resulting yoga of virtual classes does not agree with the one Gromov-Witten theory requires. In fact, as it understood today, virtual classes are not natural to this setting but instead, they are part of the framework of derived algebraic geometry \cite{MR3285853,toen-vezzosi-hag2}. Thanks to Sch\"urg-To\"en-Vezzosi \cite{2011-Schur-Toen-Vezzosi}, the Deligne-Mumford stack $\overline{\mathcal{M}}_{g,n}(X,\beta)$ has a natural quasi-smooth derived 
enrichment $\mathbb{R}\overline{\mathcal{M}}_{g,n}(X,\beta)$ whose structure sheaf
$\mathcal{O}_{\mathbb{R}\overline{\mathcal{M}}_{g,n}(X,\beta)}$ is expected to produce to the virtual cycle of
Behrend-Fantechi via a Chern character yet to be defined - see the discussion at the end of this introduction and \cite[\S 3.1]{MR3285853}.

In this case one can replace the diagrams \eqref{eq:4} by their natural derived version, 
\begin{equation}
\label{maincorrespondence}
\xymatrix{
&\R\Mgnstable(X,\beta)\ar[dr]\ar[dl]&\\
\Mgnstable\times X^{n-1} & & X
}
\end{equation}
noticing that as $\overline{\mathcal{M}}_{g,n}$ and $X$ are smooth they don't have non-trivial derived enhancements. These new diagrams are main the protagonists of our main result, which is a highly non-trivial geometric phenomena behind theorem \ref{thm:Coh,field,theory}:

\begin{thm}(See Theorem \ref{laxactiongwgraded})
\label{laxactiongwnaive}
Let $X$ be a smooth projective complex variety. Then $X$ seen as an object in correspondences in derived stacks carries a lax associative action of the operad of stable curves of genus zero $\Mstablemonoidal_0$ with multiplication given by the correspondences

\begin{equation}
\label{laxcorrespondences}
\xymatrix{
&\coprod_\beta\R\overline{\mathcal{M}}_{0,n}(X,\beta)\ar[dl]\ar[dr]&\\
\overline{\mathcal{M}}_{0,n}\times X^{n-1}&& X
}
\end{equation}

\noindent This action is lax associative, with lax structure given by the gluing maps

\begin{equation}
\label{explicitformulalaxstructure}
\resizebox{1. \hsize}{!}{
\xymatrix{
(\coprod_\beta \R\overline{\mathcal{M}}_{0,n}(X,\beta)\times_X (\coprod_\beta \R\overline{\mathcal{M}}_{0,m}(X,\beta))\ar[r]& (\coprod_\beta \R\overline{\mathcal{M}}_{0,n+m-2}(X,\beta))\times_{\overline{\mathcal{M}}_{0,n+m-2}}(\Mnstable\times \overline{\mathcal{M}}_{0,m})
}
}
\end{equation}

\noindent $\forall n, m\geq 2$, which are not equivalences. Moreover, this action respects the gradings.
\end{thm}

The structure of lax associativity (meaning, the family of maps (\ref{explicitformulalaxstructure}) and their compatibility under gluings) is the mechanism that encodes the shape of the splitting principle. The fiber product in the left hand side of \ref{explicitformulalaxstructure} appears naturally from the gluing of a stable map with $n$ marked points to a stable map with $m$ marked points.  The fiber product on the right hand side appears when we consider directly stable maps with $n+m-2$ marked points. The failure of the maps \ref{explicitformulalaxstructure} to be equivalences means that this action is not associative in the strongest naive sense. This failure is due to the presence on the right hand side of stable maps glued along chains of trees of $\mathbb{P}^1$'s of arbitrary length, which disappear after stabilization.  As we shall explain in section \ref{section-quantumKfromlax}, the maps \ref{explicitformulalaxstructure} are surjective and in fact the left hand side is the zero level of an hypercover of the right hand side and these excess of trees of $\mathbb{P}^1$'s correspond to the higher codimension levels of this hypercover. Moreover, these excess trees are exactly the metric corrections introduced by Givental-Lee in Quantum $\Ktheory$-theory \cite{MR2040281,MR1786492} to explain the splitting principle. \footnote{In ordinary cohomological Gromov-Witten invariants these lax structures becomes equivalences as the maps (\ref{explicitformulalaxstructure}) are birational and surjective.}\\

The Theorem \ref{thm:Coh,field,theory} is then obtained by applying the functor $\Qcoh$ to the theorem \ref{laxactiongwnaive}.\\

 \subsection{Plan of the Paper}
 \label{section-strategy}
 In this section we sketch the plan of the paper together with the strategy to prove the theorems \ref{laxactiongwnaive} and \ref{thm:Coh,field,theory}.\\
 
A first problem that we face when trying to prove theorem \ref{laxactiongwnaive} is that by working with derived stacks we are automatically pushed into the setting of higher categories \cite{lurie-htt, lurie-ha} where everything works up to specifying homotopies, homotopies between homotopies, and so on. In this setting the process of assembling the diagrams (\ref{maincorrespondence}) as part of an action of the operad of stable curves becomes more sophisticated. In practical terms, operads have to be replaced by $\infty$-operads \cite{lurie-ha} and actions can no longer be constructed by hand. The only solution is to prescribe some assembly mechanism that produces and ensures these coherences for free.  For this purpose in we will explore the idea of \emph{brane actions} discovered by To\"{e}n in \cite[Theorem 0.1]{1307.0405}. Section \ref{section-braneactions} is dedicated to reproducing the results of To\"{e}n in the setting of higher operads. More precisely, we show (see theorem \ref{thmbranes}) that if if $\Op$ is a single-colored $\infty$-operad in spaces that is \emph{coherent} - in the sense introduced by Lurie in \cite{lurie-ha} - then its space of binary operations, seen as an object in the category of co-correspondences in spaces, has a natural $\Op$-algebra structure. The idea that brane actions are related to Gromov-Witten invariants was first suggested in \cite{1307.0405} where the present work was first announced.\\

\begin{remark}
All the results in this paper concern Gromov-Witten invariants in genus zero. To adapt this results to higher genus one would first need to developed the foundations of modular $\infty$-operads, replacing dendroidal sets by a more general notion of graphical sets. We believe that our results will also work in this setting, in particular the idea of brane action. We leave this for future works. One also expects this action to be compatible with the structure of cyclic $\infty$-operad, as recently studied in \cite{1611.02591}.
 \end{remark}

Section \ref{section-stableactionsgw} is dedicated to the proof of the theorem \ref{laxactiongwnaive}. This proof will require several technical stages. Mainly, in order to make the spaces of stable maps and the correspondences (\ref{maincorrespondence}) appear as part of a brane action we will need to consider a certain modified version of the operad of stable curves, introduced by Costello in  \cite{Costello-HighergenusGW2006}. Denote by $\NE(X)\subset \mathrm{H}_{2}(X,\mathbb{Z})$ the class of effective curves in $X$. For any $(n,\beta)\in \mathbb{N}\times\NE(X)$, denote by $\mathfrak{M}_{0,n,\beta}$ the moduli stack classifying nodal curves of genus $0$ with $n$ marked points where each
 irreducible component comes with the data of an element of $\NE(X)$ and the sum of all these gradings is $\beta$. Depending on these gradings we impose certain stability conditions. These are smooth Artin stacks. See Section \ref{stacksofcostello}. The collection of these moduli spaces defines a graded $\infty$-operad in the $\infty$-topos of derived stacks $\MKO$, with grading given by $\NE(X)$. In order to make sense of this we will have to define a notion of graded operads in the setting of higher categories and prove that the construction of brane actions of Section \ref{section-braneactions} extends to this graded context. This is done in Sections \ref{gradedoperads} and \ref{gradedbraneactiontopos}.

 Moreover we have maps 
 $\mathbb{R}\overline{\mathcal{M}}_{0,n}(X,\beta) \to \mathfrak{M}_{0,n,\beta}$ sending a stable map $(C,f)$
 to the curve  $C$ (without stabilizing) and marking each irreducible component $C_{i}$ with the grading $\beta_{i}:=[(f\mid_{C_{i}})_{*}C_{i}] \in \NE(X)$. 
In this case we have diagrams given by evaluation at the marked points

\begin{equation}
\label{Costellocorrespondence}
\xymatrix{
&\R\Mnstable(X,\beta)\ar[dr]\ar[dl]&\\
\mathfrak{M}_{0,n,\beta}\times X^{n-1} & & X
}
\end{equation}

The main technical result behind the theorem \ref{laxactiongwnaive} is the following result:

\begin{thm}
\label{mainthm} (See Cor. \ref{CostelloactiononX} and Prop. \ref{corollary-stableaction}).
Let $X$ be a smooth projective algebraic variety. Then $X$, seen as an object in the $\infty$-category of correspondences in derived stacks, has a natural structure of $\MKO$-algebra induced by the brane action of this operad, given by correspondences in the formula (\ref{Costellocorrespondence}).
\end{thm}

\begin{remark}
The action in the theorem \ref{mainthm} is strongly associative, in the sense that the lax structure analogue of the maps \ref{explicitformulalaxstructure} appearing naturally in this theorem are equivalences. As we shall see right below (in the formula \ref{laxstructure}) the appearance of a weaker form of associativity is hidden in the passage from the operad of Costello to the usual operad of stable curves.
\end{remark}

We believe that this result is the fundamental mechanism behind the organization of Gromow-Witten invariants.\\

Finally, in Section \ref{section-gwaction} we explain how to pass from this action to a lax action of the usual operad of stable curves $\Mstablemonoidal$, thus concluding the proof of theorem \ref{laxactiongwnaive}. Indeed, there is a map of operads $\MKO\to \Mstablemonoidal$ obtained by stabilizing the curve and forgetting the gradings. The key observation is that at the level of correspondences in derived stacks this map can also be seen as a lax associative map of operads in correspondences from $\Mstablemonoidal$ to $\MKO$, given by correspondences

\begin{equation}
\xymatrix{
&\coprod_\beta \mathfrak{M}_{0,n,\beta}\ar[dl]\ar@{=}[dr]&\\
\overline{\mathcal{M}}_{0,n}&&\coprod_\beta \mathfrak{M}_{0,n,\beta}
}
\end{equation}

This is not strongly associative, as the natural maps

\begin{equation}
\label{laxstructure}
\mathfrak{M}_{0,n,\beta}\times \mathfrak{M}_{0,m,\beta'}\to \mathfrak{M}_{0,n+m-2,\beta+\beta'}\times_{\overline{\mathcal{M}}_{0,n+m-2}}(\overline{\mathcal{M}}_{0,n}\times \overline{\mathcal{M}}_{0,m})
\end{equation}

\noindent are not invertible. On the right hand side we have pre-stable curves
that are glued via a tree of $\mathbb{P}^{1}$'s with two marked points whereas for the left hand side of
 we only have pre-stable curves that are glued directly. After stabilizing the two become the same. \\

We will denote this lax associative map as $\Mstablemonoidal\leadsto  \MKO$. The Theorem \ref{laxactiongwnaive} is obtained by composing the action of the theorem \ref{mainthm}
with this map of operads.\\

\personal{
\begin{remark}
To give a concise sense to this lax action one would need many aspects of the theory of $(\infty,2)$-categories and operads therein. As this is not currently available in the literature the Theorem \ref{laxactiongwnaive} will be formulated in a less evident but equivalent form - see Theorem \ref{laxactiongw}. We hope to improve this formulation in future versions and in particular to include more $(\infty,2)$-categorical details.\\
\end{remark}
}

In Section \ref{section-categorification}  we address the categorification of Gromov-Witten invariants as suggested by B. To\"{e}n in \cite{MR2483943} and prove Theorem \ref{thm:Coh,field,theory}. Namely, we explain how the actions of the Theorems \ref{mainthm} and \ref{laxactiongwnaive}  both pass to the derived category $\Qcoh(X)$ by taking pullback-pushforward along the correspondences (\ref{laxcorrespondences}) and how these actions restrict to the (dg)-derived categories of coherent $\Coh$ and perfect complexes $\Perf$, by taking pullback-pushfoward along the correspondences (\ref{laxcorrespondences}). These restricted action follow essentially formally from the theorem \ref{laxactiongwnaive}. The non-formal result is in  section \ref{section-quantumKfromlax} and concerns an explicit description of the splitting principle for $\Perf$ in terms of an h-descent theorem in derived algebraic geometry.\\

Theorem \ref{thm:Coh,field,theory}-(2) was the main motivation to this work. Y. Manin suggested to us that this action could be happening at the level of non-commutative motives - see also \cite{Manin-Gromov}. This is now a consequence of our results and the theory developed in \cite{MR3281141}.  However, it seems in fact more interesting that this action happens before motives, and even before the world of derived categories. It happens in the geometric world of derived stacks and correspondences between them as explained by our theorem \ref{laxactiongwnaive}.\\

Recall that by definition the $\Gtheory$-theory of $X$ is the $\Ktheory$-theory spectrum of the dg-category $\Coh(X)$. In the end of Section \ref{section-quantumKfromlax} and in Section \ref{section-comparisonourswithLee} we show the third part of theorem \ref{thm:Coh,field,theory}, namely, that applying the $\Gtheory$-theory functor to Thm \ref{thm:Coh,field,theory}-(2) we obtain a lax associative $\{\Gtheory(\overline{\mathcal{M}}_{0,n})\}_n$ -algebra structure on the spectrum $\Gtheory(X)$. This action recovers the formulas of Quantum $\Ktheory$-theory of Givental-Lee of \cite{MR2040281,MR1786492} and the lax structure explains the metric. 
To conclude the paper, in Section \ref{section-comparisonourswithLee} we observe that  the $\Ktheory$-theoretic GW classes obtained in Thm \ref{thm:Coh,field,theory}-(3) are exactly the same as the GW classes introduced by Givental-Lee. This follows from results of \cite{MR2496057} and \cite{1208.6325}. See Cor. \ref{cor-comparisonwithLee}.



\subsection{Future Directions}

\begin{center}
\textit{Connection with usual Gromov-Witten invariants in cohomology}\\
\end{center}

We believe the action given by the Theorem \ref{mainthm} is the main mechanism behind the yoga of
Gromov-Witten invariants. In fact, if a theory of Chow groups for derived Deligne-Mumford stacks
were already available with the correct pushforward functoriality, then, applying these Chow groups
directly to the action in correspondences, by deformation to the normal cone (also to be developed)
we believe we would immediately be able to recover the usual Gromov-Witten invariants in the sense
constructed by Beherend-Fantechi-Manin using virtual cycles.
However, a more interesting question remains: can the usual cohomogical Gromov-Witten invariants also be detected directly from the categorical action of the Theorem \ref{thm:Coh,field,theory} using some version of the derived category of $X$? Our result gives us an action in the periodic cyclic homology $\mathrm{HP}(X)$ but we don't know how to identify it with the usual invariants. A first ingredient to achieve this comparison is a Riemann-Roch theorem for derived quasi-smooth Deligne-Mumford stacks. However, even with this at our disposable, we don't know how to complete the comparison as the invariants obtained seem to live in the twisted sectors. We will investigate this in future works.\\

\personal{
\begin{center}
\textit{Classification of Categorical Field Theories}\\
\end{center}
Classical Gromov-Witten invariants are encoded by a
cohomological field theory. Givental conjectured a very nice reconstruction theorem for semi-simple
cohomological field theories and Teleman classified them, proving Givental's
expectations. Nevertheless, the classification by Teleman is quite unclear. Our results in
in this paper suggest the notion of a \emph{Categorical Field Theory} as a more suitable
mathematical object. One expects that an interpretation of Teleman's classification theorem in this
setting might enlighten the proof. We also believe this can bring a new understanding of the
Dubrovin conjectures.\\
}

\begin{center}
\textit{Quantum Product on categories of matrix factorizations}\\
\end{center}

Starting with a weighted homogeneous polynomial $W$, the moduli space of $W$-spin curves is a
smooth Deligne-Mumford stack but to define Fan-Jarvis-Ruan-Witten invariants, the same authors
construct two vector bundles whose ``difference'' provides an ersatz of a virtual fundamental
class. Notice that a similar situation appears also in Gromov-Witten theory: indeed the moduli space
$\overline{\mathcal{M}}_{g,n}(X,0)\simeq \overline{\mathcal{M}}_{g,n}\times X$ is smooth but its virtual fundamental class is not the
fundamental class as one needs to twist by the top Chern class of the Hodge bundle to get the
correct numbers. We think that a similar phenomena exists for spin curves, namely there is a
canonical
derived structure on it that is hidden. If one defines the correct derived structure, one expects to produce a structure of a lax action of $\{\Qcoh(\overline{\mathcal{M}}_{g,n})\}$ on the category of  matrix factorization of $W$. One expects this will provide a geometrical explanation for the results in \cite{MR3210178}.\\

\subsection{Acknowlegments} The original motivation for this project was the idea of categorifying Gromov-Witten as suggested by Bertrand To\"{e}n in \cite{MR2483943}. This was the topic of a weekly seminar ``Gromov-Witten invariants and Derived Algebraic Geometry'' organized in the University of Montpellier in 2012. It was there that the idea of brane actions appeared, together with the insight of their connection to Gromov-Witten theory. The authors then took the task of adapting these results to Gromov-Witten theory. We would like to thank everyone involved in the seminars and in many posterior discussions: Anthony Blanc, Benjamin Hennion, Thierry Mignon and B. To\"{e}n. We would particularly like to express our gratitude to B. To\"{e}n for his multiple comments and mathematical insights and for the continuous support during the long gestation period of this paper. The key idea of brane actions is due to him. Also, special thanks to Benjamin Hennion who gave us crucial insights during different stages of the project. We would also like to express our gratitude to  D. Calaque, A. Chiodo, H. Iritani,  M. Kontsevich, Y. Manin, T. Pantev, M. Porta, T. Sch\"{u}rg, G. Vezzosi and T. Yue YU for conversations, email changes and many insights related to this project.  \\
Finally we would like to thank the anonymous referee for this comments, corrections and suggestions that helped improving the paper.

\subsection{Prerequisites}
We assume the reader is familiar with derived algebraic geometry (in the sense of \cite{MR3285853,toen-vezzosi-hag2}) and with the tools of higher category theory and higher algebra \cite{lurie-ha, lurie-htt}, particularly with the theory of $\infty$-operads.\\

\section{Brane actions}
\label{section-braneactions}

We start by providing an alternative construction of the brane action of \cite[Thm 0.1]{1307.0405} that has the advantage of being formulated purely in terms of $\infty$-categories, avoiding stricitification arguments. This reformulation is crucial to the proofs given in this paper.

\subsection{Brane actions for $\infty$-operads in spaces}

The sphere $S^n$ is a $E_{n+1}$-algebra in the category of cobordisms of dimension $n+1$. In the case when $n=1$ the multiplication map is given by the pair of pants. Recall that $S^n$ is the space of binary operations of the topological operad $E_{n+1}$. These are the standard examples of the so-called \emph{brane actions}, where an operad acts on its space of binary operations. In \cite{1307.0405} brane actions were constructed for any monochromatic $\infty$-operad satisfying some mild conditions (being of configuration type, or, equivalently, coherent). In order for this generalization to have a sense, we need to understand brane actions not at the level of cobodisms but rather at the level of co-spans. The construction requires some non-trivial strictification arguments. In this section we provide an alternative construction that avoids strictifications. As we shall explain, this action is deeply related to the definition of co-correspondences and to the definition of coherent operad introduced by J.Lurie in \cite{lurie-ha}.

\subsubsection{ Algebras in Correspondences and Twisted Arrows}
\label{correspondencesandtwistedarrows}

Let $\C$ be an $(\infty,1)$-category with finite limits. Then we can form an $(\infty,2)$-category of correspondences in $\C$, which we will denote as $\mathrm{Spans}_1(\C)$. This was constructed in \cite[Section 10]{1212.3563}. Moreover, admits a symmetric monoidal structure where every object is fully dualizable - see \cite[Thm 1.1]{1409.0837}. Dually, if $\C$ has finite colimits one can also form an $(\infty,2)$-category of co-spans in $\C$. We dispose of two canonical functors

$$
\C\to \mathrm{Spans}_1(\C)
$$
\noindent and

$$
\C^{op}\to \mathrm{Spans}_1(\C)
$$

\noindent both given by the identity on objects. The first sends a map $f:X\to Y$ in $\C$ to the correspondence $X=X\to Y$ and the second sends $Y\to X$ in $\C^{op}$ given by $f$, to $Y\leftarrow X=X$. The canonical functor $\C^{op}\to \mathrm{Spans}_1(\C)$ has a universal property - it is universal with respect to functors out of $\C^{op}$ to an $(\infty,2)$-category and satisfying base-change pullback-pushfoward. More precisely, whenever we have an $\infty$-functor $F:\C^{op}\to \D$ with $\D$ an $(\infty,2)$-category such that 

\begin{enumerate}
\item for every morphism $f:X\to Y$ in $\C$, the 1-morphism $F(f)$ has a right adjoint $f_*$ in $\D$;
\item for every pullback square in $\C$

$$
\xymatrix{
X\ar[r]^f\ar[d]^g& Y\ar[d]^u\\
X'\ar[r]^v& Y'
}
$$

\noindent the canonical $2$-morphism in $\D$

$$
F(v)\circ u_*\to g_*\circ F(f)
$$

\noindent is an equivalence.
\end{enumerate}

\noindent then $F$ extends in a essentially unique way to an $\infty$-functor of $(\infty,2)$-categories $\mathrm{Spans}_1\to \D$, informally defined by sending a correspondence

$$
\xymatrix{
&\ar[dl]_a Z \ar[dr]^b&\\
X&&Y
}
$$

\noindent to the 1-morphism in $\D$ given by $b_*\circ F(a)$.\\

A precise proof of this fact is given in \cite[Part V]{Gaitsgory-Nick-book}. In our case we will be mostly concerned not with functors out of $\mathrm{Spans}_1(\C)$ but instead, with functors with values in $\mathrm{Spans}_1(\C)$. 

We now recall a characterization of the maximal $(\infty,1)$-category inside $\mathrm{Spans}_1(\C)$ (which we will denote as $\C^{\mathrm{corr}}$ for simplicity). For this purpose we have to recall some notation: Let $\D$ be an $(\infty,1)$-category. In this case we can define a new $(\infty,1)$-category $\Tw(D)$ as follows:

\begin{itemize}
\item objects are morphisms in $\D$;
\item a morphism from $u:X\to Y$ to $v:A\to B$ is a commutative diagram

$$
\xymatrix{
X\ar[r]\ar[d]^u &A\ar[d]^v \\
Y&\ar[l] B
}
$$

\end{itemize}

This definition can be made precise and defines a new $(\infty,1)$-category - so called of twisted arrows in $\D$. See \cite[Section 5.2.1]{lurie-ha}. It is also important to remark that the assignment $\C\mapsto \Tw(\C)$ can be seen as an $\infty$-functor

$$
\Tw:\iCat\to \iCat
$$

\noindent which commutes with all small limits \cite[5.2.1.19]{lurie-ha}.

The main reason why we are interested in twisted arrows is the following universal property:

\begin{prop}
\label{label1}
Let $\C$ be an $(\infty,1)$-category with finite limits and let $\D$ be an $(\infty,1)$-category. There is a canonical equivalence between the space of $\infty$-functors $F:\D\to \C^{\mathrm{corr}}$ and the space of $\infty$-functors $\tilde{F}:\Tw(\D)\to \C$ having the property that for any pair of morphisms in $\D$, $f:x\to y$ and $g:y\to z$, the object $\tilde{F}(g\circ f)$ is the fiber product of $\tilde{F}(g)$ and $\tilde{F}(f)$ over $\tilde{F}(\mathrm{Id}_y)$.

\end{prop}

\begin{remark}
A dual result holds when we replace correspondences by co-correspondences. This corresponds to replacing $\C$ by $\C^{op}$.
\end{remark}

This result can be used as a definition of $\C^{\mathrm{corr}}$. See \cite{1301.4725} and more recently \cite{1409.0837} The result also appears in the appendix of \cite{Samthesis}.

For our purposes we will need a monoidal upgrade of the previous proposition: As $\C$ admits finite limits, $\C^{\mathrm{corr}}$ acquires a symmetric monoidal structure (remark that it won't be a cartesian monoidal structure). Let us denote it as $\C^{\mathrm{corr},\otimes_{\times}}$. Let now $\D$ have a symmetric monoidal structure $\Dmonoidal$. Then \cite[5.2.2.23]{lurie-ha} shows that $\Tw(\D)$ inherits a symmetric monoidal structure induced from $\D$, $\Tw(\D)^{\otimes}$. Objectwise it corresponds to the tensor product of $1$-arrows in $\D$ \footnote{The existence of this monoidal structure is easily deduced from the fact that $\Tw$ commutes with products and therefore sends algebras to algebras $\Tw:\CAlg(\iCat)\to \CAlg(\iCat)$}. At the same time the construction mapping a category with products to its category of correspondences can also be interpreted as an $\infty$-functor $(-)^{\mathrm{corr}}:\iCat^{prod}\to \iCat$ which also commutes with products and therefore sends algebras to algebras \todo{Reconfirm all this. I don't remember this story.}. Prop. \ref{label1} can now be understood as saying that the constructions $\Tw$ and $(-)^{\mathrm{corr}}$ are adjoint\footnote{More precisely, one can proceed as in \cite{1409.0837} and see that $\Tw$ has a right adjoint and $(-)^{\mathrm{corr}}$ is a sub-functor of this adjoint.}. By this discussion, the adjunction extends to algebras and we have the following corollary:

\begin{cor}
\label{label2}
Let $\C$ be an $(\infty,1)$-category with finite limits and let $\Dmonoidal$ be a symmetric monoidal $(\infty,1)$-category. There is a canonical equivalence between the space of monoidal $\infty$-functors $F:\Dmonoidal\to \C^{\mathrm{corr},\otimes_{\times}}$ and the space of monoidal $\infty$-functors $\tilde{F}:\Tw(\D)^{\otimes}\to \C^{\times}$ having the property that for any pair of morphisms in $\D$, $f:x\to y$ and $g:y\to z$, the object $\tilde{F}(g\circ f)$ is the fiber product of $\tilde{F}(g)$ and $\tilde{F}(f)$ over $\tilde{F}(\mathrm{Id}_y)$. Here $\C^{\times}$ denotes the cartesian monoidal structure.
\end{cor}

\begin{remark}
Again replacing $\C$ by $\C^{op}$ (if $\C$ has pushouts) we can replace correspondences by co-correspondences.
\end{remark}

We will now discuss how this result allows us to describe algebras in correspondences. To start with, let $\Opmonoidal$ be a $\infty$-operad in spaces. For the moment, let us suppose that $\Opmonoidal$ has a unique color so that we can think of $\Opmonoidal$ in a more traditional form as a family of spaces $\{\Op(n)\}_{n\in \mathbb{N}}$ together with certain operations. In this case, intuitively, an $\Opmonoidal$-algebra in $\C^{\mathrm{corr}}$ consists of an object $X\in \C^{\mathrm{corr}}$ together with operations $t_n(\sigma)\in \Map_{\C^{\mathrm{corr}}}(X^n, X)$, indexed by $\sigma\in \Op(n)$. These operations are required to satisfy some coherence conditions encoded in the fact that the assignments

$$
t_n: \Op(n)\to  \Map_{\C^{\mathrm{corr}}}(X^n, X)
$$

\noindent form a map of operads up to coherent homotopies. To give a formal definition, we will see the spaces $\Op(n)$ as mapping spaces in a multicategory with a single color. This is the strategy of \cite{lurie-ha}. In this case the data of an $\Opmonoidal$ algebra in $\C^{\mathrm{corr}}$ is given as a commutative diagram that preserves inert morphisms:

\begin{equation}
\label{eq-formatofalgebrasincorrespondeces2}
\xymatrix{
\Opmonoidal\ar[dr]\ar[rr]&& \C^{\mathrm{corr},\otimes_\times}\ar[dl]\\
&\Nerve(\Fin)&
}
\end{equation}

\noindent and we recover the object $X$ as the image of the unique color of $\Opmonoidal$. We will assume that the reader is familiar with this language. We will also assume that $\Opmonoidal$ is a unital $\infty$-operad (recall that this means that the space of nullary operations is contractible).\\

As shown in \cite[2.2.4.9]{lurie-ha}, the functor that sees a symmetric monoidal $(\infty,1)$-category as an $\infty$-operad has a left adjoint: if $\Opmonoidal$ is an $\infty$-operad, then we can define a symmetric monoidal $(\infty,1)$-category - the \emph{symmetric monoidal envelope} of $\Opmonoidal$, which we will denote as $\Env(\Op)^{\otimes}$. Explicitely, if we model $\infty$-operads as $(\infty,1)$-categories over $\Nerve(\Fin)$, then $\Env(\Op)^{\otimes}$ is the pullback $\Opmonoidal\times_{\Nerve(\Fin)} \mathrm{Act}(\Nerve(\Fin))$ so that $\Env(\Op)^{\otimes}_{\onefin}$ is the subcategory $\Op_{act}^{\otimes}$ of $\Opmonoidal$ spanned by all objects and active morphisms between them (see \cite[2.2.4.1, 2.2.4.3]{lurie-ha}).

In this case, to give an $\Op$-algebra in $\C^{\mathrm{corr},\otimes_\times}$ is the same as giving a monoidal functor 

$$
\Env(\Op)^{\otimes}\to \C^{\mathrm{corr},\otimes_\times}
$$

But now, using \ref{label2}, this corresponds to the data of a strongly monoidal functor 

$$
\Tw(\Env(\Op))^{\otimes}\to \C^{\times}
$$

\noindent such that the underlying functor $\Tw(\Env(\Op))\to \C$ satisfies the pullback condition. By \cite[Prop. 2.4.1.7-(2)]{lurie-ha}, this corresponds to an $\infty$-functor

$$
\Tw(\Env(\Op))^{\otimes}\to \C
$$

\noindent satisfying the conditions of a weak Cartesian structure (see \cite[Def. 2.4.1.1]{lurie-ha}).\\

Our main interest is the case of co-correspondences. Suppose that $\C$ admits finite colimits. Then we can apply this discussion to $\C^{op}$ and the data of an $\Op$-algebra in $\C^{\mathrm{co-corr}}:=(\C^{op})^{\mathrm{corr}}$ is equivalent to the data of an $\infty$-functor

\begin{equation}
\label{eq-formatofalgebrasincorrespondeces}
\Tw(\Env(\Op))^{\otimes}\to \C^{op}
\end{equation}

\noindent satisfying the conditions of weak Cartesian structure in $\C^{op}$ and the conditions of the Corollary \ref{label2}. Recall that by definition,  we have $\Env(\Op)\simeq \Op^{\otimes}_{act}$ so that we can identify objects of $\Tw(\Env(\Op))^{\otimes}$ with sequences $(\nfin, \sigma_1:X_1\to Y_1,..., \sigma_n:X_n\to Y_n)$ of active morphisms in $\Opmonoidal$.

\subsubsection{Construction of Brane Actions}
We can now use the description of algebras in correspondences (\ref{eq-formatofalgebrasincorrespondeces2}) as functors (\ref{eq-formatofalgebrasincorrespondeces}) to construct brane actions. In this section we work with $\infty$-operads in spaces and in the next section we will extend these results to $\infty$-operads enriched in a topos. Let us introduce some notation. Let $\Opmonoidal$ be a unital $\infty$-operad and let $\sigma: X\to Y$ be an active morphism. An extension of $\sigma$ consists of an object $X_0\in \Op^{\otimes}_{\onefin}$ together with an active morphism $\tilde{\sigma}:X\oplus X_0\to Y$ such that the restriction to $X$ recovers $\sigma$. If $p$ denotes the structural projection $\Opmonoidal\to \Nerve(\Fin)$, then the canonical map $X\to X\oplus X_0$ is defined over the inclusion $\nfin:=p(X)\to \langle n+1\rangle$ that misses a single element in $\langle n+1\rangle$. The collection of extensions of $\sigma$ can be organized in a $(\infty,1)$-category $\Ext(\sigma)$. See \cite[Def. 3.3.1.4]{lurie-ha}.

\begin{remark}
\label{label3}
In the case $\Opmonoidal$ is monochromatic with color $c$, $\Op^{\otimes}_{\onefin}$ is an $\infty$-groupoid and $\sigma$ is an active map $(c,...,c)\to c$ over $\nfin\to \onefin$, $\Ext(\sigma)$ is a space and it is equivalent to the fiber over $\sigma\in \Op(n)$ of the map $\Op(n+1)\to \Op(n)$ obtained by forgetting the last input. This makes sense because the operad is unital. Here $\Op(n):=\Map^f_{\Opmonoidal_{act}}((c,...,c), c)$. In the case where $\Op(1)$ is contractible, we have $\Ext(\mathrm{Id}_c)\simeq \Op(2)$
\end{remark}

We now recall the notion of a coherent $\infty$-operad \cite[3.3.1.9]{lurie-ha}:

\begin{defn}
Let $\Opmonoidal$ be a $\infty$-operad. We say that $\Opmonoidal$ is coherent if:

\begin{enumerate}
\item $\Opmonoidal$ is unital;
\item The underlying $(\infty,1)$-category $\Opmonoidal_{\onefin}$ of $\Opmonoidal$ is a Kan Complex;
\item Suppose we are given two composable active morphisms in $\Opmonoidal$

$$
\xymatrix{
X\ar[r]_{f}& Y\ar[r]_{g}& Z
}
$$

Then the commutative diagram

$$
\xymatrix{
\Ext(\mathrm{Id}_Y)\ar[d]\ar[r]& \Ext(g)\ar[d]\\
\Ext(f)\ar[r]&\Ext(g\circ f)
}
$$

\noindent is a pushout.
\end{enumerate}
\end{defn}

\begin{thm}(Toën)
\label{thmbranes}
Let $\C=\Spaces$ be the $(\infty,1)$-category of spaces. Let $\Opmonoidal$ be a unital coherent monochromatic $\infty$-operad with a unique color $c$ and $\Op(1)\simeq *$ \footnote{Recall that the of being unital condition is equivalent to say that $\Op(0)$ is a contractible space}. Then the space $\Ext(\mathrm{Id}_c)\simeq \Op(2)$ is an $\Op$-algebra in $\Spaces^{\mathrm{co-corr}}$. More precisely, there exists a map of $\infty$-operads

\begin{equation}
\xymatrix{
\Opmonoidal\ar[dr]\ar[rr]&& \Spaces^{\mathrm{co-corr}, \otimes_{\coprod}}\ar[dl]\\
&\Nerve(\Fin)&
}
\end{equation}

\noindent sending the unique color of $\Opmonoidal$ to the space $\Ext(\mathrm{Id}_c)$.

\end{thm}

This theorem is proved in \cite[Thm 0.1]{1307.0405} using non-trivial stricification arguments. Here we suggest an alternative proof that avoids those arguments and gives a more conceptual explanation. Moreover, this new strategy will be very useful throughout the rest of this paper.\\

As discussed above, we are reduced to construct an $\infty$-functor 

\begin{equation}
\label{eq-formatofalgebrasincorrespondeces3}
\Tw(\Env(\Op))^{\otimes}\to \Spaces^{op}
\end{equation}

\noindent sending the identity map $\mathrm{Id}_c:c\to c$ seen as an object of $\Tw(\Env(\Op))^{\otimes}$ to the space $\Ext(\mathrm{Id}_c)$, an object $\sigma:(c,c,...,c)\to c$ in $\Tw(\Env(\Op))^{\otimes}$ to the space $\Ext(\sigma)$, and satisfying certain conditions.

\begin{proof}[Proof of Theorem \ref{thmbranes}:]
To produce the functor (\ref{eq-formatofalgebrasincorrespondeces3}), we can use the Grothendieck construction and instead, construct a right fibration

$$
\xymatrix{
\BO\ar[d]^{\pi}\\
\Tw(\Env(\Op))^{\otimes}
}
$$
\noindent such that the fiber over an object $\sigma:(c,c,...,c)\to c$ in $\Tw(\Env(\Op))^{\otimes}$ is the space $\Ext(\sigma)$. We construct it as follows: Start with the source map

$$
\xymatrix{
\Fun(\Delta[1], \Tw(\Env(\Op))^{\otimes})\ar[d]^{ev_0}\\
\Tw(\Env(\Op))^{\otimes}
}
$$

\noindent which we know to be a cartesian fibration via the composition of morphisms (see \cite[2.4.7.5, 2.4.7.11]{lurie-htt}). We let $\BO$ denote the (non-full) subcategory of $\Fun(\Delta[1], \Tw(\Env(\Op))^{\otimes})$ defined as follows:

\begin{enumerate}
\item its objects are those twisted morphisms 

$$
\xymatrix{
\sigma:=(\nfin, \sigma_1:X_1\to Y_1,..., \sigma_n:X_n\to Y_n)\ar[r]^(0.65){f} & \delta:=(\onefin, \delta:U\to V)
}
$$

\noindent over the unique active map $\nfin\to \onefin$ such that the corresponding twisted arrow

$$
\xymatrix{
\bigoplus_{i\in \nfin^\circ} X_i\ar[r]^(0.6){x}\ar[d]^{\oplus_{i\in \nfin^\circ}\sigma_i}& U\ar[d]^{\delta}\\
\bigoplus_{i\in \nfin^\circ} Y_i&\ar[l]^(0.4){y} V
}
$$

\noindent satisfies the following two conditions:

\begin{enumerate}
\item the active map $x:\bigoplus_{i \in \nfin^\circ} X_i\to U$ is semi-inert in $\Opmonoidal$ and is defined over one of the maps 
$$\langle m \rangle:=p(\bigoplus_{i\in \nfin^\circ} X_i)\to \langle m+1 \rangle$$

\noindent  corresponding to an inclusion that misses a single $\langle m+1\rangle$;\\

\item the map $y$ is an equivalence.\\
\end{enumerate}

\item A morphism in $\Fun(\Delta[1], \Tw(\Env(\Op))^{\otimes})$ over a morphism 

$$
\xymatrix{
\lambda:=(\langle \gamma \rangle, \lambda_1:A_1\to B_1,...,\lambda_\gamma: A_\gamma\to B_\gamma)\ar[r]^(0.8){g}& \sigma
}
$$

\noindent in $\Tw(\Env(\Op))^{\otimes}$ is a commutative square

$$
\xymatrix{
(\onefin, \omega: W\to Z)\ar[r]^(0.7){h} & \delta\\
\lambda\ar[r]^g\ar[u]^t& \sigma \ar[u]^f
}
$$

\noindent over

$$
\xymatrix{
\onefin\ar[r]^{id} & \onefin\\
\langle \gamma \rangle \ar[r]^g\ar[u]^{p(t)}& \nfin \ar[u]^{p(f)}
}
$$

\noindent such that

\begin{enumerate}
\item both $t$ and $f$ satisfy the conditions of item 1);
\item in the induced diagram

$$
\xymatrix{
W\ar[rr]^h && U\\
\ar[u]\bigoplus_{\alpha\in \langle \gamma \rangle^\circ} A_\alpha \simeq \bigoplus_{i\in \nfin^\circ} \bigoplus_{j\in g^{-1}\{i\}} A_j\ar[rr]^(0.7){g} && \bigoplus_{i\in \nfin^\circ}X_i\ar[u]
}
$$

\noindent the map $h$ sends the unique element $p(W)-p(\bigoplus_{\alpha\in \langle \gamma \rangle^+} A_\alpha)$ to the missing element in $\langle m+1 \rangle$.

\end{enumerate}

\end{enumerate}

It follows now from \cite[Def. 3.3.1.4]{lurie-ha} that the fiber of the composite $\pi:\BO\subseteq \Fun(\Delta[1], \Tw(\Env(\Op))^{\otimes})\to \Tw(\Env(\Op))^{\otimes}$ over an object $\sigma:=(\nfin, \sigma_1:X_1\to Y_1,..., \sigma_n:X_n\to Y_n)$ is the space $\Ext(\bigoplus_{i\in \nfin^\circ}\sigma_i)\simeq \coprod_{i\in \nfin^\circ} \Ext(\sigma_i)$. We remark now that $\pi$ remains a cartesian fibration under the composition of twisted morphisms. We will check it in the case when we have a single active morphism as the general case can easily be reduced to this one. Let $\sigma:X\to Y$ be an active map, and $\tilde{\sigma}:X'\to Y$ be an extension of $\sigma$ with $X\to X'$ semi-inert over the inclusion $\nfin\to \langle n+1 \rangle$ that misses a single element $a_i\in\langle n+1 \rangle^\circ$. Consider a 

$$
\xymatrix{
U\ar[r]^{g_1}\ar[d]^{\lambda}& X\ar[d]^{\sigma}\\
V&\ar[l]^{g_2} Y
}
$$

\noindent  a twisted morphism from an active morphism $\lambda$ to $\sigma$. We remark that as the operad has a unique color and $\Op(1)$ is contractible, for any $U\to U'$ semi-inert over an inclusion that misses a single element  $p(U):=\langle k \rangle \to \langle k+1 \rangle$, the space of factorizations

\begin{equation}
\label{cartesianbraneproof}
\xymatrix{
U'\ar@{-->}[rr]^-{h}&& X\\
U\ar[rr]^-{g_1} \ar[u]&& X\ar[u]\\
}
\end{equation}

\noindent where  $h$ satisfies the conditions in (2)-b), is contractible.  Indeed, the definition of $\infty$-operad together with the condition 2-b) tells us that

$$
\Map^{p(h)}_{\Opmonoidal}(U', X')\simeq \Map^{p(g_1)}_{\Opmonoidal}(U, X)\times \Op(1)\simeq \Map_{\Opmonoidal}(U, X)
$$

\noindent showing that $h$ is essentially unique once $g_1$ is given. Moreover, the same argument tells us also that as the operad has a unique color and $\Op(1)$ is contractible, all semi-inert morphisms $U\to U'$ over an inclusion that miss a single element $\langle k \rangle \to \langle k+1\rangle$ are equivalent in a canonical contractible way via the permutations that change the choice of the missing elements. 
This implies that $\pi$ is a cartesian fibration.\\

To conclude we have to show that the functor associated to $\pi$

 $$
\Tw(\Env(\Op))^{\otimes}\to \Spaces^{op}
$$

\noindent  satisfies 1) the condition of the Corollary \ref{label2} and 2) it is a weak Cartesian structure in $\Spaces^{op}$ (in the sense of \cite[2.4.1.1]{lurie-htt}. But this is exactly where the coherence conditon plays its role: 1) is equivalent to the definition of coherent $\infty$-operad and  2) follows from the fact that $\Ext((\sigma_1,..., \sigma_n))\simeq \Ext(\oplus_i \sigma_i)\simeq \coprod_i \Ext(\sigma_i)$
 (this is clear from the construction and from the definition of $\Ext(-)$ in \cite[3.3.1.4, 3.3.1.7]{lurie-ha}). In particular we have $\Ext(\mathrm{Id}_{(c,....c)})\simeq\coprod_n \Ext(\mathrm{Id}_c)$.

\end{proof}

\vspace{1cm}

\begin{remark}
\label{contractibleO1}
It follows from the cartesian morphisms in the theorem that co-correspondence 

$$
\coprod_n \Ext(\mathrm{Id}_c)  \rightarrow \Ext(\sigma)\leftarrow \Ext(\mathrm{Id}_c)
$$

\noindent induced by $\sigma\in \Op(n)$ can be canonically identified with the pullback of the diagram considered in \cite[Thm 0.1]{1307.0405}:

$$
\xymatrix{
\coprod_n \ar[dr] \Op(2)\times \Op(n)\ar[r] &\ar[d] \Op(n+1)& \ar[l] \Op(2)\times \Op(n)\ar[dl]\\
&\Op(n)&
}
$$

\noindent along the map $\sigma:*\to \Op(n)$. See the Remark \ref{label3}. Following \cite[Prop. 3.5]{1307.0405}, we say that a monochromatic unital $\infty$-operad in spaces $\Opmonoidal$ is of \emph{configuration type} if for every $n\geq 2$ and $m \geq 2$, the natural composition diagram

$$
\label{configurationtypediagram}
\xymatrix{
\Op(n)\times \Op(m+1)\coprod_{\Op(2)\times \Op(n)\times \Op(m)} \Op(n+1)\times \Op(m)\ar[d]\ar[r]& \Op(n+m)\ar[d]\\
\Op(n)\times \Op(m)\ar[r]& \Op(n+m-1)
}
$$

\noindent is a pullback. Notice that in this case the condition of  $\Opmonoidal$ being of configuration type is equivalent to being coherent - the compatibility between the spaces of extensions of two operations $\sigma:*\to \Op(n)$ and $\rho:*\to \Op(m)$ is obtained by taking the fibers of the diagram \ref{configurationtypediagram} over the map $\sigma\times \rho:*\to \Op(n)\times \Op(m)$. 
\end{remark}

\begin{example}
As shown in \cite[Thm 5.1.1.1]{lurie-ha} the $\infty$-operads $E_n^{\otimes}$ are coherent. In this case, for $\sigma\in E_n(k)$, the space $\Ext(\sigma)$ is equivalent to a wedge $\vee_k S^{n-1}$ and the brane action is given by the usual cobordism-style action. When $n=1$, the co-span

$$
\coprod_k S^1\to \vee_k S^1 \leftarrow S^1
$$

\noindent can be identified with the usual pants with $k$-legs.
\end{example}

\subsubsection{Functoriality of Extensions}
\label{functoriality}

Let $F:\Opmonoidal\to \Op^{`\otimes}$ be a map of $\infty$-operads and suppose both $\Opmonoidal$ and $\Op^{`\otimes}$ are unital monochromatic with $\Op(1)\simeq \Op(1)'\simeq *$. Then we have a natural morphism of right fibrations induced by $F$

\begin{equation}
\label{equation-branesfunctorial}
\xymatrix{
\BO\ar[d]^\pi \ar[r]^F & \mathrm{B}\Op'\ar[d]^{\pi'}\\
\Tw(\Env(\Op))^{\otimes}\ar[r]^F& \Tw(\Env(\Op'))^{\otimes}
}
\end{equation}

Indeed, as we know, both constructions $\Tw$ and $\Env$ are functorial. As the construction $\Fun(\Delta[1], -)$ is also functorial and the source map $ev_0$ is a natural transformation, we know that $F$ induces a commutative diagram of right fibrations

\begin{equation}
\label{mapbranes}
\xymatrix{
\Fun(\Delta[1], \Tw(\Env(\Op))^{\otimes}) \ar[d]^{ev_0} \ar[r]^F &\Fun(\Delta[1], \Tw(\Env(\Op'))^{\otimes})\ar[d]^{ev_0}\\
\Tw(\Env(\Op))^{\otimes}\ar[r]^F& \Tw(\Env(\Op'))^{\otimes}
}
\end{equation}

\noindent More generally, the same argument gives us an $\infty$-functor $\mathrm{B}:\mathrm{Op}_{\infty}^{*}\to \Fun^{rf}(\Delta[1], \iCat)$ where $\mathrm{Op}_{\infty}^{*}$ is the full subcategory of $\mathrm{Op}_{\infty}$ spanned by the $\infty$-operads satisfying the conditions at the beginning of this section and $ \Fun^{rf}(\Delta[1], \iCat)$  is the full subcategory of  $\Fun(\Delta[1], \iCat)$  spanned by those functors that are right fibrations.  We are now left to check that if $F$ is a map of operads then the induced map of right fibrations \ref{mapbranes} sends the full subcategory $\BO$ to $\mathrm{B}\Op'$, or in other words, $F$ preserves the conditions $(1)$ and $(2)$ in the proof of \ref{thmbranes}. But this follows immediately from the fact $F$ is a map of $\infty$-operads and maps of $\infty$-operads preserve semi-inert morphisms. This follows from the fact inert edges are cocartesian by definition. It also follows from the proof of the Thm. \ref{thmbranes} that $F$ sends $\pi$-cartesian edges to $\pi'$-cartesian edges.

\subsubsection{Some Examples and Remarks}

\begin{remark}(From Co-spans to Spans)
\label{cospanstospans}
Let $\Opmonoidal$ be a unital coherent monochromatic $\infty$-operad with color $c$. Then the Theorem \ref{thmbranes} tells us that the space $\Ext(\mathrm{Id}_c)$ is an $\Op$-algebra in $\Spaces^{\mathrm{co-corr},\otimes_{\amalg}}$. Fix now $X$ a space . Then we have a functor $\Map(-,X):\Spaces^{op}\to \Spaces$. We can use the definition of correspondences to see that this functor produces an $\infty$-functor $(\Spaces^{\mathrm{co-corr}})^{op}\to \Spaces^{\mathrm{corr}}$ which is monoidal with respect to the opposite of $\otimes_{\amalg}$ on $\Spaces^{\mathrm{co-corr}}$ and $\otimes_{\times}$ on $\Spaces^{\mathrm{corr}}$\footnote{Notice that if $\mathcal{C}$ is a category with finite products, the operation $\times$ induces a symmetric monoidal structure in correspondences in $\mathcal{C}$ but this structure is no longer cartesian as it does not verify the required universal property for the description of the mapping spaces to a product.}. In this case the space $\Map(\Ext(\mathrm{Id}_c), X)$ becomes an $\Op$-algebra in $\Spaces^{\mathrm{corr}}$ by means of the composition

$$
\Tw(\Env(\Op))^{\otimes}\to \Spaces^{op}\to \Spaces
$$

\end{remark}

\begin{example}($n$-Loop Stacks)
Let $\C$ be the $\infty$-topos of derived stacks over a field of characteristic zero, considered with the cartesian structure. See  \cite{toen-vezzosi-hag2,MR3285853}. As $\C$ is presentable we have a canonical monoidal, colimit-preserving $\infty$-functor

$$
\Spaces^{\times}\to \C^{\times}
$$

By the universal property of correspondences, this functor provides an $\infty$-functor

$$
\Spaces^{\mathrm{co-corr}, \otimes_{\amalg}}\to \C^{\mathrm{co-corr}, \otimes_{\amalg}}
$$

In particular, for any unital coherent monochromatic  $\infty$-operad (in spaces) $\Opmonoidal$, we can consider $\Ext(\mathrm{Id}_c)$ as an $\Op$-algebra in $\C^{\mathrm{co-corr}, \otimes_{\amalg}}$ via the composition.

$$
 \Opmonoidal\to \Spaces^{\mathrm{co-corr}, \otimes_{\amalg}}\to \C^{\mathrm{co-corr}, \otimes_{\amalg}}
  $$
   
   In this case, for any derived stack $X$, the mapping stack $\Map(\Ext(\mathrm{Id}_c), X)$ becomes an $\Op$-algebra in $\C^{\mathrm{corr}, \otimes_{\times}}$. In particular, when $\Opmonoidal=E_n^{\otimes}$, as $E_n(2)\simeq S^{n-1}$, we find that the mapping stack $\Map(S^{n-1}, X)$ becomes and $E_n$-algebra in correspondences.
\end{example}

\begin{example}

At the same time, when working with derived stacks which have compactly generated derived categories of quasi-coherent sheaves, $\Qcoh$ provides a monoidal $\infty$-functor $(\dst)^{op}\to \dgcontk$ where $\dgcont$ is the $(\infty,2)$-category of $k$-linear presentable dg-categories together with continuous functors as 1-morphisms. For nice enough derived stacks, such as the notion of perfect stacks introduced in \cite[Section 3.1]{MR2669705},  this functor factors through the sub-category of compactly generated dg-categories and satisfies the pullback-pushfoward base-change and the $(\infty,2)$-monoidal universal property of correspondences tells us that $\Qcoh$ factors as a monoidal functor

$$\mathrm{Spans}_1(_{nice}\dst)\to \dgcontk$$

We restrict to the maximal $(\infty,1)$-categories and obtain a monoidal functor

$$
_{nice}\dst^{\mathrm{corr}}\to \dgcontk
$$

As a corollary of the theorem and the previous example we deduce that if $X$ is a nice enough stack, the dg-category $\Qcoh(\Map(S^n, X))$ is an $E_{n+1}$-monoidal dg-category, thus recovering by a different method the result of \cite[Section 6]{MR2669705}. In \cite[Corollary 5.4]{1307.0405} this is used the prove higher formality. 
\end{example}

\subsection{Brane actions for $\infty$-operads in a $\infty$-topos}
\label{branesfortopos}
\subsubsection{\,}
We want to be able to work with operads enriched in derived stacks. The first task is to define what these objects are. Thanks to the works of  \cite{MR2366165, 1305.3658,MR2805991, 1606.03826, MR3100887, MR3100888} we can model $\infty$-operads in the sense of Lurie \cite{lurie-ha}, using the category of non-planar rooted trees $\Omega$, either via dendroidal sets (i.e. presheaves on $\Omega$) or dendroidal Segal spaces, i.e. $\infty$-functors $\Nerve(\Omega)^{op}\to \Spaces$ satisfying a local condition with respect to certain Segal maps. We use this as an inspiration to define $\infty$-operads in a hypercomplete topos. \footnote{Recall from \cite{lurie-htt} that hypercomplete topos can always be described as $\infty$-sheaves over a site.}

\begin{defn}
Let $\T$ be hypercomplete $\infty$-topos. The $(\infty,1)$-category of $\infty$-operads in $\T$ is 
$$
\mathrm{Op}_{\infty}(\T):= \Fun^{\mathrm{Segal}}(\Nerve(\Omega^{op}), \T)
$$
\end{defn}

Of course, when $\T=\Spaces$, the comparison theorem of \cite{1305.3658, 1606.03826}  tells us that $\mathrm{Op}_{\infty}(\Spaces)\simeq \mathrm{Op}_{\infty}$ and in this case, when $\T$ is the topos of sheaves in a $\infty$-site $(\C,\tau)$, we have 

$$
\mathrm{Op}_{\infty}(\T):= \Fun^{\mathrm{Segal}}(\Nerve(\Omega^{op}), \Sh(\C))\simeq \Fun^{\mathrm{Segal}, \tau}(\Nerve(\Omega^{op})\times \C^{op}, \Spaces)\simeq \Sh(\C, \mathrm{Op}_{\infty})
$$

\noindent so that, according to our definition here, an operad in $\T$ is just a sheaf of $\infty$-operads on the site $\C$. Equivalently, one can also describe $\mathrm{Op}_{\infty}(\T)$ as the $(\infty,1)$-category of limit preserving functors $\T^{op}\to \mathrm{Op}_{\infty}$: when $\T$ is a $\infty$-topos, the Yoneda inclusion provides an equivalence $\T\simeq \Fun^{Limits}(\T^{op}, \Spaces)$.\\

Let us now explain the construction of brane actions for an $\infty$-operad in a topos. Later on we will be interested in the $\infty$-topos of derived stacks over a field of characteristic zero. Let $\M^\otimes\in \Sh(\C, \mathrm{Op}_{\infty})$ be an $\infty$-operad in $\T=\Sh(\C)$. For the rest of this section we will be working under the following assumption: 

\begin{enumerate}[label=\Alph*)]
\item \label{Alph*} For each $Z \in \C$, the $\infty$-operad $\Mmonoidal(Z)\in \mathrm{Op}_{\infty}$ is unital, coherent, has a unique color, which we will designate by $c_Z$ and the underlying $(\infty,1)$-category of $\Mmonoidal(Z)$ is a contractible $\infty$-groupoid.
\end{enumerate}

   In this case we know from the discussion in the previous section that each $\Mmonoidal(Z)$ admits a brane action in $\Spaces^{\mathrm{co-corr}}$. Our task now is to understand the compatibilites between these brane actions. For that purpose we will need some preliminaries. The first observation is that as $\Mmonoidal$ has a unique color, we can use the equivalence $\theta:\mathrm{Op}_{\infty}(\T)\simeq \Fun^{\mathrm{Segal}}(\Nerve(\Omega^{op}), \T)$  (see the details in \cite{1305.3658, 1606.03826}) to think of $\Mmonoidal$ as a collection of objects in $\T$, $\{\M_n\}_{n\geq 0}$, defined by means of the following universal property: for every $Z\in \C$, we have canonical equivalences:

$$
\Map_{\T}(Z,\M_n)\simeq \theta(\Mmonoidal)(T_n)(X)\simeq \Map_{\Mmonoidal(Z)_{act}}((c_Z,..., c_Z), c_Z)
$$
\noindent  where $T_n\in \Omega$ is the non-planar rooted tree with $n$-leafs. Of course, in this case, the Yoneda's lemma gives us canonical maps in $\T$

$$
\M_n\times \M_{i_1}\times ... \times \M_{i_n}\to \M_{i_1+...+ i_n}
$$

\noindent that determine the composition operations in $\Mmonoidal$. In other words, we can think of $\Mmonoidal$ using our familiar intuition of operadic objects and their standard operations. This machine captures all the necessary coherences. Moreover, as $\Mmonoidal$ is unital, we will have $\M_0\simeq *$ so that we will also have operations $\M_{n+1}\to \M_n$ which correspond to forgeting the last input. Our assumption $\mathrm{A}$) implies also that $\M_1\simeq *$. \\

It follows from the Remark \ref{contractibleO1} that a monochromatic $\infty$-operad in a $\infty$-topos $\T$ is coherent if and only if the commutative diagram in $\T$

\begin{equation}
\label{coherencetoposmonochromatic1}
\xymatrix{
\M_n\times \M_{m+1}\coprod_{\M_2\times \M_n \times \M_m} \M_{n+1}\times \M_m\ar[r]\ar[d]& \M_{n+m}\ar[d]\\
\M_n\times \M_m\ar[r]& \M_{n+m-1}
}
\end{equation}

\noindent is cartesian.

Given $Z\in \C$, the brane action $b_Z:\Mmonoidal(Z)\to \Spaces^{\mathrm{co-corr}, \otimes_{\coprod}}$ of the Theorem \ref{thmbranes} endows the space of extensions $\Ext(\mathrm{Id}_{c_Z})$ in $\Mmonoidal(Z)$ with a structure of $\Mmonoidal(Z)$-algebra in co-correspondences of spaces. Following our assumptions, and as explained in the Remark \ref{contractibleO1}, for a given $Z$ this space is given by $\Map_\T(Z, \M_2)$. To describe the action we can also mimic the arguments of the Remark \ref{contractibleO1}. Indeed, Yoneda's lemma ensures the existence of universal diagrams in $\T$

\begin{equation}
\label{universaldiagram}
\xymatrix{
\coprod_n \M_n\times \M_2 \ar[r]\ar[dr]& \M_{n+1}\ar[d]& \M_{2}\times \M_n\ar[dl]\ar[l]\\
&\M_n&
}
\end{equation}

\noindent with the universal property: for a given $\sigma: Z\to \M_n$, the effect of the action of $\sigma$ on $\Map_\T(Z, \M_2)\simeq \Map_{\T/Z}(Z, Z\times \M_2)$ is the co-correspondence obtained by pulling back the universal diagram along $\sigma$ and taking sections over $Z$, namely

$$
\coprod_n  \Map_{\T/Z}(Z, Z\times \M_2) \to \Map_{\T/Z}(Z, C_{\sigma}) \leftarrow \Map_{\T/Z}(Z, Z\times \M_2) 
$$

\noindent where

\begin{equation}
\label{curveprototype}
C_\sigma: =Z\times_{\M_n} \M_{n+1}
\end{equation}

\begin{remark}
\label{coherencetoposmonochromatic}
The coherence criterium of the diagram (\ref{coherencetoposmonochromatic}) can now be measure in terms of the objects $C_{\sigma}$. In fact, $\Mmonoidal$ is coherent if and only if for any $Z$ and any two operations $\sigma:Z\to \M_n$ and $\tau:Z\to \M_m$, the map induced between the fibers

\begin{equation}
\label{coherencetoposmonochromatic2}
\xymatrix{
C_{\sigma}\coprod_{Z\times \M_2} C_{\tau}\ar[r] \ar[d]& C_{\sigma\circ \tau}\ar[d]\\
\M_n\times \M_{m+1}\coprod_{\M_2\times \M_n \times \M_m} \M_{n+1}\times \M_m\ar[r]\ar[d]& \M_{n+m}\ar[d]\\
\M_n\times \M_m\ar[r]& \M_{n+m-1}
}
\end{equation}

\noindent is an equivalence. 
\end{remark}

This two-step description (pulling back along $\sigma$ and taking sections), suggests that in fact this algebra structure in $\Spaces^{\mathrm{co-corr}}$ exists before taking sections, or in other words, that the space $Z\times \M_2$ is itself an $\Mmonoidal$-algebra in $(\T/Z)^{\mathrm{co-corr}}$. Intuitively, given $\sigma:Z\to M_n$, the action of $\sigma$ on $Z\times \M_2$ is simply the co-correspondence over $Z$ given by the pullback along $\sigma$ of the universal diagram \ref{universaldiagram}, namely

\begin{equation}
\label{universaldiagram2}
\xymatrix{
\coprod_n Z\times \M_2 \ar[r]\ar[dr]& C_\sigma \ar[d]& \M_{2}\times Z\ar[dl]\ar[l]\\
&Z&
}
\end{equation}

\noindent so that the brane action in co-correspondences in spaces can be recovered by applying the monoidal functor

$$
(\T/Z)^{\mathrm{co-corr}}\to \Spaces^{\mathrm{co-corr}} 
$$

\noindent which applies $\Map_{\T/Z}(Z,-)$ both at the level of objects and correspondences.\\

Let us explain how to construct this action over $Z$. We start with the following nice consequence of Rezk's characterization of $\infty$-topoi:

\begin{prop}
\label{stackofstacks}
Let $\T$ be an $\infty$-topos. The construction

$$Z\in \C^{op}\mapsto (\T/Z)^{\mathrm{co-corr}, \otimes_{\coprod}}\in \mathrm{Op}_{\infty}$$

\noindent is an $\infty$-operad in $\T$.
\begin{proof}Recall that $\T/Z$ is again an $\infty$-topos \cite[6.3.5.11]{lurie-htt}. Thanks to Charles Rezk's characterization of $\infty$-topoi \cite[6.1.6.3]{lurie-htt}, the assignment $Z\in \C^{op}\mapsto \T/Z$ is a sheaf with respect to the topology in $\C$ and admits a classifying object, which we shall denote as $\T/(-)\in \T$. In this case, as both $(-)^{op}$ and $(-)^{\mathrm{corr}}$ are functorial and are right adjoints (see section \ref{correspondencesandtwistedarrows}), they commute with limits so that the assignment  $Z\in \C^{op}\mapsto (\T/Z)^{\mathrm{co-corr}, \otimes_{\coprod}}\in \mathrm{Op}_{\infty}$  will also be a sheaf and representable in $\T$. We will denote this operadic object as  $(\T/(-))^{\mathrm{co-corr}, \otimes_{\coprod}}$.
\end{proof}
\end{prop}

In this case, the compatibilities between the different brane actions when $X$ varies are encoded by  the following result:

\begin{prop}\label{propbranestopos}
Let $\T$ be an $\infty$-topos and let $\Mmonoidal$ be an $\infty$-operad in $\T$ satisfying the assumptions in the beginning of this section. Then there exists a map of $\infty$-operads in $\T$
 
$$
\Mmonoidal\to (\T/(-))^{\mathrm{co-corr}, \otimes_{\coprod}}
$$

\noindent encoding the brane action informally described as follows: given $Z\in \C$, it sends an active map $\sigma \in \Mmonoidal(Z)$ to the operation given by the diagram (\ref{universaldiagram2}).
\end{prop}

Before giving the proof of this proposition let us recall a technical fact which will be used several times throughout the paper.

\begin{remark}
\label{remark-cartesian-cocartesian}
Let $\pi:X\to S^{op}$ be a cocartesian fibration between $(\infty,1)$-categories classifying a $\infty$-functor $p:S^{op}\to \iCat$. Then the cartesian fibration that also classifies $p$, $\int_{\mathrm{Cart}}p\to S$ can be obtained as follows: define a new simplicial set $Y$ over $S^{op}$ such that maps of simplicial sets over $S^{op}$, $\Map_{S^{op}}(T, Y)$ are in bijection with maps of simplicial sets $\Map(T\times_{S^{op}}Y, \Spaces)$. In particular, an object of $Y$ over a vertice $s\in S^{op}$ is just a presheaf on $X_s^{op}$. Let $Y_0\subset Y$ be the full subcategory of $Y$ spanned by those vertices corresponding to representable presheaves and let $Y^{op}_0\to S$ denote the opposite of the projection $Y\to S^{op}$. Then this map is a cartesian fibration that classifies $p$. Conversely, if $\alpha:X\to S$ is a cartesian fibration classifying a diagram $p$ then the cocartesian fibration classifying the same diagram can be obtained by applying these steps to the cocartesian fibration $\alpha^{op}$ and then taking the opposite of the output.
\end{remark}

\textit{Proof of Prop. \ref{propbranestopos}}

As  $(\T/(-))^{\mathrm{co-corr}, \otimes_{\coprod}}$ is a sheaf, the data of a morphism of operadic sheaves $\Mmonoidal\to  (\T/(-))^{\mathrm{co-corr}, \otimes_{\coprod}}$ equivalent to the data of a morphism in $\Fun(\C^{op}, \mathrm{Op}_{\infty})$. Using the adjunctions

$$
\xymatrix{
\Fun(\C^{op}, \CAlg(\iCat))\ar@{^{(}->}[r]&\Fun(\C^{op}, \mathrm{Op}_{\infty})\ar@/_8pt/[l]_{\Env}
}
$$

\noindent and

$$
\xymatrix{
\Fun(\C^{op}, \CAlg(\iCat))\ar[r]_{(-)^{\mathrm{corr}}}&\Fun(\C^{op}, \CAlg(\iCat))\ar@/_8pt/[l]_{\Tw}
}
$$

\noindent this is the same as the data of a natural transformation

$$
\Tw(\Env(\Mmonoidal))\to ((\T/(-))^{op})^{\times}
$$

But as in \ref{thmbranes}, this is the same as a natural transformation in $\Fun(\C^{op}, \iCat)$

\begin{equation}
\label{natural3}
\Tw(\Env(\M))^{\otimes}\to ((\T/(-))^{op})
\end{equation}

\noindent objectwise satisfying the conditions of a weak Cartesian structure and the conditions of the Cor. \ref{label2}. It will be more useful now to see them as functors $\T^{op}\to \iCat$ via Kan extension. In this case to give (\ref{natural3}) is equivalent to construct a map between their associated cartesian fibrations over $\T$

\begin{equation}
\label{diagram-fibrations1}
\xymatrix{
\intcart \Tw(\Env(\M))^{\otimes}\ar[dr] \ar[rr]&& \intcart ((\T/(-))^{op})\ar[dl]\\
&\T&
}
\end{equation}

\noindent preserving cartesian edges. Here the symbol $\intcart$ denotes the unstraightening construction of \cite[Chapter 3]{lurie-htt}. We now remark that the cartesian fibration $\intcart ((\T/(-))^{op})\to \T$ can be described by applying the discussion in the remark \ref{remark-cartesian-cocartesian} to the cocartesian fibration $X:=\intcocart ((\T/(-))^{op})\to S^{op}:=\T^{op}$. But this we can easily see, verifies canonically $\intcocart ((\T/(-))^{op})\simeq (\intcart \T/(-))^{op}$ where now $\intcart \T/(-)\to \T$ is the cartesian fibration classifying the categorical sheaf $X\mapsto \T/X$. For this one we have an explicit description, namely, it is given by the evaluation map $ev_1: \Fun(\Delta[1], \T)\to \T$. See \cite[5.2.2.5]{lurie-htt}. Therefore, and using the notations of the Remark \ref{remark-cartesian-cocartesian} we have  $\intcart ((\T/(-))^{op})=Y_0^{op}$ so that the data of a map (\ref{diagram-fibrations1}) is uniquely determined by the data of a $\infty$-functor

\begin{equation}
\label{diagram-fibrations2}
(\intcart \Tw(\Env(\M))^{\otimes})^{op}\times_{S^{op}}X^{op}\to \Spaces
\end{equation}

\noindent which, unwinding the notations, can be written as

\begin{equation}
\label{diagram-fibrations3}
\intcart \Tw(\Env(\M))^{\otimes}\times_{\T}\Fun(\Delta[1], \T)\to \Spaces^{op}
\end{equation}

\noindent or, in other words, as the data of a (fiberwise over $\T$) left representable right fibration

\begin{equation}
\label{diagram-fibrations4}
\xymatrix{
\mathrm{B}(\T, \M)\ar[d]\\
\intcart \Tw(\Env(\M))^{\otimes}\times_{\T}\Fun(\Delta[1], \T)
}
\end{equation}
 
We will now explain how to construct the correct right fibration. First we let $\mathrm{B}\M$ denote the image of $\Mmonoidal$ through the composition

\begin{equation}
\label{diagram-fibrations5}
\xymatrix{
\mathrm{Op}_{\infty}(\T)\subseteq \Fun(\T^{op}, \iCat)\ar[rr]^{\mathrm{B}\circ-}&& \Fun(\T^{op}, \Fun^{rf}(\Delta[1], \Spaces))\ar[rr]^{ev_0}&& \Fun(\T^{op},\iCat)
}
\end{equation}

\noindent where $\mathrm{B}$ is the functor constructed in section \ref{functoriality}. By construction, $\mathrm{B}\M$ comes equipped with a natural transformation 

\begin{equation}
\label{diagram-fibrations6}
\mathrm{B}\M\to  \Tw(\Env(\M))^{\otimes}
\end{equation}
 
\noindent and the transition maps preserve cartesian edges. In this case it is an easy exercise to see that the map induced between their associated cartesian fibrations

\begin{equation}
\label{diagram-fibrations7}
\intcart \mathrm{B}\M\to \intcart \Tw(\Env(\M))^{\otimes}
\end{equation}

\noindent is again a right fibration.\\

Let now $\pi:X\to S$ be a generic cartesian fibration between $(\infty,1)$-categories and let $\Fun(\Delta[1], X)^{\mathrm{Cart}}$ denote the full subcategory of  $\Fun(\Delta[1], X)$ spanned by the $\pi$-cartesian edges. Then it is also an easy exercise to check that the natural map

\begin{equation}
\label{diagram-fibrations8}
\Fun(\Delta[1], X)^{\mathrm{Cart}}\to X\times_S \Fun(\Delta[1], S)
\end{equation}

\noindent sending $x\to y \mapsto (y, \pi(x)\to \pi(y))$ is an equivalence of $(\infty,1)$-categories. Indeed, the fact that it is fully faithful follow from the definition of $\pi$-cartesian edges (using the characterization of $\pi$-cartesian edges via the mapping spaces of \cite[2.4.1.10 (2)]{lurie-htt}) and the fact that it is essentially surjective follows from the definition of cartesian fibration. In this case it admits a section $s$.\\

Back to our situation we apply this discussion to the cartesian fibration $X=\intcart \Tw(\Env(\M))^{\otimes}\to S=\T$ and composing the section $s$ with the evaluation at $0$ we obtain a map

\begin{equation}
\label{diagram-fibrations9}
\intcart \Tw(\Env(\M))^{\otimes}\times_\T \Fun(\Delta[1], \T)\to \Fun(\Delta[1], \intcart \Tw(\Env(\M))^{\otimes})^{\mathrm{Cart}}\to  \intcart \Tw(\Env(\M))^{\otimes}
\end{equation}

Finally, we define the right fibration $\mathrm{B}(\T, \M)$ in (\ref{diagram-fibrations4}) to be the pullback

\begin{equation}
\label{diagram-fibrations10}
\xymatrix{
\mathrm{B}(\T, \M)\ar[d]\ar[r]& \intcart \mathrm{B}\M\ar[d]\\
\intcart \Tw(\Env(\M))^{\otimes}\times_\T \Fun(\Delta[1], \T)\ar[r]& \intcart \Tw(\Env(\M))^{\otimes}
}
\end{equation}

It is clear from the construction that the fiber over an object $(\sigma \text{ over } Z, u:Y\to Z)$ is the space of extensions $\Ext(u^*(\sigma))$ in $\Mmonoidal(Y)$. In the case when $\sigma$ consists of a single active map $\nfin\to \onefin$ in $\Mmonoidal(Z)$ classified by a map $Z\to \M_n$ in $\T$, then we have canonical identifications

\begin{equation}
\Ext(u^*(\sigma))\simeq \Map_Y(Y, C_\sigma\times_Z Y)\simeq \Map_Z(Y, C_\sigma)
\end{equation}

\noindent where $C_\sigma$ is defined as in (\ref{curveprototype}). More generally, if $\sigma$ classifies a list of active maps $\sigma_i:\langle n_i \rangle \to\onefin$ in $\Mmonoidal(Z)$ corresponding to maps $\sigma_i: Z\to \M_{n_i}$, then by the defining properties of extensions and because $u^*$ is a map of operads, we have $\Ext(u^*(\sigma))\simeq \coprod_i \Ext(u^*(\sigma_i))$ which we can write as 

\begin{equation}
\label{equation-representable}
\Map_{/Z}(Y, C_\sigma)\simeq \coprod_{i=1}^n \Map_{/Z}(Y, C_{\sigma_i})
\end{equation}

\noindent where we set 

\begin{equation}
\label{diagram-fibrations11}
C_\sigma:=\coprod_i C_{\sigma_i}
\end{equation}

\noindent in $\T/Z$ (as we can always assume $Y$ to be affine and therefore, absolutely compact).\\

The formula (\ref{equation-representable}) gives us the representability condition specified in the Remark \ref{remark-cartesian-cocartesian}.  

To conclude let us remark that the map (\ref{diagram-fibrations1}) thus obtained, preserves cartesian edges. Indeed, if $\sigma\to \sigma'$ is a cartesian edge in $\intcart \Tw(\Env(\M))^{\otimes}$ over a map $f:Z\to Z'$, by definition, this means that $\sigma\simeq f^*(\sigma')$. By construction we then have $C_{\sigma}\simeq C_{\sigma'}\times_{Z'} Z$, which is exactly what characterizes cartesian edges in $\intcart ((\T/(-))^{op})$. \\

It remains to show that the map (\ref{natural3}) produced by this construction indeed satisfies 1) the conditions of weak cartesian structure and 2) the conditions of the Corollary \ref{label2} objectwise. But 1) follows from the formulas (\ref{diagram-fibrations11}) and 2) from the fact that as $\Mmonoidal$ is coherent the compositions are classified by pushouts as in the formula (\ref{coherencetoposmonochromatic2}). 
 
\hfill $\qed$

\begin{remark}
To conclude this section let us mention that as in the Remark \ref{contractibleO1}, and thanks to the Yoneda's lemma, to check that a monochromatic unital $\infty$-operad $\Mmonoidal$ in $\T$ having a single color with $\Mmonoidal(Z)(c_Z, c_Z)\simeq * $, is coherent, it is enough to check that for every $n\geq 2$ and $m\geq 2$, the composition diagram in $\T$ 

$$
\label{configurationtypediagramstacks}
\xymatrix{
\M_n\times \M_{m+1}\coprod\limits_{\M_2\times \M_n\times \M_n} \M_{n+1}\times \M_n\ar[d]\ar[r]& \M_{n+m}\ar[d]\\
\M_n\times \M_m\ar[r]& \M_{n+m-1}
}
$$
  
\noindent is a pullback diagram.
\end{remark}

\subsubsection{From Co-Spans to Spans}
\label{cospanstospans2}

In this section we explore the content of the Remark \ref{cospanstospans} in the case when we fix an object $E\in \T$ and consider branes mapping to $E$. As in the Remark, fixing $E$ we have a natural strongly monoidal map of $\infty$-operads in $\T$

$$
\rhom_{(-)}(-, E\times (-)): (\T/(-))^{\mathrm{co-corr}, \otimes_{\coprod}}\to  (\T/(-))^{\mathrm{corr}, \otimes_{\times}}
$$

\noindent mapping the induced coproduct structure in co-spans to the product structure in spans. \footnote{This is indeed a strongly monoidal map of $\infty$-operads because $\T$ is an $\infty$-topos}. We can now compose with the brane action of the Prop. \ref{propbranestopos} to produce a map of $\infty$-operads

\begin{equation}
\label{branetopostarget}
\Mmonoidal\to (\T/(-))^{\mathrm{co-corr}, \otimes_{\coprod}}\to  (\T/(-))^{\mathrm{corr}, \otimes_{\times}}
\end{equation}

Intuitively, this map is defined by the formula sending an operation $\sigma:Z\to \M_n$ to
$$
\xymatrix{
 \ar[dr] \rhom_{/Z}(\coprod_n Z\times \M_2, E\times Z)&\ar[l]\rhom_{/Z}(C_\sigma, E\times Z)\ar[r]\ar[d]&\ar[dl]\rhom_{/Z}( Z\times \M_2, E\times Z)\\
&Z&
}
$$

\noindent where $C_\sigma$ is the pullback

$$
\xymatrix{
C_\sigma:=Z\times_{\M_n}\M_{n+1}\ar[r]\ar[d]& \M_{n+1}\ar[d]\\
Z\ar[r]^{\sigma}& \M_n
}
$$

This correspondence is of course equivalent to

\begin{equation}
\label{braneactiontoatarget2}
\xymatrix{
 \prod_n \ar[dr] E^{\M_2}\times Z&\ar[l]\rhom_{/Z}(C_\sigma, E\times Z)\ar[r]\ar[d]&\ar[dl]E^{\M_2}\times Z\\
&Z&
}
\end{equation}
\noindent where $E^{\M_2}:=\rhom_\T(\M_2, E)$.

In this case, and using the adjunction $(-\times Z):\T\to \T/Z$, (\ref{braneactiontoatarget2}) is equivalent to the data of a correspondence in $\T$

\begin{equation}
\label{braneactiontoatarget}
\xymatrix{
&\ar[ld]\rhom_{/Z}(C_\sigma, E\times Z)\ar[dr]& \\\
( \prod_n   E^{\M_2})\times Z&&E^{\M_2}
}
\end{equation}

Again, by Yoneda, we have a universal diagram when $Z=\M_n$ and $\sigma$ is the identity map:

\begin{equation}
\label{braneactiontoatarget}
\xymatrix{
&\ar[ld]\rhom_{/\M_n}(\M_{n+1}, E\times \M_n)\ar[dr]& \\\
( \prod_n   E^{\M_2})\times \M_n&&E^{\M_2}
}
\end{equation}

\noindent so that for any $\sigma:X\to \M_n$, the correspondence assigned to $\sigma$ is the pullback of this universal one, along $\sigma$.

\begin{remark}
\label{remark-descriptionextensionstopos2}
Using the discussion in the Remark \ref{remark-cartesian-cocartesian} and the arguments in the beginning of the proof of the Prop. \ref{propbranestopos} the composition (\ref{branetopostarget}) is determined by the data of a (fiberwise over $\T^{op}$ -  left representable) left fibration

\begin{equation}
\label{remark-fibrationstarget1}
\xymatrix{
\mathrm{B}(\T, \M, E)\ar[d]\\
\intcocart \Tw(\Env(\M))^{\otimes}\times_{\T^op} \Fun(\Delta[1], \T)^{op}
}
\end{equation} 

\noindent where $\intcocart \Tw(\Env(\M))^{\otimes}\to \T^{op}$ is the cocartesian fibration classifying the categorical presheaf  $\Tw(\Env(\M))^{\otimes}$ and $\Fun(\Delta[1], \T)^{op}\to \T^{op}$ is the opposite of $ev_1$. Indeed, to define a natural transformation $\Tw(\Env(\M))^{\otimes}\to \T/(-)$ is equivalent to have a map between their associated cocartesian fibrations

\begin{equation}
\label{remark-fibrationstarget2}
\xymatrix{
\intcocart \Tw(\Env(\M))^{\otimes}\ar[rr]\ar[dr]&&\ar[dl] \intcocart \T/(-)\\
&\T^{op}&
}
\end{equation} 

\noindent that preserves cocartesian edges. Applying the discussion in Remark \ref{remark-cartesian-cocartesian} to the cocartesian fibration $X:= (\intcart \T/(-))^{op}\to \T^{op}$, such a map is equivalent to a $\infty$-functor

\begin{equation}
\label{remark-fibrationstarget3}
\intcocart \Tw(\Env(\M))^{\otimes}\times_{\T^{op}} \Fun(\Delta[1], \T)^{op}\to \Spaces
\end{equation} 

\noindent or, in other words, to a left fibration (\ref{remark-fibrationstarget1}). Informally speaking, the fiber over $(\sigma \text{ over } Z, u:Y\to Z)$ is the mapping space $\Map_Z(Y,  \rhom_Z(C_{\sigma}, E\times Z))$

\end{remark}

\subsection{Brane actions for graded $\infty$-operads}

In this section we discuss the notion of graded $\infty$-operads and explain how to construct graded brane actions. We will first deal with graded $\infty$-operads in spaces and at the end of the section we explain how to extend these results to graded $\infty$-operads in a topos.

\subsubsection{Graded $\infty$-operads and graded brane actions}
\label{gradedoperads}

Intuitively, a graded  $\infty$-operad is an $\infty$-operad $\Opmonoidal\to \Nerve(\Fin)$ such that for every $n\geq 0$, the space of operations $\Map_{\Opmonoidal_{act}}((X_1,..., X_n), Y)$ has a natural decomposition 

$$\coprod_{\beta\in B} \Map^\beta_{\Opmonoidal_{act}}((X_1,..., X_n), Y)$$ 

\noindent where $B$ is a monoid in sets. In other words, every operation is indexed by some $\beta\in B$. Moreover, if $\sigma: X\to Y$ is an active map of degree $\beta$ and for $1\leq i\leq n$, $\sigma_i:Z_i\to X_i$ are active maps of degree $\beta_i$, then the composition $\oplus_i Z_i\to Y$ is of degree $\beta+\sum_i \beta_i$. In order to formalize this idea, given $B$, we will construct an $\infty$-operad in spaces $\Nerve(\Fin^B)\to \Nerve(\Fin)$ which, essentialy, adds gradings to the morphisms in $\Nerve(\Fin)$. Then, we define a \emph{graded $\infty$-operad} to be an $\infty$-operad $\Opmonoidal$ equipped with a map of  $\infty$-operads $\Opmonoidal\to \Nerve(\Fin^B)$.

\begin{construction}
\label{Bsets}
Let $B$ be a monoid in sets with indecomposable zero\footnote{Recall that indecomposable zero means that if $\beta, \beta'\in B$ are such that $\beta+\beta'=0$ then both $\beta$ and $\beta'$ are zero.}. We define a category $\Fin^B$ as follows:

\begin{enumerate}
\item its objects are the objects of $\Fin$.
\item a morphism $\nfin\to \mfin$ is a pair $(f,\beta)$ where $f$ is a map in $\Fin$ from $\nfin \to \mfin$ and $\beta$ is a function $\beta: \mfin^{\circ}\to B$ 

\item the composition is dictated by the following rule:

\begin{itemize}
\item given $(f,\beta):\nfin\to \mfin$ and $(g, \lambda):\mfin\to \langle k \rangle$, then $\lambda\circ \beta: \langle k \rangle^{\circ}\to B$ is defined by the formula

\[
    (\lambda\circ \beta)(i):=
\begin{cases}
    \lambda(i)  ,& \text{if }  g^{-1}(\{i\})= \emptyset \\
    \lambda(i) + \sum_{j\in g^{-1}(\{i\})} \beta(j),              & \text{otherwise}
\end{cases}
\]

\end{itemize}

\end{enumerate}

It is clear from this definition that the composition law is well-defined. It is also clear that $\Fin^B$ has a forgetful functor $\Fin^B\to \Fin$ which simply forgets the grading functions $\beta$.

\end{construction}

\begin{prop}
\label{finalgradedoperad}
Let $B$ be a monoid with indecomposable zero. Then the map $\Nerve(\Fin^B)\to \Nerve(\Fin)$ is an $\infty$-operad.
\begin{proof}
We verify the three conditions of \cite[2.1.1.10]{lurie-ha}.

\begin{enumerate}
\item Every inert morphism $f:\nfin\to \mfin \in \Nerve(\Fin)$ has a coCartesian lifting. Indeed, we can lift $f$ by choosing the grading given by the zero function $(f,0)$. The grading has to be zero because of the fact $B$ has indecomposable zero. It is easy to see that this is a coCartesing lifting:
given $(u,\beta): \nfin\to \langle k \rangle$ and a commutative diagram in $\Fin$

$$
\xymatrix{
\nfin \ar[dr]_u\ar[r]^f & \mfin \ar[d]^g\\
&\langle k \rangle
}
$$

\noindent we can show that there exists a unique dotted arrow in $\Nerve(\Fin^B)$

$$
\xymatrix{
\nfin \ar[dr]_{(u, \beta)}\ar[r]^{(f,0)} & \mfin \ar@{-->}[d]^{(g, \lambda)}\\
&\langle k \rangle
}
$$

\noindent that makes the diagram commute. Choose $\lambda=\beta$.

\item Fixing $f:\nfin\to \mfin$, we have

$$
\Map^f_{\Nerve(\Fin^B)}( \nfin, \mfin)\simeq \prod_{i\in \mfin^\circ}\Map^{\rho^i\circ f}_{\Nerve(\Fin^B)}( \nfin, \onefin)
$$

\noindent where $\rho^i:\mfin\to \onefin$ is the inert map that sends $i\to 1$ and all the others to $0$. In this case we have

$$
\Map^f_{\Nerve(\Fin^B)}( \nfin, \mfin)\simeq \{f\}\times \mathrm{Hom}_{Sets}(\mfin^\circ, B)\simeq 
$$

$$
\{f\}\times  \mathrm{Hom}_{Sets}(\coprod_{i\in \mfin^\circ} \onefin^\circ, B)\simeq  \prod_{i\in \mfin^\circ} \{f\}\times \mathrm{Hom}_{Sets}( \onefin^\circ, B)\simeq
$$

$$
\simeq \prod_{i\in \mfin^\circ} B
$$

\noindent which is equivalent to 

$$
 \prod_{i\in \mfin^\circ} \{\rho^i\circ f\}\times  \mathrm{Hom}_{Sets}( \onefin^\circ, B)\simeq \prod_{i\in \mfin^\circ}\Map^{\rho^i\circ f}_{\Nerve(\Fin^B)}( \nfin, \onefin)
$$

\item $\Nerve(\Fin^B)_{\nfin}\simeq \Nerve(\Fin^B)^n_{\onefin}$. This is obvious from the definition.
\end{enumerate}

\end{proof}
\end{prop}

\begin{remark}
\label{remarkgraded1}
It follows from the construction and from the assumption that $B$ has indecomposable zero that a map $(f, \beta)$ in $\Fin^B$ is an isomorphism if and only if $\beta=0$ and $\sigma$ is an equivalence. Moreover, it is also clear that a map in $\Fin^B$ is inert \cite[Def. 2.1.2.3]{lurie-ha} if and only if it is inert in $\Fin$ and its grading function is zero. The same for semi-inert morphisms.
\end{remark}

\begin{defn}
\label{definition-gradedoperad}
Let $B$ be a monoid in sets with indecomposable zero. A \emph{$B$-graded $\infty$-operad} is a map of $\infty$-operads $p:\Opmonoidal\to \Nerve(\Fin^B)$.
\end{defn}

\begin{remark}
As the inert morphisms in $\Nerve(\Fin^B)$ are exactly the inert morphims of $\Nerve(\Fin)$ endowed with a zero grading, thanks to the Remark \ref{remarkgraded1} any map of $\infty$-operads $\Opmonoidal\to \Nerve(\Fin^B)$ is a fibration of $\infty$-operads \footnote{See \cite[2.1.2.10]{lurie-ha} for the definition of fibration of $\infty$-operads.}. In particular, and thanks to \cite[2.1.2.22]{lurie-ha}, our definition \ref{definition-gradedoperad} is equivalent to the data of an $(\infty,1)$-category $\Opmonoidal$ together with a map to $\Nerve(\Fin^B)$ satisfying the obvious graded analogues of Lurie's definition of $\infty$-operads \cite[2.2.1.10]{lurie-ha}.
\end{remark}

There exists a combinatorial simplicial model structure in the category of marked simplicial sets over $\Nerve(\Fin^B)$ such that the fibrant objects are exactly $B$-graded $\infty$-operads. \footnote{For instance, one can use the theory of categorical patterns of \cite[Appendix B]{lurie-ha} as in the proof of \cite[2.1.4.6]{lurie-ha}.} We let $\mathrm{Op}_{\infty}^{B-gr}$ denote its underlying $(\infty,1)$-category. It is clear from the definition that we have $\mathrm{Op}_{\infty}^{B-gr}\simeq \mathrm{Op}_{\infty}/\Nerve(\Fin^B)$. Notice also that we have a functor 

\begin{equation}
\label{equation-forgetgradingadjunction}
\mathrm{Op}^{B-gr}_{\infty}\to \mathrm{Op}_{\infty}
\end{equation}

\noindent that forgets the graded structure and admits a right adjoint, namely,  the pullback along the map $\Nerve(\Fin^B)\to \Nerve(\Fin)$. This functor admits a section that sees an operad as a graded operad with zero gradings.

\begin{remark}
\label{remark-gradeddendroidal}
There is also a dendroidal approach to graded $\infty$-operads \todo{ALSO HERE, the comparison of Mordjeik is wrong.}. Indeed, we can define a category $\Omega_B$ of trees where each vertice $v$ comes with the extra data of an element $\beta_v\in B$ and the morphisms of contraction sum the $\beta$'s. A graded dendroidal segal object in spaces is then an $\infty$-functor $\Nerve(\Omega_B)^{op}\to \Spaces$ satisfying the analogue of the Segal conditions for dendroidal spaces. One can show that the $\inftyone$-category
$\Fun^{\mathrm{Segal}}(\Nerve(\Omega_B)^{op}, \Spaces)$ is equivalent to  $\mathrm{Op}_{\infty}^{B-gr}$. Indeed, consider the equivalence $\mathrm{Op}_{\infty}\simeq \Fun^{\mathrm{Segal}}(\Nerve(\Omega)^{op}, \Spaces)$ and remark that the last is also equivalent to the full subcategory of $\iCat/\Nerve(\Omega)$ spanned by those functors that are right fibrations and satisfy the Segal condition after the Grothendieck construction. In this case the $\infty$-operad $\Fin^B$ produces a right fibration over $\Nerve(\Omega)$ whose fibers are discrete spaces. Its total category can be identified with $\Omega_B$ and the map $\Omega_B\to \Omega$ is just the functor that forgets the gradings of the vertices. Following this discussion we have a chain of equivalences

\begin{equation}
\mathrm{Op}_{\infty}^{B-gr}\simeq \mathrm{Op}_{\infty}/\Nerve(\Fin^B)\simeq (\iCat^{rf, \mathrm{Segal}}/\Nerve(\Omega))_{/\Nerve(\Omega^B)}\simeq \Fun^{\mathrm{Segal}}(\Nerve(\Omega_B)^{op}, \Spaces)
\end{equation}\\

\end{remark}

We now discuss the notion of coherent $\infty$-operad \cite{lurie-ha} in the graded setting. This time we want the extensions to fix the gradings. Let $\Opmonoidal\to \Nerve(\Fin^B)$ be a graded $\infty$-operad. Given an active morphism $\sigma:X\to Y \in \Opmonoidal$ over $(f:\nfin\to \mfin, \beta:\mfin^\circ\to B)$ we want to study the space $\Ext^\beta(\sigma)$ of all active morphisms $\tilde{\sigma}:X\oplus X_0\to Y$ such that the composition with the semi-inert map $X\to X\oplus X_0$ is $\sigma$ and $\tilde{\sigma}$ is also of degree $\beta$. As discussed in the Remark \ref{remarkgraded1}, every semi-inert map in $\Fin^B$ must have zero grading. This condition forces the grading function of any extension $\tilde{\sigma}$ to be necessarily equal to the one of $\sigma$ so that the definition of the space of extensions $\Ext^\beta(\sigma)$ is just the same $\Ext(\sigma)$ as in the non-graded case  (\cite[3.3.1.4, 3.3.1.9]{lurie-ha}). It is also obvious from this that the coherence of a graded $\infty$-operad is determined by the coherence of its underlying $\infty$-operad.\\

Let $\Opmonoidal$ be a  $B$-graded coherent $\infty$-operad. Let us now deal with the construction of brane actions for $\Opmonoidal$, compatible with the gradings. For that purpose we need to construct a map of $B$-graded $\infty$-operads

\begin{equation}
\label{gradedbrane1}
\xymatrix{
\Opmonoidal\ar[dr]^p\ar[rr]&& \Spaces^{\mathrm{co-corr}, \otimes_{\coprod}}\times_{\Nerve(\Fin)}\Nerve(\Fin^B)\ar[dl]\\
&\Nerve(\Fin^B)&
}
\end{equation}
 
\noindent but using the adjunction (\ref{equation-forgetgradingadjunction}) this is the same as a map of $\infty$-operads $\Opmonoidal\to \Spaces^{\mathrm{co-corr}, \otimes_{coprod}}$ where we forget the grading of $\Opmonoidal$. In other words, a diagram

$$
\xymatrix{
\Opmonoidal\ar[r]^b\ar[d]^p& \Spaces^{\mathrm{co-corr}, \otimes_{\coprod}}\ar[d]\\
\Nerve(\Fin^B)\ar[r]& \Nerve(\Fin)
}
$$

\noindent sending inert morphism to inert morphisms. But as inert morphisms in $\Fin^B$ have always zero grading, being inert in the graded sense is equivalent to being inert when we forget the gradings.\\

Following this discussion, the construction of brane actions can be performed exactly as in Thm \ref{thmbranes}:

\begin{cor}
Let $\Opmonoidal$ be a $B$-graded $\infty$-operad such that its underlying $\infty$-operad is unital, coherent has a unique color $c$ with $\mathcal{O}(1)_0=\ast$ and $\mathcal{O}(1)_\beta=\emptyset$ for $\beta\neq 0$. Then there exists a map of $B$-graded $\infty$-operads (\ref{gradedbrane1}) encoding the brane action for $\Opmonoidal$. 
\end{cor}

The functoriality arguments of section \ref{functoriality} carry over to the graded context providing an $\infty$-functor $\mathrm{B}:\mathrm{Op}_{\infty}^{B-gr, *}\to \mathrm{Op}_{\infty}^* \to \Fun^{rf}(\Delta[1], \iCat)$ sending 

$$\Opmonoidal\mapsto (\mathrm{B}\Op\to \Tw(\Env(\Op)^{\otimes})$$

\subsection{Graded $\infty$-operads in a $\infty$-topos and Brane actions}
\label{gradedbraneactiontopos}
The arguments of Section \ref{branesfortopos} carry over to the context of $B$-graded $\infty$-operads in an $\infty$-topos $\T=\Sh(\C)$, which as in Section \ref{branesfortopos}, we can define as objects in $\Sh(\C, \mathrm{Op}_{\infty}^{B-gr})$. One can also combine the arguments of Section \ref{gradedoperads} and \ref{branesfortopos} to produce brane actions for a graded-coherent $\infty$-operad $\Mmonoidal\in \Sh(\C, Op_{\infty}^{B-gr})$.\\

As in \ref{branesfortopos}, and as $\Mmonoidal$ is monochromatic, we can think of $\Mmonoidal$ as a collection of objects in $\T$, $\{\M_{n,\beta}\}_{n\geq 0, \beta\in B}$ with the following universal property: for each $Z\in \C$ we have canonical equivalences

$$
\M(Z)(n, \beta):=\Map_{\Mmonoidal(Z)_{act}}^{\beta}((c_Z,..., c_Z), c_Z)\simeq \Map_\T(Z,\M_{n,\beta})$$

\noindent and again, by Yoneda, the composition laws of $\Mmonoidal$ are represented by maps in $\T$ 

$$
\M_{n,\beta}\times \M_{m, \beta'}\to \M_{n+m-1, \beta+\beta'}
$$
\noindent satisfying the expected coherences up to homotopy. \\

To construct the associated brane action we can proceed as in the Section \ref{branesfortopos} and obtain a map of $B$-graded $\infty$-operads in $\T$

$$
\Mmonoidal\to (\T/(-))^{\mathrm{co-corr}, \times_{\coprod}}\times_{\Nerve(\Fin)} \Nerve(\Fin^B)
$$

\noindent which, as in the proof of the Prop. \ref{propbranestopos}, sends an operation $\sigma: Z\to \M_{n,\beta}$ to the object 

\begin{equation}
\label{curveprototypegraded}
C_{\sigma}:= Z\times_{\M_{n, \beta}}\M_{n+1, \beta}
\end{equation}

\noindent in $\T$. Moreover, the same arguments of Sections \ref{cospanstospans2} tell us that the brane action of a graded $\infty$-operad $\Mmonoidal$ with respect to a fixed target $E$,

$$
\Mmonoidal\to (\T/(-))^{\mathrm{co-corr}, \otimes_{\coprod}}\times_{\Nerve(\Fin)} \Nerve(\Fin^B) \to (\T/(-))^{\mathrm{corr}, \otimes_{\times}}\times_{\Nerve(\Fin)} \Nerve(\Fin^B)
$$

\noindent informally given by the assignment\footnote{See the Remark \ref{remark-descriptionextensionstopos2}.}

$$
(\sigma: Z\to \M_{n,\beta})\mapsto \rhom_{/Z}(C_{\sigma}, E\times Z)
$$

 \noindent is induced by universal correspondences in $\T$ of the form

\begin{equation}
\label{gradeduniversalcorrespondence}
\xymatrix{
&\rhom_{/\M_{n,\beta}}(\mathrm{M}_{n+1,\beta}, E\times \M_{n,\beta}  )\ar[dr]\ar[dl]&\\
\prod_n E^{\M_{2,0}}\times \M_{n,\beta}&& E^{\M_{2,0}}
}
\end{equation}

\section{Stable Actions}
\label{section-stableactionsgw}
From now on, $\T$ will denote the $\infty$-topos of derived stacks $\dst$ over a field of characteristic zero $k$, with respect to the \'etale topology. We set $\C:=\dst^{\mathrm{aff}}$ as a notation for the $(\infty,1)$-category of derived affine schemes.

\subsection{The operad $\MKO$ of Costello and its brane action}

\subsubsection{The Stacks $\MnK$ of Costello}
\label{stacksofcostello}

Here we follow \cite{Costello-HighergenusGW2006}. Let $X$ be a smooth projective variety and let $\mathrm{B}:=\NE(X)\subset H_{ 2}(X,\mathbb{Z})$ be the Mori cone
of $X$, generated as a monoid by the numerical classes of irreducible curves in $X$. The Mori cone satisfies the following properties: 

\begin{enumerate}
\item $\NE(X)$ has indecomposable zero : $\beta+\beta'=0$ implies $\beta=\beta'=0$.
\item $\NE(X)$ has finite decomposition : for every $\beta\in \NE(X)$, the set
  $\{(\beta_{1},\beta_{2})\in \NE(X)\times\NE(X)\mid \beta_{1}+\beta_{2}=\beta\}$ is finite.
\end{enumerate}

The Mori cone will play the role of the semi-group in \cite{Costello-HighergenusGW2006} and will be our grading monoid $B$.\\

We will use the  definition of the pre-stack in 1-groupoids $\MnK$ of \cite{Costello-HighergenusGW2006}. It classifies (possible unstable)  connected nodal genus $0$ curves $C$ with $n$-marked smooth points and an index $\beta_i\in B$ attached to each irreducible component $C_i$. Moreover, we impose the following stability conditions:

\begin{itemize}
\item $\beta=\sum_{i} \beta_{i}$.
\item if $\beta_{i}=0$ then $C_{i}$ has at least three special points, meaning, marked or nodal points. 
\end{itemize}

The stability conditions of \cite[Section 2]{Costello-HighergenusGW2006} force $\mathfrak{M}_{0,1,\beta=0}=\emptyset$ and
$\mathfrak{M}_{0,2,\beta=0}=\emptyset$.

\begin{remark}
 Notice that if $\beta=0$ then $\mathfrak{M}_{0,n,\beta}$ is the usal Deligne-Mumford stack of stable curves $\Mnstable$. In particular,  $\mathfrak{M}_{0,3,\beta=0}= *$, classifying $\mathbb{P}^1$ with three marked points.
\end{remark}

The following proposition summarizes the main features of the pre-stacks $\MnK$:

\begin{prop}
\label{prop:Costello}
The following holds:
\begin{enumerate}
 \item For all $n\geq 0$ and $\beta\in B$, the pre-stack $\MnK$ is a smooth and proper algebraic stack in 1-groupoids, locally of finite type and non-separated.
 \item Forgetting the last marked point and stabilising the curve gives a proper separated morphism
    $\mathfrak{M}_{0,n+1,\beta}\to\mathfrak{M}_{0,n,\beta}$ which is the universal curve when $n\geq 3$.
  \end{enumerate}
  \begin{proof}
  \cite[Pag. 2, Propositions 2.0.2 and 2.1.1]{Costello-HighergenusGW2006}
  \end{proof}
\end{prop}

\begin{remark}
\label{remarkuniversalcurve}
As the universal family is flat, the properties (1) and (2) combined mean that the object $C_{\sigma}$ associated to a map $\sigma:Z\to \MnK$ and given by the fiber product in derived stacks (\ref{gradeduniversalcorrespondence}) is the same as the fiber product computed in usual stacks and therefore, corresponds exactly the curve over $Z$ classified by $\sigma$ when $n\geq 3$. When $n=2$, $C_\sigma=Z$.
\end{remark}

\subsubsection{The collection $\{\mathfrak{M}_{0,n,\beta}\}_{n\in \mathbb{N},\beta\in B}$ as a graded $\infty$-operad in derived stacks}
\label{section-costelloisoperad}

We will now use the moduli spaces of Costello to construct a graded operad in stacks. More precisely, we take the collection $\{\mathfrak{M}_{0,n,\beta}\}_{\{n\geq 3,\beta\in B\}}$ together with  ad-hoc spaces of unary and nullary operations replacing the role, respectively, of $\mathfrak{M}_{0,2,\beta}$ and $\mathfrak{M}_{0,1,\beta}$. We have to do this in order to produce a graded operad that to which we can apply our brane action - and this requires the spaces of nullary and unary operations to be contractible. Therefore, we introduce the stacks

$$\mathfrak{M}^{\mathrm{fake}}_{0,2,0}=\mathfrak{M}^{\mathrm{fake}}_{0,1, 0}= *$$

\noindent which we think, respectively, as a $\mathbb{P}^1$ with two (resp. one) marked points, and considered only with the identity automorphism. Moreover, we also set

$$\mathfrak{M}^{\mathrm{fake}}_{0,2,\beta}=\mathfrak{M}^{\mathrm{fake}}_{0,1,\beta}= \emptyset$$

\noindent for $\beta\neq 0$. By imposing this we will have to modify the gluing operation of curves. See below. We now remark that the collection $\{\mathfrak{M}_{0,n,\beta}\}_{\{n\geq 3,\beta\in B\}}\cup \{\mathfrak{M}^{\mathrm{fake}}_{0,2,\beta}, \mathfrak{M}^{\mathrm{fake}}_{0,1,\beta}\}_{\beta}$ forms a $B$-graded (symmetric) $\infty$-operad in derived stacks.
We proceed as follows: recall that the 1-category of stacks in groupoids embeds fully faithfully in the $\inftyone$-category of derived stacks $\dst$\footnote{Compose with the Nerve functor from groupoids to simplicial sets and take the Kan extension along the inclusion $\mathrm{Aff}^{\mathrm{classic}}\subseteq \mathrm{dAff}$}. In general, this inclusion commutes with colimits but not with products. However, the compatibility with products holds when the stacks involved are smooth, as smoothness implies flatness, which means, the derived tensor product is isomorphic to the ordinary one. In our case smoothness follows from  the Prop. \ref{prop:Costello}. Therefore, it will be enough to show that the family  $\{\mathfrak{M}_{0,n,\beta}\}_{\{n\geq 3,\beta\in B\}}\cup \{\mathfrak{M}^{\mathrm{fake}}_{0,2,\beta}, \mathfrak{M}^{\mathrm{fake}}_{0,1,\beta}\}_{\beta}$ forms a $B$-graded operad in classical $1$-stacks. Intuitively, the composition operation corresponds to gluing curves along the marked points. The last point is thought as the output of the operation. For this we need to make a shift in our notations:  We set

\begin{equation}
\label{formula-shiftnotations}
\MnKO:=\mathfrak{M}_{0,n+1,\beta}
\end{equation}

\noindent if $n\geq 3$ and

\begin{equation}
\label{formula-shiftnotations455}
\mathfrak{M}(1,\beta):=\mathfrak{M}^{\mathrm{fake}}_{0,2,\beta} \text{ and } \mathfrak{M}(0,\beta):=\mathfrak{M}^{\mathrm{fake}}_{0,1,\beta}
\end{equation}

 \noindent and with this definition we have $\mathfrak{M}(0,\beta)\simeq \mathfrak{M}(1,0)\simeq *$.

\begin{prop}\label{prop:unitary,operad}
 The collection $\{\MnKO \}_{\{\beta \in B, n\in \mathbb{N} \}}$ forms a unital $B$-graded (symmetric) operad in classical 1-stacks. The unit is the unique element of $\mathfrak{M}^{fake}_{0,2,0}$ given by a $\mathbb{P}^1$ with two marked points. Moreover, as $\mathfrak{M}(2,0):=\mathfrak{M}_{0,3,0}$ is the moduli of stable curves with 3 marked points, it is contractible.
\end{prop}

\begin{proof}
The proof is the same as remarked in \cite{MR1601666}. Composition is given by gluing curves along marked points. To force the (fake) projective space with two marked points to be the identity one declares the maps $\mathfrak{M}_{0,n,\beta}\times \mathfrak{M}^{\mathrm{fake}}_{0,2,0}\to \mathfrak{M}_{0,n,\beta}$ to be the identity. The operad is unital as $\mathfrak{M}^{\mathrm{fake}}_{0,1,0}$ is contractible and  $\mathfrak{M}^{\mathrm{fake}}_{0,1,\beta}$ is empty for $\beta\neq 0$. The maps $\mathfrak{M}_{0,n+1,\beta}\to \mathfrak{M}_{0,n,\beta}$ ($n\geq 3$) corresponding to the composition with the nullary operation in $\mathfrak{M}^{\mathrm{fake}}_{0,1,0}$ are declared to be obtained by forgetting the last marked point and stabilizing. If $n=2$ and $\beta\neq 0$, there is no such map in this operad.
\end{proof}

\begin{remark}
The necessity of replacing $\mathfrak{M}_{0,2,0}$ by its fake version $\mathfrak{M}^{\mathrm{fake}}_{0,2,\beta}$ is related to our conditions for the existence of a brane action, as it requires the space of unary operations to be contractible. By using this fake version we will essentially lose the moduli spaces of stable maps with two marked points. As we shall see below this is not relevant in the structure of the quantum product but it is rather crucial to explain the lax structure of our Gromov-Witten action. As we shall see in the next section the lax associativity is controlled by these two-marked stable maps. 
\end{remark}

This $B$-graded operad in classical smooth $1$-stacks can be written as a graded dendroidal segal object (see the Remark \ref{remark-gradeddendroidal}) with values in stacks in 1-groupoids , sending a graded tree $T\in \Nerve(\Omega^{op}_B)$ to the stack $\prod_{v\in \mathrm{Vert}(T)}\mathfrak{M}_{0,{n(v)}, \beta_v}$ where $n(v)$ is the number of edges attached to the vertice $v$ and $\beta_v$ is the grading of the vertice $v$. This satifies the Segal conditions and as the inclusion of smooth stacks in derived stacks is monoidal for the cartesian product we find that the collection $\{\MnKO\}_{\{n,\beta\}}$ forms a monochromatic unital $B$-graded $\infty$-operad in derived stacks. We will denote it as $\MKO$.

\begin{remark}
\label{remark-compositioncurves}
  Given an object $Z\in \C$, the graded $\infty$-operad in spaces $\MKO(Z)\to \Nerve(\Fin^B)$ verifies

$$
\Map^\beta_{\MKO(Z)_{act}}((c_Z,..., c_Z), c_Z)\simeq \Map_{\T}(Z, \mathfrak{M}(n, \beta))\simeq \Map_{\T}(Z, \mathfrak{M}_{0,n+1, \beta})
$$

It is also important to remark that by definition of the operadic structure in the moduli spaces of Costello, the composition of two active morphisms in $\MKO(Z)$

$$
\xymatrix{
\nfin \ar[r]^{(f, \beta)}& \mfin \ar[r]^{(g, \lambda)}&\onefin
}
$$

\noindent corresponds to a gluing of curves over $Z$. More precisely, if $f$ classifies a family of curves over $Z$ $\{C_f^i\}_{i\in \mfin^{\circ}}$ with $C_f^i$ with grading $\beta_i$, and $g$ classifies a curve  $C_g$ with grading $\lambda$, then the composition $g\circ f$ classifies the curve of total grading $\lambda + \sum \beta_i$ obtained by gluing the last marked point of the $C_f^i$ to the $i$-marked point of $C_g$.\\ 

\end{remark}

Contrary to what the reader could expect at this point, the operad $\MKO$, althought it satisfies all the conditions in A) (pag. \pageref{Alph*}), it is not coherent. This occurs essentially because if $\sigma$ and $\tau$ are two composable operations classifiying two curves $C_{\sigma}$ and $\C_{\tau}$,  the gluing of these two curves along a marked point, which classifies the composition $C_{\tau\circ \sigma}$, is not equivalent to the pushout $C_{\sigma}\coprod_{\ast} \C_{\tau}$ in the $(\infty,1)$-category of derived stacks (as expressed in the Remark \ref{coherencetoposmonochromatic}). Indeed, the inclusion of schemes in derived stacks does not commute with pushouts in general, even along closed immersions. All we have is canonical map

$$
\theta:C_{\sigma}\coprod_{\ast} \C_{\tau}\to C_{\sigma}\coprod^{\mathrm{Sch}}_{\ast} \C_{\tau}\simeq C_{\tau\circ \sigma}
$$

Neverthless, part of the proof of the theorem \ref{thmbranes} still makes sense. Namely, we don't need coherence to have the natural transformation

\begin{equation}
\label{vistaalegre}
\Tw(\Env(\mathfrak{M}))^{\otimes}\to \T/(-)^{op}
\end{equation}

\begin{remark}
This map sends an operation $\sigma$ consisting of a single active map over $Z$ to $C_{\sigma}$ as defined in the formula (\ref{curveprototypegraded}). Moreover, it sends a twisted arrow $\tau\to \sigma$ to a map $C_\sigma\to C_\tau$. Thanks to the description of compositions in $\MKO(Z)$ as gluing of curves (see the Remark \ref{remark-compositioncurves}) we known that the data of a morphism $\theta:\tau\to \sigma$ in  $\Tw(\Env(\mathfrak{M}(Z)))^{\otimes}$ consists of a way to express  $C_\tau$ as obtained from $C_\sigma$ by attaching some components determined by $\theta$. For simplicity, consider the case where $\theta$ is encoded by a commutative diagram $\Psi$ \footnote{In fact we can always reduce to this case.}

\begin{equation}
\label{basictwist}
\xymatrix{
\nfin \ar[d]^\tau\ar[r]^u\ar@{}[dr]|{\Psi}& \mfin \ar[d]^\sigma\\
\onefin \ar@{=}[r]& \onefin
}
\end{equation}

The commutativity of this square in $\MKO(Z)$ means that in fact the curve $C_\tau$ classified by $\tau$ is equivalent (in this case isomorphic) to the curve obtained from the curve $C_\sigma$ classified by $\sigma$ by attaching the curves $\{C_u^i\}_{i\in \mfin^\circ}$ classified by $u$. This pushout (in schemes!) attaches the last marked point of $C_u^i$ to the $i$-marked point of $C_\sigma$. The commutativity of the diagram is given by the data of an equivalence $\Psi$ between $C_\sigma$ and the result of this gluing. In this case the canonical map $C_\sigma\to C_\tau$ can be naturally identified with the inclusion, composed with $\Psi$.\\
\end{remark}

The main problem with (\ref{vistaalegre}) is that it doesn't satisfy the conditions of the Cor. \ref{label2}. But in fact, to our purposes, this is not a real issue. In fact, although the map $\theta$ is not an equivalence in $\T$, if we take $E=X$ the smooth projective variety fixed at the beginning of this section, we know that $\rhom(-,X)$ will see $\theta$ as an equivalence. Indeed, this follows from the characterization of pushouts of schemes along closed immersions in terms of pullbacks for Quasi-coherent sheaves, via Tannakian duality. More precisely, we can use the results of \cite[Thm 1.4]{1507.01925} and \cite[Thm 7.1]{lurie-dagIX} to obtain the criterion of \cite{MR3572635}. More generally, following loc.cit and \cite[p. 4]{1405.1887} we could also take $X$ to be a perfect stack in the sense of \cite{MR2669705}.

In this case the composition with the natural transformation $\rhom_{(-)}(-, X\times-)$

$$
\Tw(\Env(\mathfrak{M}))^{\otimes}\to \T/(-)^{op}\to \T/(-)
$$

\noindent gives a map satisfying the conditions of the corollary \ref{label2} and therefore a map of $\infty$-operads

$$
\MKO\to (\T/(-))^{\mathrm{corr}, \otimes_{\times}}
$$

\begin{cor}
\label{CostelloactiononX}
Let $X$ be a smooth projective algebraic variety. Then $X$ is an $\MKO$-algebra. The algebra structure is encoded by the correspondences

\begin{equation}
\label{prestablebraneaction}
\xymatrix{
&\rhom_{/\mathfrak{M}_{0,n+1,\beta}}(\mathfrak{M}_{0,n+2,\beta}, X\times \mathfrak{M}_{0,n+1,\beta} )\ar[dr]\ar[dl]&\\
X^n\times \mathfrak{M}_{0,n+1,\beta}&&X
}
\end{equation}
\end{cor}

\subsection{The stable sub-action of $\MKO$}
\label{Section-stablesubaction}

We now want to consider a certain \emph{sub-action} of the one constructed in the previous section. Following the Prop. \ref{prop:Costello} and the Remark \ref{remarkuniversalcurve}, the stack $\mathfrak{M}_{0,n+1,\beta}$ is the universal curve over $\mathfrak{M}_{0,n,\beta}$. In this case the derived stack $\rhom_{/\MnK}(\mathfrak{M}_{0,n+1,\beta}, X\times \MnK )$ classifies pairs $(C,f)$ where $C$ is classified by $\MnK$ and $f$ is a map $f:C\to X$. Inside this stack there is an open sub-stack $\stablemaps$ (in fact, a connected component) that classifies stable maps to $X$ of total degree $f_*[C]=\beta$ - see \cite[Def. 2.7]{2011-Schur-Toen-Vezzosi}. Moreover, by \textit{loc.cit.} we know that this stack is a proper derived Deligne-Mumford stack which is quasi-smooth. The reason we are interested in the derived stack $\stablemaps$ is the fact that its structure sheaf is the origin of all the virtual phenomena in Gromov-Witten theory. We will come back to this later on. For now we are merely interested in producing a new brane action - a \emph{stable action} - where the universal correspondences (\ref{prestablebraneaction}) are replaced by (\ref{Costellocorrespondence}) where the arrows are obtained by composition the maps of (\ref{prestablebraneaction}) with the open immersion $\stablemaps\subseteq \rhom_{/\MnK}(\mathfrak{M}_{0,n+1,\beta}, X\times \MnK )$. Our goal in this section is to show that this restriction still carries all the coherences defining an action of $\MKO$. We start with the brane action to the fixed target $X$, encoded by the map of $B$-graded $\infty$-operads in $\T$

$$\MKO\to \T/(-)^{\mathrm{corr}}\times_{\Nerve(\Fin)}\Nerve(\Fin^B)$$

\noindent of the Corollary \ref{CostelloactiononX}. By adjunction, this is the same as a map of $\infty$-operads in $\T$, $\MKO\to (\T/(-))^{\mathrm{corr}, \otimes_\times}$ which, by repeating the arguments in the Remark \ref{remark-descriptionextensionstopos2} is given by a (fiberwise over $\T^{op}$ - left representable) left fibration

\begin{equation}
\label{formula-branetoXnonstable}
\xymatrix{
\mathrm{B}(\T,\mathfrak{M},X)\ar[d]^{\pi_X}\\
\intcocart \Tw(\Env(\mathfrak{M}))^{\otimes} \times_{\T^{op}} \Fun(\Delta[1], \T)^{op}
}
\end{equation}

\noindent whose fiber over $(\sigma=(\sigma_1,..., \sigma_n) \text{ over } Z, u:Y\to Z)$ can be described as the mapping space $\Map_{Z}(Y, \rhom_Z(C_{\sigma}, X\times Z))$ where now $C_{\sigma}$ is the coproduct $\coprod_i C_{\sigma_i}$ where each $C_{\sigma_i}$ is defined as in the formula (\ref{curveprototypegraded}). 

To construct the stable action we consider the full subcategory 

$$\mathrm{B}^{\mathrm{Stb}}(\T,\mathfrak{M},X)\subseteq \mathrm{B}(\T,\mathfrak{M},X)$$

\noindent whose fiber over an object $(\sigma=(\sigma_1,..., \sigma_n) \text{ over } Z, u:Y\to Z)$ is spanned by those maps

\begin{equation}
\xymatrix{
Y\ar[dr]\ar[rr]&&\rhom_{Z}(C_{\sigma}, X\times Z)\ar[dl]\\
&Z&
}
\end{equation}

\noindent such that for each $i$ the map $Y\to \rhom_Z(C_{\sigma_i}, X\times Z)$ factors through the open su-bstack (in fact, connected component) $\rhom_Z^{\mathrm{Stb}}(C_{\sigma_i}, X\times Z)$ classifying families of maps

$$
\xymatrix{
C_{\sigma_i}\ar[dr]\ar[rr]^{f_i}&& Z\times X\ar[dl]\\
&Z&
}
$$

\noindent such that for each geometric point $z$ of $Z$, the base-change $f_z$ satisfies $(f_z)_{\ast}([C_{\sigma_i, z}]=\beta_{i}$ in cohomology. Here $\beta_i$ is the degree associated to the active map $\sigma_i$. It follows from the definition of $\MnK$ that such $f_i$ are necessarily stable maps. In particular, when $Z=\MnK$ and $\sigma$ is the identity of the unique color, we see that the derived stack $\rhom_{\MnK}^{\mathrm{Stb}}( \mathfrak{M}_{0,n+1,\beta}, Z\times X)$ is exactly the derived enhancement of the stack of stable maps $\overline{\mathcal{M}}_{0,n}(X,\beta)$ of \cite[Def. 2.7]{2011-Schur-Toen-Vezzosi}.\\

The main result of this section is the following:

\begin{prop}
\label{corollary-stableaction}
The composition 

\begin{equation}
\label{compositionstable}
\mathrm{B}^{\mathrm{Stb}}(\T,\mathfrak{M}, X)\subseteq \mathrm{B}(\T,\mathfrak{M},X) \to \intcocart \Tw(\Env(\mathfrak{M}))^{\otimes} \times_{\T^{op}} \Fun(\Delta[1], \T)^{op}
\end{equation}

\noindent  is a (fiberwise in $\T^{op}$) left representable left fibration. Moreover, it defines a new map of $B$-graded $\infty$-operads in $\T$

\begin{equation}
\label{equation-stableaction}
\MKO\to \T/(-)^{\mathrm{corr}}\times_{\Nerve(\Fin)}\Nerve(\Fin^B)
\end{equation}

\noindent explicitely given by the correspondences in the formula (\ref{Costellocorrespondence}). We will call it the \emph{stable brane action}.

\begin{proof}
Let $t: (\sigma \text{ over } Z, u:Y\to Z)\to (\sigma' \text{ over } Z', v:Y'\to Z')$ be a morphism in $\intcocart \Tw(\Env(\mathfrak{M}))^{\otimes} \times_{\T^{op}} \Fun(\Delta[1], \T)^{op}$ over a map $f:Z\to Z'$ in $\T^{op}$ and let 

\begin{equation}
\label{equation-stabilitypullbacks}
\xymatrix{
Y\times_Z C_{\sigma}\ar[dr]\ar[rr]&&\ar[dl] X\times Z\\
&Z&
}
\end{equation}

\noindent be an object in $\mathrm{B}^{B-gr, \mathrm{Stb}}(\T,\mathfrak{M}, X)$ over $ (\sigma \text{ over } Z, u:Y\to Z)$. As part of the data of $t$ we are given a commutative diagram in $\T$

\begin{equation}
\xymatrix{
Y'\ar[r]\ar[d]^v & Y\ar[d]^u\\
Z'\ar[r]^f& Z
}
\end{equation}

\noindent and by construction of $\pi_X$, cocartesian liftings of $t$ are given by first taking the base-change of the diagram (\ref{equation-stabilitypullbacks}) under $f:Z'\to Z$

\begin{equation}
\label{equation-stabilitypullbacks}
\xymatrix{
(Y\times_Z C_{\sigma})\times_Z Z' \simeq (Y\times_Z Z')\times_{Z'} C_{f^*(\sigma)}\ar[dr]\ar[rr]&&\ar[dl] X\times Z'\\
&Z'&
}
\end{equation}

\noindent and then composing with the canonical map $Y'\to Y\times_Z Z'$. The conclusion now follows because the pullback of a family of stable maps is stable as stability is determined at the level of geometric points.\\

To conclude the proof we have to justify why (\ref{compositionstable}) provides again a map of $\infty$-operads. The condition of weak cartesian structure follows by the arguments used in the proof of the Prop. \ref{propbranestopos}: the tensor structure in  $\Tw(\Env(\mathfrak{M}))^{\otimes}(Z)$ corresponds to the disjoint union of curves. The conditions of the Corollary \ref{label2} follow because the compositions of operations are classified by the gluings of curves along marked points as in the formula (\ref{coherencetoposmonochromatic2}) and the gluing of stable maps is stable, as stability is a local condition.\\

\end{proof}

\end{prop}

\subsection{Gromov-Witten lax action}
\label{section-gwaction}

So far we have been using the operad $\MKO$ that assembles the moduli stacks of Costello. The reason is merely technical: it provides a natural context where the moduli of stable maps appears as part of the brane action. One would now like to extend this to an action of the usual operad of stable curves provided by the family of smooth and proper Deligne-Mumford stacks of stable curves with marked points $\Mnstable$. The composition operation is given by gluing curves along the marked points as in $\MKO$. It is well-known after \cite{MR1601666} and \cite[Section 1.3.9]{math/9906063} that this family forms a (symmetric) unital operad in $1$-stacks by declaring $\overline{\mathcal{M}}_{0,2}$ to be a point thought of as a copy of the projective space with two marked points and only the trivial automorphism and, by modifying the composition law by performing stabilizations after gluing the curves. As in section \ref{section-costelloisoperad} and repeating the shifting of notations in the formula \ref{formula-shiftnotations}, this provides an operad in the $\infty$-topos of derived stacks, which we will denote as $\Mstablemonoidal$. We will leave it to the reader to verify that the canonical maps $\MnK\to \Mnstable\times \{\beta\}$ given by stabilization, assemble to a map of graded operads in $1$-stacks and as these are smooth, to a map of graded $\infty$-operads in derived stacks. More precisely, this is a map of $B$-graded $\infty$-operads in $\T$

$$
\MKO \to \Mstablemonoidal\times_{\Nerve(\Fin)}\Nerve(\Fin^B)
$$

\noindent which by the adjunction (\ref{equation-forgetgradingadjunction}), we can also write as a map of $\infty$-operads 

\begin{equation}
\label{mapstabilizationoperads}
\mathrm{Stb}:\MKO\to \Mstablemonoidal
\end{equation}

\noindent given by the maps

$$
\coprod_\beta \mathfrak{M}_{0.n,\beta}\to \overline{\mathcal{M}}_{0,n}
$$

Our goal in this section is to explore the interaction of the stable action (\ref{equation-stableaction}) with this stabilization morphism. Our main result is the theorem \ref{laxactiongwnaive} written in a somewhat less natural language. To present the results as written in the theorem \ref{laxactiongwnaive} one would need many aspects of the theory of $(\infty,2)$-categories that are not yet available in the literature. We found a way to avoid those aspects that allows us to still give a precise statement while remaining in the setting of $(\infty,1)$-categories and without changing the content, at the cost of a less evident formulation.\\ 

To explain the result, the first observation is that the sheaf of $\infty$-operads in $\T$ given by $\T/(-)^{\mathrm{corr}, \otimes_\times}$ is in fact the $1$-categorical truncation of a sheaf of symmetric monoidal $(\infty,2)$-categories

\begin{equation}
\mathrm{Spans}(\T/(-))^{\otimes_\times}:\, Z\in \T^{op}\mapsto \mathrm{Spans}(\T/Z)^{\otimes_\times}
\end{equation}

This follows from the same arguments as for correspondences, as the construction of spans commutes with products (see \cite{1409.0837}).\\

We claim that any such sheaf can be presented as a sheaf of \emph{categorical operads} in $\T$. Recall from \cite{1307.0405} that a categorical operad in spaces is an $\infty$-functor $\Omega^{op}\to \iCat$ satisfying the Segal conditions. Of course, the inclusion $\Spaces\subseteq \iCat$ produces a fully faithful functor $\mathrm{Op}_{\infty}\to \mathrm{Op}(\iCat):= \Fun^{\mathrm{Segal}}(\Omega^{op}, \iCat)$. Informally speaking, these correspond to multicategories where the collections of $n$-ary operations form $(\infty,1)$-categories. A natural source of categorical operads in spaces are symmetric monoidal $(\infty,2)$-categories: let $\Cmonoidal$ be a symmetric monoidal $(\infty,2)$-category. One can define a categorical operad as follows: to a corolla $T_n$ one assigns the disjoint union of 

$$
\coprod_{(X_1,..., X_n, Y)} \Map_{\C}(X_1\otimes ... \otimes X_n, Y)
$$

\noindent where the $(X_1,..., X_n, Y)$ runs over all the lists of $n+1$ objects in $\C$ and $\Map_{\C}$ is the $(\infty,1)$-category of maps in $\C$. For a general T one imposes the Segal conditions. This construction can be made functorial using the tensor products in $\C$. For the moment being we will avoid to give a precise construction of this assignment and we will just assume it has been constructed. We hope to give a precise construction in a later version of this project.\\

In this case we can exhibit the data of $\mathrm{Spans}(\T/(-))^{\otimes_\times}$ as a limit preserving $\infty$-functor $\T^{op}\to \mathrm{Op}(\iCat)$ which we will again denote as $\mathrm{Spans}(\T/(-))^{\otimes_\times}$. At the same time, both $\MKO$ and $\Mstablemonoidal$ can be presented as categorical operads via the inclusion

$$
\T^{op}\to \mathrm{Op}_{\infty}\subseteq \mathrm{Op}(\iCat)
$$

\noindent and the map of $\infty$-operads in $\T$ encoding the brane action $\MKO\to \T/(-)^{\mathrm{corr}, \otimes_\times}$ is equivalent via the universal property of the $1$-categorical truncation to a map of categorical operads $\MKO\to \mathrm{Spans}(\T/(-))^{\otimes_\times}$ that factors through the maximal $(\infty,1)$-category. As a result we find a correspondence of categorical operads in $\T$

\begin{equation}
\label{correspondencelaxorigin}
\xymatrix{
\Mstablemonoidal& \MKO\ar[l]\ar[r]& \mathrm{Spans}(\T/(-))^{\otimes_\times}
}
\end{equation}

Using the Grothendieck construction, each categorical operad in $\T$ can also be presented as a bifibration over $\Omega^\mathrm{op}\times \T$ and the maps of operads in (\ref{correspondencelaxorigin}) produce maps that preserve $\T$-cartesian edges and $\Omega^\mathrm{op}$-cocartesian edges (ie, the Segal conditions).

\begin{equation}
\label{correspondencelaxorigin}
\xymatrix{
\int \Mstablemonoidal\ar[dr]^r& \ar[d]^q\int \MKO\ar[l]_{\mathrm{Stb}}\ar[r]&\ar[dl]^p \int \mathrm{Spans}(\T/(-))^{\otimes_\times}\\
&\Omega^\mathrm{op}\times \C&
}
\end{equation}

We can now present our formulation of the Theorem \ref{laxactiongwnaive}:

\begin{thm}
\label{laxactiongw}
There exists an $\infty$-functor

\begin{equation}
\label{correspondencelaxorigin}
\xymatrix{
\int \Mstablemonoidal \ar[dr]^r\ar[rr]&&\ar[dl]^p \int \mathrm{Spans}(\T/(-))^{\otimes_\times}\\
&\Omega^\mathrm{op}\times \T&
}
\end{equation}

\noindent given informally as follows: for each $Z\in \C$ and for each corolla $T_n$, it sends a curve $\sigma: Z\to \Mnstable$ to the correspondence

\begin{equation}
\xymatrix{
\prod_n X\times Z \ar[dr]&\ar[d]\ar[l]\ar[r](\coprod_\beta \R\overline{\mathcal{M}}_{0,n}(X,\beta))\times_{\overline{\mathcal{M}}_{0,n}}Z&  \ar[dl]X\times Z\\
&Z&
}
\end{equation}

Moreover, this map does preserve cartesian edges with respect to the projection to $\T$ but does not preserve cocartesian edges with respect to $\Omega^\mathrm{op}$.

\begin{proof}
We construct the required $\infty$-functor as a relative left Kan extension

\begin{equation}
\label{correspondencelaxorigin45}
\xymatrix{
\int \MKO\ar[d]\ar[r]\ar[d]^{\mathrm{Stb}}&\ar[d]^p \int \mathrm{Spans}(\T/(-))^{\otimes_\times}\\
\int \Mstablemonoidal\ar[r]^r\ar@{-->}[ur]& \T\times \Omega^\mathrm{op}
}
\end{equation}

%

In order to proceed we remark that $p$ is in fact a locally coCartesian fibration. So far $p$ was constructed as a bifibration: cocartesian with respect to $\Omega^\mathrm{op}$ and cartesian with respect to $\T$. Given a corolla $T_n\in \Omega^\mathrm{op}$ and $Z\in T$, the fiber of $p$ over $(T_n,Z)$ is given by a comma category in the topos theory, where the relavant part is the category $\Map_{Spans(T/Z)}(\prod_n X\times Z, X\times Z)$. We now remark that in fact, the cartesian fibration defined by taking the fiber over a corolla $T_n$, 

$$(\int \mathrm{Spans}(\T/(-))^{\otimes_\times})\times_{\Omega^\mathrm{op}}\{T_n\} \to \T$$
is in fact a biCartesian fibration. Indeed, as the the cartesian structure is given by base change, the biCartesian structure is given by the left adjoint of the base-change, meaning, the composition functors. 
We are in the following situation:

\begin{enumerate}
\item The composition $\mathrm{Spans}(\T/(-))^{\otimes_\times}\to \T\times \Omega^\mathrm{op}$ is a Cartesian fibration; 
\item For each $Z\in \T$, the projection $\mathrm{Spans}(\T/(-))^{\otimes_\times}\times_\T\{Z\}\to \Omega^\mathrm{op}$ is cocartesian;
\item The fiber over a corolla $T_n$, $(\int \mathrm{Spans}(\T/(-))^{\otimes_\times})\times_{\Omega^\mathrm{op}}\{T_n\}\to \T$ is a cocartesian fibration via the forgetful functor.
\end{enumerate}

It follows then from the same arguments as in \cite[4.5.3.4]{lurie-ha} (using the Segal conditions instead of the inert cocartesian liftings) that $p$ is a locally coCartesian fibration. It follows then by \cite[4.3.1.10]{lurie-htt} that $p$-colimits are colimits in the fibers that remain colimits under change of fiber via $p$-locally cocartesian morphisms. As the forgetful functors between comma categories in a topos commutes with colimits and the cocartesian liftings of maps in $\Omega^{\mathrm{op}}$ are given by taking pullbacks in a topos, thus also commuting with colimits, it follows that $p$ admits all $p$-colimits. This is enough to deduce the existence of a lifting map as in the diagram (\ref{correspondencelaxorigin45}), using the existence theorem \cite[4.3.2.13]{lurie-htt} and the results of \cite[Section 4.3.3]{lurie-htt}.

One must now show that this relative left Kan extension gives back the formula in the statement of the theorem. For this we make a second preliminary observation: that in fact the functor 

$$\mathrm{Stb}: \int \MKO \to \int \Mstablemonoidal$$ 

\noindent admits a right adjoint relatively to the projection to $\Omega^\mathrm{op}$. Indeed, this follows from  a dual version of \cite[7.3.2.6]{lurie-ha} as for each corolla  $T_n\in \Omega^\mathrm{op}$ we have canonical identifications 
$$(\int \MKO)\times_{\Omega^\mathrm{op}}\{T_n\}\simeq \T/(\coprod_\beta\mathfrak{M}_{0,n,\beta})$$
and
$$(\int  \Mstablemonoidal)\times_{\Omega^\mathrm{op}}\{T_n\}\simeq \T/\overline{\mathcal{M}}_{0,n}$$
and as $\mathrm{Stb}$ is given by the composition with the stabilization map $(\coprod_\beta\mathfrak{M}_{0,n,\beta})\to \overline{\mathcal{M}}_{0,n}$, it has a right adjoint given by pulling back also along the stabilization map.  Combined with the arguments in the proof of \cite[4.3.3.9]{lurie-htt} (as both projections to $\Omega^{\mathrm{op}}$ are cocartesian fibrations), this implies that for for each $\sigma$ in $\int \Mstablemonoidal$ over $Z\in \T$, the comma category  $(\int \MKO)_{/\sigma}$ is equivalent to the comma topos $\T/(Z\times_{\overline{\mathcal{M}}_{0,n}}(\coprod_\beta \mathfrak{M}_{0,n, \beta)})$
and the formula for the $p$-relative Kan extension sends $\sigma$ to the colimit in the fiber over $Z$

$$
\mathrm{colim}_{Z'  \to (\coprod_\beta \mathfrak{M}_{0,n, \beta})\times_{\overline{\mathcal{M}}_{0,n}} Z}\,\, \, \mathrm{forget} ( \mathrm{Hom}_{/Z'}^{\mathrm{Stb}}(C_{\tilde{\sigma}}, X\times Z' )) $$

\noindent where $\mathrm{forget}$ is the functor that sees an object over $Z'$ as an object over $Z$ along the composition with $u$. This is the same as 
 
$$\simeq \mathrm{colim}_{Z'  \to \mathfrak{M}_{0,n}\times_{\overline{\mathcal{M}}_{0,n}} Z}\,\, \,  \mathrm{forget} ( \coprod_\beta \R\overline{\mathcal{M}}_{0,n}(X,\beta))\times_{\coprod_\beta \mathfrak{M}_{0,n}}Z' )$$

\noindent which, as the forgetful functor is a left adjoint, is the same as

$$\simeq \mathrm{forget}\, ( \mathrm{colim}_{Z'  \to \mathfrak{M}_{0,n}\times_{\overline{\mathcal{M}}_{0,n}} Z}\,\, \,   \coprod_\beta \R\overline{\mathcal{M}}_{0,n}(X,\beta))\times_{\mathfrak{M}_{0,n}}Z' )$$

$$\simeq \mathrm{forget}\, ( \mathrm{colim}_{Z'  \to \mathfrak{M}_{0,n}\times_{\overline{\mathcal{M}}_{0,n}} Z}\,\, \,   \coprod_\beta \R\overline{\mathcal{M}}_{0,n}(X,\beta))\times_{\mathfrak{M}_{0,n}}(Z\times_{\overline{\mathcal{M}}_{0,n}}\mathfrak{M}_{0,n})\times_{(Z\times_{\overline{\mathcal{M}}_{0,n}}\mathfrak{M}_{0,n})} Z'$$

$$\simeq \mathrm{forget}\, ( \coprod_\beta \R\overline{\mathcal{M}}_{0,n}(X,\beta))\times_{\mathfrak{M}_{0,n}}(Z\times_{\overline{\mathcal{M}}_{0,n}}\mathfrak{M}_{0,n})\times_{(Z\times_{\overline{\mathcal{M}}_{0,n}}\mathfrak{M}_{0,n})}(\mathrm{colim}_{Z'  \to \mathfrak{M}_{0,n}\times_{\overline{\mathcal{M}}_{0,n}} Z}\,\, \, Z'))$$

$$\simeq \mathrm{forget}\, ( \coprod_\beta \R\overline{\mathcal{M}}_{0,n}(X,\beta))\times_{\overline{\mathcal{M}}_{0,n}} Z)\times_{(Z\times_{\overline{\mathcal{M}}_{0,n}}\mathfrak{M}_{0,n})}(Z\times_{\overline{\mathcal{M}}_{0,n}}\mathfrak{M}_{0,n}))$$

$$\simeq \mathrm{forget}\, ( \coprod_\beta \R\overline{\mathcal{M}}_{0,n}(X,\beta))\times_{\overline{\mathcal{M}}_{0,n}} Z)$$
 
It is immediate to see from this description  that this map preserves cartesian edges relatively to $\T$. Moreover, it is also clear that it is defined over $\Omega^\mathrm{op}$ as both $\mathrm{Stb}$ and the brane action are, and the $p_\T$-cocartesian structure is defined fiberwise relatively to $\Omega^\mathrm{op}$. We will now explain why it does not preserve cocartesian edges relatively to $\Omega^\mathrm{op}$. To simplify the notations let us write

\begin{equation}
\label{simpleformstablemaps}
\R\overline{\mathcal{M}}^\sigma_{0,n}(X) :=(\coprod_\beta \R\overline{\mathcal{M}}_{0,n}(X,\beta))\times_{\overline{\mathcal{M}}_{0,n}}Z
\end{equation}

\noindent for $\sigma:Z\to  \overline{\mathcal{M}}_{0,n}$. Let $T$ be a tree in $\Omega$ consisting of a gluing of a corolla $T_{n-1}$ to a corolla $T_{m-1}$ where the root of $T_{n-1}$ is attached to the first leaf of $T_{m-1}$. Then, because of the Segal condition we can think of an object in $\int \Mstablemonoidal$ over $(T,Z)$ as a pair $(\sigma,\tau)$ of composable curves over $Z$ where $\sigma$ has $n$ marked points and $\tau$ has $m$ marked points. Then by the previous discussion, the relative Kan extension sends the object  $(\sigma, \tau)$ over $(T,Z)$ to the pair of arrows $(\R\overline{\mathcal{M}}^\sigma_{0,n}(X)\to X^n, \R\overline{\mathcal{M}}^\tau_{0,m}(X)\to X^m)$ - here we use again the Segal condition for $\mathrm{Spans}(\T/(-))^{\otimes_\times}$ to identify objects over $(T,Z)$ as pairs. By definition, a $r$-cocartesian lifting for the contraction map $ (T_{n+m-2}, Z)\to (T, Z)$ in $\Omega$ gives the gluing of the two curves $(\sigma, \tau)\to \tau\circ \sigma$. Its target is sent to the map $\R\overline{\mathcal{M}}^{\tau\circ \sigma}_{0,n+m-2}(X)\to X^{n+m-2}$ while by definition of the cocartesian fibration $p$ (relatively to $\Omega^{\mathrm{op}}$), a  $p$-cocartesian lifting of the same contraction map has target the map $\R\overline{\mathcal{M}}^\sigma_{0,n}(X)\times_X \R\overline{\mathcal{M}}^\tau_{0,m}(X)\to X^{n+m-2}$. The universal property of $p_\Omega^\mathrm{op}$-cocartesian morphisms then gives us a canonical map $\R\overline{\mathcal{M}}^\sigma_{0,n}(X)\times_X \R\overline{\mathcal{M}}^\tau_{0,m}(X)\to \R\overline{\mathcal{M}}^{\tau\circ \sigma}_{0,n+m-2}(X)$ which corresponds to the gluing of the two stable maps. This is not an equivalence in general.

\end{proof}

\end{thm}

Following the terminology of  \cite{1307.0405} we will say that this map obtained in the theorem is a \emph{very lax} map of categorical operads from $\Mstablemonoidal$ to $\mathrm{Spans}(\T/(-))^{\otimes_\times}$, and we will denote it as  $\Mstablemonoidal\leadsto \mathrm{Spans}(\T/(-))^{\otimes_\times}$. Unwinding the definitions this encodes the coherences of an action, given by universal correspondences

\begin{equation}
\xymatrix{
&\ar[dl]\ar[dr]\coprod_\beta \stablemaps &\\
\Mnstable && X^n
}
\end{equation}

\noindent and satisfying a lax associative law given by the gluing maps

\begin{equation}
\label{laxassociativestructurediagram}
\resizebox{1.1\hsize}{!}{
\xymatrix{
\Mnstable\times \overline{\mathcal{M}}_{0,m} &\ar[l](\coprod_\beta \R\overline{\mathcal{M}}_{0,n}(X,\beta)\times (\coprod_\beta \R\overline{\mathcal{M}}_{0,m}(X,\beta))\ar[r]& X^n\times X^m=X^{n-1}\times X \times X \times X^{m-1}\\
\Mnstable\times \overline{\mathcal{M}}_{0,m}\ar@{=}[u]&\ar @{} [ur] |{\text{\LARGE{$\urcorner_\nu$}}} \ar@{-->}[d]\ar[u]\ar[r](\coprod_\beta \R\overline{\mathcal{M}}_{0,n}(X,\beta)\times_X (\coprod_\beta \R\overline{\mathcal{M}}_{0,m}(X,\beta))& X^{n-1} \times X\times X^{m-1}\ar[u]_{id^{n-1}\times \Delta \times id^{m-1}}\ar@{=}[d]\\
\Mnstable\times \overline{\mathcal{M}}_{0,m}\ar@{=}[u]\ar[d]& \ar @{} [dl] |{\text{\LARGE{$\llcorner_\mu$}}} \ar[l]\ar[d](\coprod_\beta \R\overline{\mathcal{M}}_{0,n+m-2}(X,\beta))\times_{\overline{\mathcal{M}}_{0,n+m-2}}(\Mnstable\times \overline{\mathcal{M}}_{0,m})& X^{n-1} \times X\times X^{m-1}\ar[d]^{id^{n-1}\times *\times id^{m-1}}\\
\overline{\mathcal{M}}_{0,n+m-2}&\coprod_\beta \R\overline{\mathcal{M}}_{0,n+m-2}(X,\beta)\ar[l]\ar[r]& X^{n-1}\times * \times X^{m-1}
}
}
\end{equation}

\noindent which are non-invertible.

\begin{remark}
\label{disjointunionArtin}
Let us also remark that all the derived stacks involved in this action are derived geometric stacks. This follows from \cite[1.3.3.4, 1.3.3.5]{toen-vezzosi-hag2} which shows that the notion of being $n$-geometric is local, stable under pullbacks and in particular closed under small disjoint unions. Another important consequence of this is that we have the base-change formula for the two pullback squares in the diagram - see \cite[Cor. 1.4.5]{1108.5351} or \cite[B.15]{1402.3204}.
\end{remark}

To conclude this section we will also show that our lax action admits a graded version. This will be more useful to us in the next sections. Indeed, we can start with the correspondence of graded operads

\begin{equation}
\xymatrix{
\Mstablemonoidal\times_{\Nerve(\Fin)} \Nerve(\Fin^B)  &\ar[l]\ar[r] \MKO & \T/(-)^{\mathrm{corr}, \otimes_\times}\times_{\Nerve(\Fin)} \Nerve(\Fin^B) 
}
\end{equation}

\noindent and replacing the category of trees $\Omega$ by the category $\Omega_B$ of the Remark \ref{remark-gradeddendroidal} \todo{CAREFUL! HERE WE ARE USING THE EQUIVALENCE DENDROIDAL-LURIE}we can consider the corresponding notion of \emph{graded categorical operads in $\T$}. The definitions apply mutatis-mutandis.

Like in (\ref{correspondencelaxorigin}) we have a correspondence of such objects which using the Grothendieck construction we can exhibit as bifibrations

\begin{equation}
\label{correspondencelaxorigin2}
\xymatrix{
\int \Mstablemonoidal \times_{\Nerve(\Fin)} \Nerve(\Fin^B)  \ar[dr]^r& \ar[d]^q\int \MKO\ar[l]\ar[r]&\ar[dl]^{p_B} \int \mathrm{Spans}(\T/(-))^{\otimes_\times}\times_{\Nerve(\Fin)} \Nerve(\Fin^B) \\
&\Omega_B^\mathrm{op}\times \T&
}
\end{equation}

We have the following graded version of the lax action:

\begin{prop}
\label{laxactiongwgraded}
There exists an $\infty$-functor

\begin{equation}
\label{correspondencelaxorigin}
\xymatrix{
\int \Mstablemonoidal \times_{\Nerve(\Fin)} \Nerve(\Fin^B) \ar[dr]^r\ar[rr]&&\ar[dl]^p \int \mathrm{Spans}(\T/(-))^{\otimes_\times}\times_{\Nerve(\Fin)} \Nerve(\Fin^B) \\
&\Omega_B^\mathrm{op}\times \T&
}
\end{equation}


\noindent given informally as follows: for each $Z\in \C$ and for each corolla $T_n$, it sends a curve $\sigma: Z\to \Mnstable$ together with the choice of an element $\beta$, to the correspondence

\begin{equation}
\label{eq-thebloodmoonrisesonceagain}
\xymatrix{
\prod_n X\times Z \ar[dr]&\ar[d]\ar[l]\ar[r]  \R\overline{\mathcal{M}}_{0,n}(X,\beta)\times_{\overline{\mathcal{M}}_{0,n}}Z &  \ar[dl]X\times Z\\
&Z&
}
\end{equation}

Moreover, this map preserves cartesian edges relatively to $\T$ but does not preserve cocartesian edges relatively to the $\Omega$-direction.

\begin{proof}
The proof follows from the same arguments as in the theorem \ref{laxactiongw}. In this case the lax structure is given by the gluing maps

$$
\R \overline{\mathcal{M}}^\sigma_{0,n}(X, \beta)\times_X \R \overline{\mathcal{M}}^\tau_{0,m}(X, \beta')\to \R \overline{\mathcal{M}}^{\sigma\circ \tau}_{0,n+m-2}(X, \beta+\beta')\times_{\overline{\mathcal{M}}_{0,n+m-2}}(\overline{\mathcal{M}}_{0,n}\times \overline{\mathcal{M}}_{0,m})
$$

\noindent where $\R \overline{\mathcal{M}}^\sigma_{0,n}(X, \beta)$ is the open component of $\R \overline{\mathcal{M}}^\sigma_{0,n}(X)$ consisting of all stable maps with total degree $\beta$.

\end{proof}

\end{prop}

\section{Categorification of GW-invariants and Quantum $\Ktheory$-theory}

\subsection{Categorification}
\label{section-categorification}

We know from the theorem \ref{mainthm} that any smooth projective algebraic variety $X$, seen as an object in correspondences, carries a natural action of the graded operad $\MKO$. We now explain how to extend this action to the derived category of $X$ and that this action restricts to both perfect and coherent complexes.

Let  $\mathrm{Sp}^{\otimes}$ denote the symmetric monoidal $(\infty,1)$-category of spectra and let $\derivedkmonoidal:=\Mod_k(\Sp)^{\otimes}$ denote the $\infty$-categorical version of the derived category of $k$ with its standard symmetric monoidal structure. We set

$$\dgcontk:= \Mod_{\derivedk}(\Prlstable)$$

\noindent where $\Prlstable$ is the $(\infty,1)$-category of stable presentable $(\infty,1)$-categories. Moreover, we observe that the site $\C:=\dst^{\mathrm{aff}}$ is equivalent to $\CAlg(\derivedk^{\leq 0})$ for the natural $t$-structure in $\Sp$. In this case, and following \cite[Section 3.1]{Gaitsgory-Nick-book}, taking modules defines a lax monoidal $\infty$-functor

$$
\Mod: \C^{op}\to \CAlg(\dgcontk)
$$

\noindent informally described by the formula $A\mapsto \Mod_A(\derivedk)^{\otimes}\simeq \Mod_A(\Sp)^{\otimes}$. Thanks to \cite[Thm 1.3.7.2]{toen-vezzosi-hag2} this $\infty$-functor has fpqc descent and therefore, by Kan extension, provides a limit preserving functor $\Qcoh:\T^{op}\to \CAlg(\dgcontk)$.

One can also consider the compositon

$$
\xymatrix{\C^{op}\ar[r]^-{\Qcoh}& \CAlg(\dgcontk)\ar[rrr]^-{\Mod_{(-)}(\dgcontk)}&&& \CAlg(\mathrm{Cat}_{\infty}^{big})}
$$

\noindent which, thanks to \cite[Appendix A]{gaitsgory-ShCats} or \cite[Remark 2.5]{toen-azumaya}, satisfies fppf descent. Here $\mathrm{Cat}_{\infty}^{big})$ denotes the (very large) $(\infty,1)$-category of not necessarily small $(\infty,1)$-categories.  Again, it can be Kan extended to a limit preserving functor $\dgcont:\T^{op}\to \CAlg(\mathrm{Cat}_{\infty}^{big})$.\\

We start this section with the construction of a (lax monoidal) map of $(\infty,1)$-categories over $\T$

\begin{equation}
\label{natcat}
(\T/(-))^{op}\to \dgcont
\end{equation}

\noindent informally given by the following formula: for $Z\in \C$, the map

$$
(\T/(Z))^{op, \times}\to \dgcont(Z)
$$

\noindent sends $f:Y\to Z$ to the $(\infty,1)$-category $\Qcoh(Y)$ seen as a $\Qcoh(Z)$-module via the action by $f^*:= \Qcoh(f)$.\\

Consider first the functor $\Fun(\Delta[1], \T^{op})\to \Fun(\Delta[1], \CAlg(\dgcontk))$ obtained by composition with $\Qcoh$. Now, recall now from \cite[Section 3.3.3]{lurie-ha} the construction of a generalized $\infty$-operad $\Mod(\dgcontk)^{\otimes}\to \CAlg(\dgcontk)\times \Nerve(\Fin)$ whose fiber over $(\mathrm{V}^{\otimes}, \onefin)\in \CAlg(\dgcontk)\times \Nerve(\Fin)$ is $\Mod_V(\dgcontk)$. In general, an object in $\Mod(\dgcontk)^{\otimes}$ over $\onefin$ can be thought of as a pair $(\mathrm{V}^{\otimes}, \M)$ where $\mathrm{V}^{\otimes}$ is a presentable stable $k$-linear symmetric monoidal $(\infty,1)$-category and $\M$ is another presentable stable $k$-linear $(\infty,1)$-category equipped with a structure of $\mathrm{V}^{\otimes}$-module. There is now a natural $\infty$-functor over $\CAlg(\dgcontk)$

$$
\xymatrix{
 \Fun(\Delta[1], \CAlg(\dgcontk))\ar[dr]^{ev_0}\ar[rr]&& \Mod(\dgcontk)^{\otimes}_{\onefin}\ar[dl]\\
&\CAlg(\dgcontk)&
}
$$

\noindent which to a symmetric monoidal functor $F:\mathrm{V}^{\otimes}\to \mathrm{W}^{\otimes}$ assigns the pair $(\mathrm{V}^{\otimes}, \mathrm{W})$ where $\mathrm{W}$ is now seen as a $\mathrm{V}^{\otimes}$-module via $F$. See for instance the discussion in \cite[pag. 249]{robalo-thesis}. Finally, by composition with $\Qcoh$ we obtain a commutative diagram

\begin{equation}
\label{natcat2}
\xymatrix{
 \Fun(\Delta[1], \T^{op})\ar[d]^-{ev_0}\ar[r]& \Mod(\dgcontk)^{\otimes}_{\onefin}\ar[d]\\
\T^{op}\ar[r]^-{\Qcoh}&\CAlg(\dgcontk)
}
\end{equation}

To conclude the construction of (\ref{natcat}) we observe that 

\begin{itemize}
\item we have

$$
\intcocart (\T/(-))^{op} \simeq (\intcart \T/(-))^{op}
$$

 \noindent and the last is given by  $\Fun(\Delta[1], \T^{op})$ together with the evaluation at zero to $\T^{op}$;

\item $\intcocart \dgcont$ can be canonically identified with the fiber product 

$$\T^{op}\times_{\CAlg(\dgcontk)}\Mod(\dgcontk)^{\otimes}_{\onefin}$$

.
\end{itemize}

Therefore, the commutativity of (\ref{natcat2}) and the universal property of pullbacks give the map (\ref{natcat}).\\

For technical reasons we will need to impose some conditions in the derived stacks we work with. As in \cite{1307.0405}, we consider a full subcategory $\V\subseteq \T$, namely, we will consider $\V$ the full subcategory spanned by all perfect stacks in the sense of \cite[Section 3]{MR2669705}. In particular, thanks to \cite[Cor 3.22]{MR2669705}, stacks of the form $Y/G$ with $Y$ a quasi-projective derived scheme of finite presentation and $G$ is a smooth linear algebraic group action on $Y$ over $k$, are perfect. Also, thanks to the combination of \cite[Cor 5.2]{toen-azumaya} and \cite[Theorem 1]{MR1492534} (see also \cite[p.4]{1405.1887}) quasi-compact separated Deligne-Mumford stacks with coarse moduli space, such as the stack of stable maps, are perfect. We will now summarize the nice features of $\V$:

\begin{enumerate}
\item Any derived affine and any smooth and quasi-projective scheme of finite presentation belongs to $\V$. In particular, $X$ belongs to $\V$;\\
\item The inclusion $\V\subseteq \T$ is closed under products and $\V$ is a generating site for $\T$. Moreover, the $\infty$-functor

$$
\V/(-)^{\times}: Z\in \C^{op}\mapsto (\V/Z)^{\times} 
$$

\noindent is a stack of symmetric monoidal $(\infty,1)$-categories and the inclusion $\V\hookrightarrow \T$ induces a natural transformation

$$
\V/(-)\to \T/(-)
$$

\noindent which induces a new natural transformation of strong monoidal functors

$$
\V/(-)^{\mathrm{corr}, \otimes_\times}\to \T/(-) ^{\mathrm{corr}, \otimes_\times}
$$

\noindent This one being faithful but not full.\\

\item The morphism encoding the stable action on $X$

$$
\MKO\to \T/(-)^{\mathrm{corr}, \otimes_\times}
$$

\noindent factors through the (non-full) inclusion 

$$\V/(-)^{\mathrm{corr}, \otimes_\times}\hookrightarrow \T/(-)^{\mathrm{corr}, \otimes_\times}$$

\noindent This follows from the formula (\ref{eq-thebloodmoonrisesonceagain}) defining the action and the fact that the stacks appearing in the middle are Deligne-Mumford with course moduli a scheme, therefore perfect.

\medskip

\item The restriction $\Qcoh:\V^{op}\hookrightarrow \T^{op}\to \CAlg(\dgcontk)$ has values in the full subcategory $\CAlg(\dgcontkcompact)$ spanned by the dg-categories having compact generators. In this case $\Qcoh$ is a strongly monoidal functor. Thanks to \cite[Thm. 0.2 and Remark 2.9]{toen-azumaya}, the full sub-prestack of dg-categories having compact generators $\dgcontcompact\subseteq \dgcont$ is also a stack for the fppf topology. Moreover, this inclusion is monoidal and in this case the composition $\V/(-)^{\times}\to \T/(-)^{\times}\to \dgcontcompact$ defines a monoidal map of stacks in symmetric monoidal categories. \\

\item Notice that $\dgcontkcompact$ is the 1-categorical truncation of $(\infty,2)$-category- following the results \cite[Thm 1.4 and Cor 1.8]{Toen-homotopytheorydgcatsandderivedmoritaequivalences} the hom-categories are given by bi-modules. In the same way, $\dgcontcompact$ is a sheaf with values in symmetric monoidal $(\infty,2)$-categories. The composition $\V/(-)\to \dgcontcompact$ satisfies the base-change conditions of section \ref{correspondencesandtwistedarrows} - base change for derived Artin stacks -  see \cite[Cor. 1.4.5]{1108.5351}, \cite[B.15]{1402.3204} or \cite[Prop. 3.10 and 3.23 ]{MR2669705}. By the $(\infty,2)$-universal monoidal property of correspondences  \cite[Part V]{Gaitsgory-Nick-book} applied object-wise, it extends in a essentially unique way to a strongly monoidal $(\infty,2)$-functor\\

$$\mathrm{Spans}(\V/(-))^{\otimes_\times}\to \dgcontcompact$$\\

\end{enumerate}

Finally, combining  (3) and (5) we find a map of  $\infty$-operads in $\T$

\begin{equation}
\label{costellocategorified}
\MKO\to \V/(-)^{\mathrm{corr}, \otimes_\times}\subseteq \mathrm{Spans}(\V/(-))^{\otimes_\times}\to \dgcontcompact
\end{equation}

\noindent exhibiting an algebra structure on $\Qcoh(X)$. By construction, over each $Z$ affine, the map of graded $\infty$-operads

$$
\MKO(Z) \to (\V/Z)^{\mathrm{corr}, \otimes_\times}\to \dgcontcompact(Z) 
$$

\noindent is defined by sending a stable curve $\sigma:Z\to\MnK$ to the correspondence over $Z$

\begin{equation}
\label{formulaoverZfinal}
\xymatrix{
\prod_{n-1} X\times Z \ar[dr]& \ar[l] \ar[r]\ar[d] Z\times_{\MnK}\stablemaps & X\times Z\ar[dl]\\ 
&Z&
}
\end{equation}

\noindent and then,  unwinding the universal property of correspondences, to the functor in $\dgcont(Z)= \Mod_{\Qcoh(Z)}(\dgcontkcompact)$  given by pullback-pushforward along (\ref{formulaoverZfinal})

\begin{equation}
\label{formulaoverZfinal2}
\resizebox{1.0\hsize}{!}{
\xymatrix{
&\Qcoh( Z\times_{\MnK}\stablemaps)\simeq \Qcoh(Z\times_{\MnK}\stablemaps) \ar[dr]^-{\text{pushforward}}&\\
\Qcoh(\prod_{n-1} X\times Z) \ar[ur]^-{pullback}&& \Qcoh(X\times Z)
}
}
\end{equation}

By the items (1)-(5) above, this is equivalent to a map of compactly generated dg-categories over $Z$

\begin{equation}
\label{formulaoverZfinal3}
\xymatrix{
\Qcoh( X\times Z)^{\otimes^{n-1}_Z} \ar[rr] && \Qcoh(X\times Z)
}
\end{equation}

\noindent and as the base change $\Qcoh(Z)\otimes - : \dgcontkcompact\to \dgcontcompact(Z)$ is monoidal and admits a right adjoint given by the forgetful functor, this is equivalent to the data of a map in $\dgcontk$

\begin{equation}
\label{formulaoverZfinal4}
\xymatrix{
\Qcoh( X)^{\otimes^{n-1}} \ar[rr] && \Qcoh(X)\otimes \Qcoh(Z)
}
\end{equation}

\begin{cor}
\label{corcategorification} The map  (\ref{costellocategorified}) makes $\Qcoh(X)$ an algebra over the graded categorical operad $\{\Qcoh(\MnK)\}_{n,\beta}$. The algebra structure is completely determined by the pullback-pushfoward maps in $\dgcontk$

\begin{equation}
\label{formulaoverZfinal5}
\xymatrix{
\Qcoh( X)^{\otimes^{n-1}} \ar[rr] && \Qcoh(X)\otimes \Qcoh(\MnK)\simeq \Qcoh(X\times\MnK) 
}
\end{equation}

\noindent for $n\geq 2$ and $\beta\in \NE(X)$.

\end{cor}

We can now repeat the strategy used to construct the lax action of section \ref{section-gwaction}. As $\dgcontkcompact$ has a natural structure of symmetric monoidal $(\infty,2)$-category, we can also encode the data of the limit preserving functor $\dgcontcompact: \T^{op}\to \CAlg(\Prl)$ as the 1-categorical truncation of a categorical operad in $\T$,  $\dgcontcompact: \T^{op}\to \mathrm{Op}(\iCat)$. Repeating the same arguments as in section \ref{section-gwaction} we obtain a diagram

\begin{equation}
\label{laxcategorificationdiagram}
\xymatrix{
\int \Mstablemonoidal\times_{\Nerve(\Fin)} \Nerve(\Fin^B) \ar[dr]^r& \ar[d]^q\int \MKO\ar[l]_-{\mathrm{Stb}}\ar[r]&\ar[dl]^l \int \dgcontcompact \times_{\Nerve(\Fin)} \Nerve(\Fin^B)\\
&\Omega_B^\mathrm{op}\times \T&
}
\end{equation}

\noindent in the same conditions of the diagram (\ref{correspondencelaxorigin}). The next proposition establishes the categorification of the lax Gromov-Witten action:

\begin{prop}
\label{laxactiongwcategorification}
There exists an $\infty$-functor

\begin{equation}
\label{correspondencelaxorigin3}
\xymatrix{
\int \Mstablemonoidal \times_{\Nerve(\Fin)} \Nerve(\Fin^B) \ar[dr]^r\ar[rr]&&\ar[dl]^p \int \dgcontcompact\times_{\Nerve(\Fin)} \Nerve(\Fin^B) \\
&\Omega_B^\mathrm{op}\times \T&
}
\end{equation}


\noindent given informally by pullback-pushfoward along the universal correspondences (\ref{maincorrespondence})

 \begin{equation}
 \label{actioncategorieslaxmaps}
 \xymatrix{\Qcoh(X)^{\otimes_{n-1}}\ar[r] &\Qcoh(X \times \Mnstable)}
 \end{equation}
   
 Moreover, this map does send $\C$-cartesian edges to $\C$-cartesian edges but does not preserve cocartesian edges in the $\Omega$-direction.
\end{prop}

Notice that as the $\Mnstable$ is  smooth and proper, the dg-category $\Qcoh(\Mnstable)$ is dualizable object in $\dgcontcompact_k$ and also $\Qcoh(X\times \Mnstable)\simeq \Qcoh(X)\otimes \Qcoh(\Mnstable)$. Therefore the maps (\ref{actioncategorieslaxmaps}) are equivalent to the pullback-pushforward maps

  \begin{equation}
  \label{actioncategorieslaxmaps2}
\xymatrix{\Qcoh(\Mnstable)\otimes \Qcoh(X)^{\otimes_{n-1}}\ar[r] &\Qcoh(X)}
  \end{equation}

\begin{prop}
\label{prop-compatiblePerfCoh}
This action is compatible with the subcategories $\Coh$ and $\Perf$.
\begin{proof}
As the pullback of perfect along any map is still perfect, for $\Perf$ one only needs to justify why pushfoward along the maps in the diagrams preserve perfect. This is because $X$ and $\Mnstable$ are both proper smooth algebraic varieties and $\stablemaps$ is known to be a proper quasi-smooth derived Deligne-Mumford stack \cite[Section 2.2.4]{toen-vezzosi-hag2} so that the stabilization-evaluation maps $\stablemaps\to X^n\times \Mnstable$ are therefore proper and quasi-smooth.  In this case, as explained in \cite{properlocal} pushfowards preserve perfect complexes.

For $\Coh$ one has to justify both compatibilities with pushfowards and pullbacks: compatibility with pushfowards follows because the map is proper (see \cite[Lemma 3.3.5]{gaitsgory-IndCoh}) and for pullbacks we use the fact that pullbacks along quasi-smooth maps between proper DM-stacks are of finite Tor amplitude . \todo{\cite{} NEED TO SHOW THIS}
\end{proof}
\end{prop}

\subsection{The effects of the lax action on $\Ktheory$-theory}
\label{section-quantumKfromlax}

In this section we explore how the categorified lax action obtained at the end of the previous section induces an action on $\Ktheory$-theory.

We start by recalling that the $\Ktheory$-theory spectrum of an algebraic variety $X$ is defined to be the $\Ktheory$-theory spectrum of the dg-category of perfect complexes $\Ktheory(\Perf(X))$. Respectively, the $\mathrm{G}$-theory spectrum is defined to be the $\Ktheory$-spectrum of the dg-category $\Coh(X)$ which by definition is the full subcategory of $\Qcoh(X)$ spanned by those complexes of bounded cohomological amplitude and coherent cohomology. By a well-known theorem of Serre, if $X$ is smooth then the inclusion $\Perf(X)\subseteq \Coh(X)$ is an equivalence of $(\infty,1)$-categories and therefore the $\Ktheory$ and $\mathrm{G}$-theories agree. 

Let $F=\stablemaps$ and let $E_1,..., E_n\in \Ktheory_0(X)=\Gtheory_0(X)$ and $P\in \Ktheory_0(\Mnstable)=\Gtheory_0(\Mnstable)$. The $\Ktheory$-theoretic Gromov-Witten numbers that appear naturally from our lax action in the Prop. \ref{laxactiongwcategorification} are defined by

$$p_*(ev^*(E_1,.., E_n)\otimes \mathrm{Stb}^*(P)))=p_*(\mathcal{O}_{\stablemaps}\otimes ev^*(E_1,.., E_n)\otimes \mathrm{Stb}^*(P)))\in \Gtheory_0(\ast)=\mathbb{Z}$$

\noindent where $p$ is the projection to the point. Alternatively, these are encoded by maps

\begin{equation}
\mathrm{I}_{0,n, \beta}:\Ktheory_0(X)^{\otimes_n}\to \Ktheory_0(\Mnstable)
\end{equation}

\noindent induced from the lax action, i.e., via pullback-pushfoward along (\ref{maincorrespondence}).\\

Let us now explain how the lax associative structure produces the metric terms introduced by Givental-Lee to explain the $\Ktheory$-theoretic splitting principle. We ask the reader to recall the diagram )(\ref{laxassociativestructurediagram}). Let $\beta_0=\beta_1+\beta_2$ and let $n,m> 2$. We have two different ways to go from $\Ktheory_0(X^{n-1}\times X^{m-1})$ to $\Ktheory_0(\Mnstable\times \overline{\mathcal{M}}_{0,m})$: either we use the space of stable maps $\R\overline{\mathcal{M}}_{0,n+m-2}(X, \beta_0)$ and the pullback diagram $\mu$ or, we use the fiber product $\R\overline{\mathcal{M}}_{0,n}(X, \beta_1)\times_X \R\overline{\mathcal{M}}_{0,m}(X, \beta_2)$ and the pullback diagram $\nu$. The lax structure measures the difference between the two. Our goal for the rest of this section is to give a more accurate description of this difference.\\

We start with some general preliminary remarks. The first observation concerns the derived stack  $\mathrm{P}(X):=\coprod_\beta \mathbb{R}\overline{\mathcal{M}}_{0,2}(X,\beta)$. This stack has natural structure of monoid-object over $X$ given by the gluing to stable maps

\begin{equation}
 \mathrm{P}(X)\times_X  \mathrm{P}(X)\to  \mathrm{P}(X)
\end{equation}

This monoid structure can be obtained in a strict way as it exists already at the level of the moduli 1-stacks of Costello $\coprod_\beta \mathfrak{M}_{0,2,\beta}$. By definition, this stack classifies unparametrized rational paths on $X$ and the monoid operation corresponds to the concatenation. We then observe that for every $n\geq 2$ the stacks of stable maps $\coprod_\beta \stablemaps$ are modules over  $\mathrm{P}(X)$ (simultaneously on the left and on the right) via the gluing of stable maps along the last or first marked point. In this case, given $n,m\geq 2$ we can have a semi-simplical object in derived stacks given by the associated bar complex

\begin{equation}
\resizebox{1 \hsize}{!}{
\xymatrix{
 \ar@<1.1ex>[r] \ar@<-1.1ex>[r] \ar[r] &(\coprod_\beta \R\overline{\mathcal{M}}_{0,n}(X,\beta)\times_X \mathrm{P}(X)  \times_X  (\coprod_\beta \R\overline{\mathcal{M}}_{0,m}(X,\beta))  \ar@<1ex>[r] \ar@<-1ex>[r] & (\coprod_\beta \R\overline{\mathcal{M}}_{0,n}(X,\beta)\times_X (\coprod_\beta \R\overline{\mathcal{M}}_{0,m}(X,\beta))
}
}
\end{equation}

\noindent where the face maps are the gluing morphisms.\\

This semi-simplicial object is naturally augmented by the gluing map

\begin{equation}
\resizebox{1 \hsize}{!}{
\xymatrix{
 \ar@<1ex>[r] \ar@<-1ex>[r] & (\coprod_\beta \R\overline{\mathcal{M}}_{0,n}(X,\beta)\times_X (\coprod_\beta \R\overline{\mathcal{M}}_{0,m}(X,\beta))\ar[r]& (\coprod_\beta \R\overline{\mathcal{M}}_{0,n+m-2}(X,\beta))\times_{\overline{\mathcal{M}}_{0,n+m-2}}(\Mnstable\times \overline{\mathcal{M}}_{0,m})
}
}
\end{equation}

Let $\beta_0\in \NE(X)$. By base-change along the open inclusion 

 $$\R\overline{\mathcal{M}}_{0,n+m-2}(X, \beta_0)\subseteq \coprod_\beta \R\overline{\mathcal{M}}_{0,n+m-2}(X,\beta)$$
 
\noindent we obtain another simplicial object $\mathcal{U}(\beta_0)_\bullet$ informally described by

\begin{equation}
\resizebox{1 \hsize}{!}{
\xymatrix{
 \ar@<1.1ex>[r] \ar@<-1.1ex>[r] \ar[r] & \coprod_{\beta_{0}= d_1 + d_2 + d_3} \R\overline{\mathcal{M}}_{0,n}(X,d_1)\times_X  \R\overline{\mathcal{M}}_{0,2}(X,d_2) \times_X \R\overline{\mathcal{M}}_{0,m}(X,d_3)  \ar@<1ex>[r] \ar@<-1ex>[r] & \coprod_{\beta_{0}=d_1+ d_2} \R\overline{\mathcal{M}}_{0,n}(X, d_1)\times_X \R\overline{\mathcal{M}}_{0,m}(X,d_2)
}
}
\end{equation}

\noindent with an augmentation $\mathcal{U}(\beta_0)_\bullet\to \mathcal{U}(\beta_0)_{-1}$ given the gluing map

\begin{equation}
\label{augmentedgivenbeta}
\resizebox{1 \hsize}{!}{
\xymatrix{
 \ar@<1ex>[r] \ar@<-1ex>[r] & \coprod_{\beta_{0}=d_1+ d_2} \R\overline{\mathcal{M}}_{0,n}(X, d_1)\times_X \R\overline{\mathcal{M}}_{0,m}(X,d_2)\ar[r]& \R\overline{\mathcal{M}}_{0,n+m-2}(X,\beta_0)\times_{\overline{\mathcal{M}}_{0,n+m-2}}(\Mnstable\times \overline{\mathcal{M}}_{0,m})
}
}
\end{equation}

More generally, the level $[k]$ of $\mathcal{U}(\beta_0)_\bullet$ is given by the derived stack

\begin{equation}
\coprod_{\beta_{0}= d_0 + ... + d_{k+1}} Z_{d_0,...,d_{k+1}}
\end{equation}

\noindent where we define 

\begin{equation}
Z_{d_0,...,d_{k+1}}:=\R\overline{\mathcal{M}}_{0,n}(X,d_0)\times_X  \R\overline{\mathcal{M}}_{0,2}(X,d_1) \times_X ....\times_X \R\overline{\mathcal{M}}_{0,2}(X,d_{k}) \times_X \R\overline{\mathcal{M}}_{0,m}(X,d_{k+1})
\end{equation}

Again the face maps are the gluing maps. \\

For each $k\geq 0$ and for each partition of $\beta_0=d_0+...+ d_{k+1}$ we will let $f_{(d_0,..., d_{k+1})}$ denote the composition $ Z_{d_0,.., d_{k+1}}\subseteq\mathcal{U}(\beta_0)_k \to \mathcal{U}(\beta_0)_{-1}$. \\

\medskip

Now we list some important facts that we will need in our discussion:

\begin{enumerate}
\item The composition maps $\phi: \Mnstable\times \overline{\mathcal{M}}_{0.m}\to  \overline{\mathcal{M}}_{0.n+m-2}$ are closed immersions \cite[Cor 3.9]{Knudsen-moduli-stable-curves-II-1983}. By pullback, so is the map $\mathcal{U}(\beta_0)_{-1}\to  \R\overline{\mathcal{M}}_{0,n+m-2}(X,\beta_0)$.
\item The derived stacks of stable maps $\R\overline{\mathcal{M}}_{0,n}(X, \beta)$ are quasi-smooth. As $X$ is smooth, each stack  $Z_{d_0,...,d_{k+1}}$ will also be quasi-smooth (quasi-smooth maps are stable under pullback). Moreover, it is also known that the gluing maps  $\stablemaps \times_X \R\overline{\mathcal{M}}_{0,m}(X,\beta')\to \R\overline{\mathcal{M}}_{0,n+m-2}(X,\beta+\beta')$ are closed immersions. \personal{\label{commentquasi-smooth}this is not really well-known. In any case this is not really needed! it is enough to use h-descent.}.  

\item By the same pullback argument, the derived stack $\mathcal{U}(\beta_0)_{-1}$ is also quasi-smooth.\\

\end{enumerate}

Notice now that the augmented semi-simplicial object (\ref{augmentedgivenbeta}) lives in the full subcategory of derived Deligne-Mumford stacks and due to the finite properties of the Mori cone, for each $\beta_0$ it becomes constant after a certain level $k\geq 0$ (in fact empty, as the moduli of stable maps with two marked points and degree 0 is empty), in which case the diagram is in fact finite. 

We claim that in fact for each $\beta_0$, this augmented semi-simplicial diagram has a colimit inside derived Deligne-Mumford stacks. Not all colimits exist inside Deligne-Mumford stacks, however, and thanks to \cite[Thm 6.1]{lurie-dagIX}\footnote{Here one has to show that the theory of derived DM-stacks viewed via the functor of points approach  is equivalent to the theory of DM stacks in the sense of \cite{lurie-structuredspaces} using ringed $\infty$-topoi. This follows from \cite[Thm 2.4.1]{lurie-structuredspaces} together with the Representability Theorem \cite[2]{lurie-DAGXIV} together with the fact that a map of simplicial commutative rings $f:A\to B$ is \'etale if and only if the map $\mathrm{Spec}^{\text{\'et}}(B)\to \mathrm{Spec}^{\text{\'et}}(A)$ is \'etale . We thank Mauro Porta for explaining to us a detailed proof of this result, without using the Representability Theorem.} we know the existence of pushouts of derived DM-stacks along closed immersions. \\

We now remark that the stacks $Z_{(d_0,...,d_{k+1})}$ can be organized in a finite diagram where all maps are closed immersions and whose colimit is the same as the realization of the simplicial object $\mathcal{U}(\beta_0)_\bullet$. Let $\Delta_s$ denote the non-full subcategory of $\Delta$ with the same objects but only the injective maps as morphisms. We know from \cite[Lemma 6.5.3.7]{lurie-htt} that the inclusion $\Delta_s^{op}\subseteq \Delta^{op}$ is cofinal. In other words, to compute the realization of $\mathcal{U}(\beta_0)_\bullet$ we only need to care about face maps. To describe the diagram with the $Z_{(d_0,...,d_{k+1})}$ we introduce an auxiliary 1-category $\Lambda_{\beta_0}$. Its objects are pairs $([k], S_k)$ where $[k]$ is an object in $\Delta$ and $S_k$ is a choice of a decomposition $\beta_0=d_0+...+ d_{k+1}$. The morphisms $([k], S_k)\to ([k'], S_{k'})$ are given as in $\Delta_S$ by specifying the generating face maps. A face map  $([k], S_k=\{d_0,..., d_{k+1}\})\to ([k+1], S_{k+1}=\{\alpha_0,..., \alpha_{k+2}\})$ is by definition, the data of a face map $\partial_i:[k]\to [k+1]$ in $\Delta_s$ together with the condition that

  \[
 \left\{
          \begin{array}{ll}
                  d_j=\alpha_j &\text{ if } j< i \\
                  d_j=\alpha_j + \alpha_{j+1} & \text{ if } j= i\\
                  d_j= \alpha_{j+1} & \text{ if } j\geq i+1
                \end{array}
              \right.
  \]

The collection of the $Z_{(d_0,...,d_{k+1})}$ and closed immersions between, appears in the form of a  $\Lambda_{\beta_0}^{op}$-diagram $\Psi_{\beta_0}$ in the $(\infty,1)$-category of derived Deligne-Mumford stacks, together with a cone  $(\Lambda_{\beta_0}^{op})^{\triangleright}$ with vertex $\mathcal{U}(\beta_0)_{-1}$. One can construct this diagram by first constructing a similar diagram by hand in the category of usual 1-stacks using the moduli spaces of Costello. 

We  have a forgetful functor $t:\Lambda_{\beta_0}\to \Delta_s$. It follows from the definitions that the fibers of $t$ are discrete and it is an easy exercise to check that $t$ is a right fibration. Its opposite $\Lambda^{op}_{\beta_0}\to \Delta_s^{op}$ is a left fibration with discrete fibers and it follows that for every $[k]\in \Delta_s^{op}$, the canonical inclusion $t^{-1}([k])\subseteq (\Lambda^{op}_{\beta_0})_{/[k]}$ is cofinal by Quillen's Theorem A for $\infty$-categories \cite[Thm. 4.1.3.1]{lurie-htt}. In this case, we conclude that the left Kan extension of $\Psi_{\beta_0}$ along the forgetful map  $t:\Lambda_{\beta_0}^{op}\to \Delta^{op}$ is precisely the simplicial object $\mathcal{U}(\beta_0)_{\bullet}$. Moreover, as colimits are given as left Kan extensions along the projections to the constant diagram, the colimit of $\Psi_{\beta_0}$ is canonically equivalent to the colimit of $\mathcal{U}(\beta_0)_{\bullet}$. Moreover, one can freely add degeneracies to $\mathcal{U}(\beta_0)_{\bullet}$ by considering its Kan extension along the inclusion $\Delta_s^\mathrm{op}\subseteq \Delta^\mathrm{op}$. As this inclusion is cofinal, the colimit of this new diagram is the same. With the necessary care, we will use the same notation for this new simplicial object.





It is clear that the colimit of this diagram can be computed in a finite number of steps using only pushout diagrams. In this case, by the discussion above, there exists a colimit internal to the theory of derived Deligne-Mumford stacks $\mathrm{D}_{\beta_0}:=\mathrm{colim}^{\mathrm{DM-stk}}\mathcal{U}(\beta_0)_\bullet$.\\

\medskip

The key to understand the combinatorics of Gromov-Witten invariants is hidden in the canonical colimit map

\begin{equation}
\label{executeorder66}
\resizebox{1 \hsize}{!}{
\xymatrix{
f:\mathrm{D}_{\beta_0}:=\mathrm{colimit}^{\mathrm{DM-stk}}_{[k+1]\in \Delta}\, \coprod_{d_0+d_1+...+ d_{k+1}=\beta_0} \mathcal{Z}_{d_0,..., d_{k+1}} \ar[r]& \mathcal{U}(\beta_0)_{-1}:=\mathbb{R}\mathcal{M}_{0,n+m-2}(X,\beta_0)\times_{\overline{\mathcal{M}}_{0,n+m-2}}(\Mnstable\times \overline{\mathcal{M}}_{0,m})
}
}
\end{equation}

We believe that this map is an equivalence of derived Deligne-Mumford stacks. We will not prove this. Instead, we will restrict ourselves to show the properties required by Gromov-Witten theory. In fact, for cohomological invariants what matters is that this map is birational and for the K-theoretic invariants, the important point is that the $f_*(\mathcal{O}_{\mathrm{D}_{\beta_0}})\simeq \mathcal{O}_{\mathcal{U}(\beta_0)_{-1}}$. The first fact follows from a standard easy argument - the map $f$ induces an isomorphism when restricted to the open substacks of stable maps whose underlying curves are already stable.  This is an easy exercice which we will leave to the reader. The second statement concerning the structure sheaves is more involved: the proof we give below will use the results on derived h-cech-descent for almost coherent and perfect complexes of \cite{1402.3204} together with the fact that our augmented semi-simplicial diagram $\mathcal{U}(\beta_0)_\bullet$ is bounded and provides an h-hypercover (in the bounded situation, descent and hyperdescent are known to be equivalent). This derived h-descent for pseudo-coherent and perfect complexes is an important feature of derived algebraic geometry, which is not true in the classical setting.

\medskip

\vspace{1cm}

Recall that a (derived) h-cover, following \cite[Section 4.2]{1402.3204}, is a morphism which is represented by a relative algebraic space \todo{Why is this true in our case?} and which is a universal topological submersion. Proper surjections are the example we are interested in. 

One is interested in the semi-simplicial augmented object $\mathcal{U}(\beta_0)_\bullet$. Notice first that it only has finite many levels. This follows from the decomposition assumptions on $\beta_0$. Starting from a certain level $\lambda\geq 1$ all the terms are empty, as the moduli spaces of stable maps with two marked points and degree 0 are empty \footnote{Here of course we use the real moduli with 2 points and now the fake ones introduced before}. We can therefore replace $\mathcal{U}(\beta_0)_\bullet$ by the $\lambda$-coskeletal simplicial diagram $\mathrm{cosk}_\lambda(\mathcal{U}(\beta_0)_\bullet)$ whose colimit inside the theory of derived DM-stacks is again $\mathrm{D}_{\beta_0}$

One could try to show that the semi-simplicial augmented object $\mathcal{U}(\beta_0)_\bullet$ is a  bounded h-hypercover. To check this one would have to check that for each $k\leq \lambda$ the canonical map

\begin{equation}
\label{h-cover}
\resizebox{1 \hsize}{!}{
\xymatrix{
\coprod_{d_0+...+ d_{k+1}=\beta_0}\R\overline{\mathcal{M}}_{0,n}(X,d_0)\times_X  \R\overline{\mathcal{M}}_{0,2}(X,d_1) \times_X ....\times_X \R\overline{\mathcal{M}}_{0,2}(X,d_{k}) \times_X \R\overline{\mathcal{M}}_{0,m}(X,d_{k+1})\ar[r]& \mathrm{cosk}_{k-1}(\mathcal{U}(\beta_0)_\bullet))_{k}
}
}
\end{equation}

\noindent is an h-cover. At the first level we have

\begin{equation}
\label{h-cover2}
\xymatrix{
\coprod_{d_0 + d_{1}=\beta_0}\R\overline{\mathcal{M}}_{0,n}(X,d_0)\times_X  \R\overline{\mathcal{M}}_{0,m}(X,d_{1})\ar[r]& \mathcal{U}(\beta_0)_\bullet)_{-1}
}
\end{equation}

\noindent which in fact is an h-cover as clearly is proper (as both the source and target are proper - see \cite[2.1.2.10, 2.1.2.13]{lurie-spa}), and surjective. 

However, starting from the second level these maps are not surjective as it is easy to see - in $\mathcal{U}(\beta_0)_0\times_{\mathcal{U}(\beta_0)_{-1}} \mathcal{U}(\beta_0)_{0}$ we can only access the first factor by gluing first and the last factor can only be accessed by gluing last, of vice-versa. As it is easily understood, one needs to have both options to have surjectivity. In order words, we have a lack of symmetry originated from the fact that the simplicial presentation of $\mathcal{U}(\beta_0)_\bullet$  required us to fix an order for the different ways of gluing. In order to make $\mathcal{U}(\beta_0)_\bullet$ an hypercover we need to consider all the possible orders. One possible way to achieve this is to consider the \emph{symmetrization} of $\mathcal{U}(\beta_0)_\bullet$. Namely, recall that the simplicial category $\Delta$ can be seen as a non-full subcategory of the category $\mathrm{Fin}$ of non-empty finite sets (also known in the literature as \emph{symmetric simplicial}) where $[n]\in \Delta$ is now seen as an unordered set with $n+1$ elements. Let $\Delta^\mathrm{op}\to \mathrm{Fin}^{\mathrm{op}}$ be the canonical inclusion.  For any category $\C$ having colimits we have a composition functor $\mathrm{Fun}(\mathrm{Fin}^{\mathrm{op}}, \C) \to \mathrm{Fun}(\Delta^{\mathrm{op}}, \C)$ which admits a left adjoint given by taking the left Kan extension along the canonical inclusion. The symmetrization of a simplicial object $U_\bullet$ is the new simplicial object defined by the restriction of this Kan extension, which we will denote as $U_\bullet^\Sigma$. Informally, it is easy to see that the level $n$ of this new simplicial object will be a disjoint union $\coprod_{\sigma\in \Sigma_n+1}U_{[n]}$ and the maps will be induced by all the possible orderings of the boundary maps.

To conclude, instead of $\mathcal{U}(\beta_0)$ we will consider its symmetrization $\mathcal{U}(\beta_0)^\Sigma$, which one can easily now check is an h-hypercover \todo{Should we add more details than this?}. To conclude we combine \cite[8.3.6]{MR2294028} with Quillen's Theorem A \cite[4.1.3.1]{lurie-htt} to deduce that the inclusion $\Delta^\mathrm{op}\to \mathrm{Fin}^{\mathrm{op}}$ is cofinal (see also \cite[4.2.20, 4.2.19, 2.2.7, 3.3.3 and 3.1.1]{MR2294028} and notice the appearance of an opposite relating these to the result in \cite[4.1.3.1]{lurie-htt}). This implies that the unit of the adjunction $\mathcal{U}(\beta_0)_\bullet\to \mathcal{U}(\beta_0)^\Sigma_\bullet$ produces an equivalence after taking colimits.

\begin{prop}
\label{prop-splitting}
We have the following:
\begin{enumerate}
\item The canonical map $\mathcal{O}_{\mathcal{U}(\beta_0)_{-1}}\to f_*(\mathcal{O}_{\mathrm{D}_{\beta_0}})$ is an equivalence. 

\item There is an equality of $\Gtheory$-classes

\begin{equation}
\label{llllllllllkjij}
[\mathcal{O}_{\mathcal{U}(\beta_0)_{-1}}]=[f_*(\mathrm{D}_{\beta_0})]=\sum_k (-1)^k \sum_{\beta_0=d_0+....+ d_{k+1}} [f_{(d_0,...,d_{k+1})_*}(\mathcal{O}_{Z_{(d_0,..., d_{k+1})}})]
\end{equation}
 
\end{enumerate}
\end{prop}

\begin{remark}
\label{remark-comparsionLeemetric}
The Proposition \ref{prop-splitting} and more importantly, the formula (\ref{llllllllllkjij}) is analogous to the key result of \cite[Prop. 11 Section 3.7]{MR2040281} which explains the correction of the splitting principle necessary to handle K-theoretic invariants. Our computations exhibit this formula as a consequence of derived h-descent.
\end{remark}

\begin{proof}[Proof of Proposition \ref{prop-splitting}]Let us start by showing (1). The crucial result that we will be using is \cite[Theorem 4.12]{1402.3204}, namely, that perfect complexes satisfy descent with respect to Cech h-covers. This, combined with the fact that descent for Cech covers implies descent for all n-coskeletal hypercovers (see \cite[6.5.3.9]{lurie-htt} or \cite[Appendix 1]{duggersharon}), tells us that the pullback-pushfoward maps produces an equivalence

$$
\mathrm{Perf}(\mathcal{U}(\beta_0)_{-1})\simeq \, \mathrm{lim}_{[k]\in \Delta^{\mathrm{op}}} \,  \bigoplus_{\beta_0=d_0+...+ d_{k+1}}\mathrm{Perf}( Z_{d_0,...,d_{k+1}})
$$

At the same time, and using \cite[Thm 7.1]{lurie-dagIX} we deduce that the canonical map induced by the pullback functors

\begin{equation}
\xymatrix{
\Qcoh(\mathrm{D}_{\beta_0})\ar[r]& \mathrm{lim}_{[k]\in \Delta^{op}} \bigoplus_{\beta_0=d_0+...+ d_{k+1}}\Qcoh( Z_{d_0,...,d_{k+1}})
}
\end{equation}

\noindent is fully faithful and the unit of the associated adjunction gives an a natural equivalence

 \begin{equation}
f_*\simeq \mathrm{lim}_{[k]\in \Delta^{op}} (\bigoplus_{\beta_0=d_0+...+d_{k+1}} f_{(d_0,...,d_{k+1})_*}f_{(d_0,...,d_{k+1})}^*)
\end{equation}

The two results combined imply the equivalence $f_*(\mathcal{O}_{\mathrm{D}_{\beta_0}})\simeq \mathcal{O}_{\mathcal{U}(\beta_0)_{-1}}$. 

\medskip 
This concludes the proof of (1). (2) is now a consequence of (1) via a standard computation in a stable $\infty$-category (as the simplicial diagram object is finite). 
\end{proof}
\medskip

We now turn to the description of the lax structure in the K-theory action.  Considering the pullback square $\mu$ for a particular grading $\beta_0$ we obtain a pullback square $\mu_0$ that we can fit as

\begin{equation}
\label{laxassociativestructurediagram22}
\xymatrix{
\Mnstable\times \overline{\mathcal{M}}_{0,m}\ar@{^{(}->}[d]^{\phi}& \ar@{}[dl] |{\text{\LARGE{$\llcorner_{\mu_0}$}}} \ar[l]^g\ar@{^{(}->}[d]^i \mathcal{U}(\beta_0)_{-1} &\\
\overline{\mathcal{M}}_{0,n+m-2}& \R\overline{\mathcal{M}}_{0,n+m-2}(X,\beta_0)\ar[l]^-{\mathrm{Stb}}\ar[r]^{ev}& X^{n-1}\times * \times X^{m-1}
}
\end{equation}

The map $\Ktheory_0(X^{n-1}\times X^{m-1})\to \Ktheory_0(\Mnstable\times  \overline{\mathcal{M}}_{0,m})$ that we are interested in, is given by the composition 

\begin{equation}
\label{serie0}
\phi^*\mathrm{Stb}_*ev^*
\end{equation}

Using base-change for derived Deligne-Mumford stacks applied to the diagram $\mu_0$, this is equivalent to 

\begin{equation}
\label{serie1}
g_*i^*ev^*
\end{equation}

But now we know that in $\Gtheory_0(\mathcal{U}(\beta_0)_{-1})$ the structure sheaf can be written as an alternated sum (\ref{llllllllllkjij}), and (\ref{serie1}) becomes

\begin{equation}
\label{serie2}
\sum_k (-1)^k \sum_{\beta_0=d_0+...+ d_{k+1}} g_*f_{(d_0,..., d_k)_*}f_{(d_0,..., d_{k+1})}^*i^*ev^*
\end{equation}

Let now $V_{d_0,..., d_{k+1}}$ denote the stack $\R\overline{\mathcal{M}}_{0,n}(X, d_0)\times \R\overline{\mathcal{M}}_{0,2}(X, d_1)\times ... \times \R\overline{\mathcal{M}}_{0,2}(X, d_{k})  \times  \R\overline{\mathcal{M}}_{0,m}(X, d_{k+1})$. We have pullback diagrams that fit in

\begin{equation}
\label{laxassociativestructurediagram21}
\xymatrix{
\Mnstable \times \overline{\mathcal{M}}_{0,m}&&\\
\ar[u]^{\mathrm{Stb}\times \mathrm{Stb}} V_{d_0,..., d_{k+1}}\ar[rr]^-{ev_{(d_0,...,d_{k+1})}}&& X^{n-1}\times X\times  \underbrace{X^2 \times \times...\times X^2}_{k} \times X \times X^{m-1}\\
\ar @{} [urr] |{\text{\LARGE{$\urcorner$}}}\ar[u]^{q_{(d_0,...,d_{k+1})}}\ar[rr]_{h_{(d_0,.., d_{k+1})}} Z_{d_0,..., d_{k+1}}&& X^{n-1} \times \underbrace{X\times...\times X}_{k} \times X^{m-1}\ar[u]_-{\psi_k:=id^{n-1}\times \underbrace{\Delta \times ... \times \Delta}_{k} \times id^{m-1}} \ar[d]^-{p_k:=id^{n-1}\times pt \times id^{m-1}}\\
&&X^{n-1}\times * \times X^{m-1}
}
\end{equation}

\noindent for each decomposition $\beta_0=d_0+...+ d_{k+1}$.  We now notice that the composition $(\mathrm{Stb}\times \mathrm{Stb})\circ q_{(d_0,...,d_{k+1})}$ is equivalent to $g\circ f_{(d_0,...,d_{k+1})}$ and using the base-change formula for the pullback diagram (\ref{laxassociativestructurediagram21}), (\ref{serie2}) becomes

\begin{equation}
\label{serie3}
\sum_k (-1)^k \sum_{\beta_0=d_0+...+ d_{k+1}} (\mathrm{Stb}\times \mathrm{Stb})_* ev_{(d_0,..., d_{k+1})}^* (\psi_k)_* p_k^*
\end{equation}

\noindent which we can write as 

\begin{equation}
\label{serie3}
(\mathrm{Stb}\times \mathrm{Stb})_* ev_{(\beta_1, \beta_2)}^* (\psi_1)_* p_1^* + \text{extra terms}
\end{equation}

Of course, the first term in the last formula corresponds to the second map $\Ktheory_0(X^{n-1}\times X^{m-1})\to \Ktheory_0(\Mnstable\times \overline{\mathcal{M}}_{0,m})$, obtained by using $\nu_0$. The extra terms (corresponding to $k\geq 1$), which we obtain as a result of derived descent, appear exactly as in \cite[Prop. 11 Section 3.7]{MR2040281}. In order to encode and manage these extra terms Lee and Givental introduced in \cite[Section 4]{MR2040281} and in \cite[p.6]{MR1786492} a combinatorial gadget - which they called a metric. Our results provides a derived computation of these terms in terms of the structure sheaves of the derived stacks $\R\overline{\mathcal{M}}_{0,2}(X,d)$. See the Remark \ref{remark-comparsionLeemetric}.

\subsection{Quantum $\Ktheory$-theory: comparison with the $\Ktheory$-theoretic invariants of Givental-Lee}
\label{section-comparisonourswithLee}
In \cite{MR2040281, MR1786492}, Givental and Lee introduced Gromov-Witten invariants living in $\mathrm{G}$-theory. The basic ingredient to define these invariants is the so-called \emph{virtual structure sheaf} which is an element $\virtualsheaf$ in the Grothendieck group $\Gtheory_0$ of the truncation $t_0(\stablemaps)$. Let $j:  t_0(\stablemaps)\to \stablemaps$ denote the inclusion. Our goal in this section is to explain that the virtual structure sheaf $\virtualsheaf$ constructed by Lee in \cite[Section 2.3]{MR2040281} can be identified with the restriction of the structure sheaf of $\stablemaps$ to the truncation. This in particular implies that their GW classes are the same as the ones obtained from our lax action studied in the previous section. 

%

Let us start with some well-known general preliminaries. Let $F$ be a derived Artin stack and $j:t_0(F)\to F$ its truncation. Then the structure sheaf $\mathcal{O}_F$ produces a family of sheaves on the truncation $\pi_i(\mathcal{O}_F)$.  In general these sheaves are not coherent but one can show that when the base field is of characteristic zero and $F$ is of finite presentation and quasi-smooth then these sheaves are coherent and vanish for $i>> 0$ - see \cite[SubLemma 2.3]{properlocal}.  Under these hypothesis one can show that the map $j_*$ sends coherent complexes to coherent complexes and by devissage induces an isomorphism $j_*:\Gtheory_0(t_0(F))\simeq\Gtheory_0(F)$. Its inverse $(j_*)^{-1}$ sends $\mathcal{O}_F \in \Gtheory_0(F)$ to  $\Sigma_i (-1)^i \pi_i(\mathcal{O}_F)$ which is a finite sum under these hypothesis. Recall also that the Euler characteristics $\chi$ are defined by taking pushfoward to the point

\begin{equation}
\xymatrix{
t_0(F)\ar[dr]^q\ar[r]^j & F\ar[d]^{p}\\
 & \ast
}
\end{equation}

\noindent so that in $\Gtheory$-theory, we have $q_*((j_*)^{-1}(E))\simeq p_*(j_*(j_*)^{-1}(E))=p_*(E)$ for any $E\in \Gtheory_0(F)$.\\

Our main goal in this section is to explain that the virtual structure sheaf $\virtualsheaf\in \Ktheory_0(t_0(F))$ of \cite[Section 2.3]{MR2040281} is given by $(j_*)^{-1}(\mathcal{O}_F)$. In order to explain this we will need some further preliminaries concerning perfect obstruction theories in the sense of \cite[Def. 5.1]{MR1437495}. These were introduced as an ad-hoc way to keep track of derived enhancements of classical stacks. Let $Y$ be an underived stack and $\mathbb{L}_{Y}\in \Qcoh(Y)$ its cotangent complex. The data of an obstruction theory on $Y$ consists of a map of quasi-coherent sheaves $t:T\to \mathbb{L}_Y$ satisfying some conditions which we will allow ourselves to omit here. Informally, $T$ is to be understood as the cotangent complex of a derived enhancement of $Y$. It is said to be \emph{perfect} if $T$ is a perfect complex. Suppose now that there exists a derived stack $F$ whose truncation $t_0(F)$ is $Y$. This produces a natural associated obstruction theory on $Y$: let $\mathbb{L}_F\in \Qcoh(F)$ denote the cotangent complex of $F$. Then we have a natural map $T:=j^*(\mathbb{L}_F)\to \mathbb{L}_{t_0(F)}$. Following \cite[Cor. 1.3]{2011-Schur-Toen-Vezzosi}, if $F$ is general geometric stack then this map defines an obstruction theory (this follows from Lurie's connectivity estimates) and if $F$ is in particular quasi-smooth then this is a $[-1,0]$-\emph{perfect obstruction theory}, meaning that $T$ is perfect and concentrated Tor-amplitude $-1$ and $0$. One can also ask if this assignment is essentially surjective: this is not true and we can identify the obstructions to produce a lifting \cite{timo-obstruction}.\\

Every perfect obstruction theory on $Y$ produces a virtual sheaf, namely an object in $\Ktheory(Y)$. This a consequence of a more structured fact: we will show that to every obstruction theory $t$ on $Y$ one can naturally associate a derived enhancement of $Y$,  $\R\mathrm{Obs}(t)$ that splits. By definition, the virtual structure sheaf associated to $t$ is given by the recipe described above, applied to the truncation map $i:Y\subseteq \R\mathrm{Obs}(t)$, meaning $\virtualsheaf(t):=(i_*)^{-1}(\mathcal{O}_{ \R\mathrm{Obs}(t)})$. In general, $i^*\mathbb{L}_{\R\mathrm{Obs}(t)}$ will be different from $T$, as, due to the splitting, it will  be of the form  $\mathbb{L}_Y\oplus \mathbb{L}_i[-1]$.

The construction of the derived enhancement $\R\mathrm{Obs}(t)$ follows from the observation that the construction of the virtual fundamental classes described in \cite{MR1437495} has a natural interpretation as homotopy fiber products in derived algebraic geometry. Indeed, let $Y$ be classical Artin stack together with the data of a perfect obstruction theory $t:T\to \mathbb{L}_Y$. Then let $\mathfrak{C}(T)$ be the (classical) cone stack associated to $T$ \cite[Section 2]{MR1437495}\footnote{Denoted there as $h^1/h^0(T^{\vee})$.} and let $\mathfrak{C}_Y$ be the intrinsic normal cone of $Y$. Then both $\mathfrak{C}(T)$ and $\mathfrak{C}_Y$ are cone stacks over $Y$, $\mathfrak{C}_Y$ is a closed sub-stack of $\mathfrak{C}(T)$ and $Y$ can be embedded in both of them via the zero section \cite[Prop 2.4, 2,6 and Def. 3.10]{MR1437495}. We now use the inclusion of classical 1-stacks inside derived stacks and see these three stacks as derived objects in a trivial way. We define a new derived stack $\R\mathrm{Obs}(t)$ to the pullback in the $(\infty,1)$-category of derived stacks

\begin{equation}
\xymatrix{
\ar[d]^{r}\ar[r]\R\mathrm{Obs}(t)&\ar@{^{(}->}[d]  \mathfrak{C}_Y \\
Y \ar@{^{(}->}[r]&\mathfrak{C}(T)
}
\end{equation}

\noindent where the map $Y\subseteq \mathfrak{C}(T)$ is the zero section and  $\mathfrak{C}_Y\subseteq \mathfrak{C}(T)$ is the closed immersion produced by the obstruction theory. In general the inclusion of classical stacks in derived stacks does not commute with homotopy fiber products. In fact, in this case, the usual  fiber product in classical stacks is equivalent to $Y$ (as $Y$ can be embedded both in $\mathfrak{C}(T)$ and  $\mathfrak{C}_Y$ via the zero section). The truncation functor however commutes with products, and therefore we deduce that $t_0( \R\mathrm{Obs}(t))=Y$, or in other words $\R\mathrm{Obs}(t)$ is a derived enhancement of $Y$. It is also clear from the definition of derived fiber products that the structure sheaf of this derived stack is responsible for the virtual structure sheaf described in \cite[Remark 2.4]{MR1437495}. This derived enhancement of $Y$, with truncation map $i:Y\subseteq \R\mathrm{Obs}(t)$, has a particular feature - it splits via the map $r$. \personal{ otice that as the inclusion of the intrinsic normal cone (vertical right arrow in the diagram above) is claimed to be a closed immersion if and only if the obstruction theory is perfect Prop.2.6 B-F, we find that by stability under pullbacks $r$ is also a closed immersion of derived schemes. Notice that this does not implie the Obstruction derived stack to be classical as the example $k\to k[\epsilon]\to k$ with $ \epsilon$  in degree -1 shows. It follows that in $\Gtheory$-theory, the truncation map $i:Y\subseteq \R\mathrm{Obs}(t)$ verifies $(i_*)^{-1}=r_*$ as $r_*\circ i_*=Id$ implies $r_*\circ i_*\circ (i_*)^{-1}=(i_*)^{-1}$ and $r_*$ being a closed immersion preserves bounded pseudo-coherent (I was confused about this because of the example $k[u]\to k$ with $u$ in degree -2. The point is that $k[u]$ is not eventually coconective so its structure sheaf is not coherent. No contradiction.). In this case, the virtual structure sheaf associated to the obstruction theory is $\virtualsheaf(t):=r_*\mathcal{O}_{ \R\mathrm{Obs}(t)}$.}.\\

 The virtual structure sheaf  $\virtualsheaf\in \Gtheory_0(t_0(F))$ of \cite[Section 2.3]{MR2040281} is defined by the recipe given in the previous paragraph using the following (relative) obstruction theory as input: Let $\Mnprestable$ denote the stack of all pre-stable curves of genus zero with $n$ marked points and let $\mathcal{C}_{0.n, \beta}\to \Mnprestable$ denote the universal pre-stable curve of total degree $\beta$. Then we have a commutative diagram

\begin{equation}
\label{obstructionad-hoc}
\xymatrix{
\mathcal{C}_{0.n, \beta}\times_{\Mnprestable}t_0(\stablemaps)\ar[r]^-{ev_0}\ar[d]^{\pi_0}& X\times \Mnprestable\ar[d]^u\\
t_0(\stablemaps)\ar[r]^{q_0}& \Mnprestable
}
\end{equation}

\noindent where $ev_0$ is the evaluation map and $\pi_0$ is the projection to the second factor. Notice that in this case the relative cotangent complex $\mathbb{L}_u$ is equivalent to $\mathbb{L}_X$. Following the steps in the discussion preceding \cite[Prop 6.2]{MR1437495}, we find a natural map $t:((\pi_0)_* (ev_0)^* \mathbb{T}_X)^{\vee} \to \mathbb{L}_{q_0}$ in $\Qcoh(Y)$ with $Y=t_0(\stablemaps)$. By \cite[Prop 6.2]{MR1437495} this is a perfect obstruction theory. The virtual structure sheaf considered by Lee in \cite[Section 2.3]{MR2040281} can be immediately identified with the element $\virtualsheaf(t)$ described above, induced by the derived stack $\R\mathrm{Obs}(t)$.\\

\begin{prop}
The obstruction theory used by Lee is the same as the obstruction theory produced by the derived enrichment $F:=\stablemaps$ of $Y:=t_0(\stablemaps)$
\begin{proof}
This follows essentially from the description of the tangent complex of a mapping stack $\rhom(U,V)$ when $U$ and $V$ are derived Artin stacks. In this case one can exhibit a canonical equivalence

\begin{equation}
\label{eq-tangentmappingstack}
\mathbb{T}_{\rhom(U,V)}\simeq \pi_* ev^* (\mathbb{T}_V)
\end{equation}

\noindent in $\Qcoh(\rhom(U,V))$. Here $\pi$ is the projection $U\times \rhom(U,V)\to \rhom(U,V)$ and $ev: U\times \rhom(U,V)\to V$ is the evaluation map. To see this we consider the diagram
 
$$
\xymatrix{
\rhom(U,V)\times U\ar[r]^-{ev} \ar[d]^{\pi}& V\\
\rhom(U,V)
}
$$

\noindent together with the fact that by definition we have  $\mathrm{Qcoh}(\rhom(U,V))\simeq lim_{Spec(A)\to\rhom(U,V) } \mathrm{D}(A)$ and the definition of tangent stack: If $V$ is a derived stack, we denote by
$$TV:= \rhom(Spec(\mathbb{C}[\epsilon]), V)$$
the derived stack of morphisms, endowed with the natural map $TV\to V$ given by the composition with the natural inclusion of the point $Spec(\mathbb{C})\to Spec(\mathbb{C}[\epsilon])$. The derived stack $TV$ is therefore completely determined by the cotangent complex of $V$ in the sense that for any $u:Spec(A)\to V$ we have

$$
Map_{\C/V}(Spec(A), TV) \simeq Map_{\mathrm{D}(A)}(A, u^*\mathbb{T}_V)\simeq Map_{\mathrm{D}(A)}(u^*\mathbb{L}_V) , A)
$$

In this case we see that by definition we have

$$T\rhom(U,V)= \rhom(Spec(\mathbb{C}[\epsilon]), \rhom(U V))\simeq\rhom(U, TV)$$

\noindent so that, given  $x_u:Spec(A)\to \rhom(U,V)$ determined by $u:Spec(A)\times U\to V$ we have that $Map_{\C/V}(Spec(A), TV)$ is equivalent to the space of extensions
$$
\xymatrix{
& \rhom(U, TV)\ar[d]\\
Spec(A)\ar[r]^{x_u}\ar@{-->}[ur]& \rhom(U,V)
}
$$

\noindent which by definition of  $\rhom$, is the space of all extensions

$$
\xymatrix{
& TV\ar[d]\\
Spec(A)\times U\ar[r]^-{u}\ar@{-->}[ur]& V
}
$$

\noindent which is equivalent to the space

$$Map_{\C/V}(Spec(A)\times U, TV)\simeq Map_{Qcoh(Spec(A)\times U)}(\mathcal{O}_{Spec(A)\times U}, u^*\mathbb{T}_V)\simeq  Map_{\mathrm{D}(A)}( A, p_* u^* \mathbb{T}_V)$$ 

\noindent with $p: Spec(A)\times U\to Spec(A)$ the projection.\\

Using the fact that  $u= (ev\circ (x_u\times id))$, the base-change property for the diagram
$$
\xymatrix{
Spec(A)\times U \ar[d]^p \ar[r]^-{x_u\times Id}& \rhom(U,V)\times U\ar[d]^{\pi}\\
Spec(A)\ar[r]^{x_u}& \rhom(U,V)
}$$

\noindent gives us

$$
p_* u^* \simeq x_u^* \pi_* ev^*
$$

The descent property for tangent complexes and reduction to the affine case allows us to conclude the proof of the formula (\ref{eq-tangentmappingstack}).

\medskip

Finally, the formula (\ref{eq-tangentmappingstack}) admits a relative version which we can apply to the diagram of derived stacks

\begin{equation}
\label{obstructionad-hoc2}
\xymatrix{
\mathcal{C}_{0.n, \beta}\times_{\Mnprestable}\stablemaps\ar[r]^-{ev}\ar[d]^{\pi}& X\times \Mnprestable\ar[d]^u\\
\stablemaps\ar[r]^{q}& \Mnprestable
}
\end{equation}

\noindent to deduce an equivalence between the relative tangent complexes $\pi_*ev^*\mathbb{T}_X\simeq \mathbb{T}_q$. The diagrams (\ref{obstructionad-hoc}) and (\ref{obstructionad-hoc2}) fit in a larger commutative diagram

\begin{equation}
\label{obstructionad-hoc3}
\xymatrix{
\ar[dr]^i\mathcal{C}_{0.n, \beta}\times_{\Mnprestable}t_0(\stablemaps)\ar[rr]^{ev_0}\ar[dd]^{\pi_0}&& X\times \Mnprestable\ar[dd]^u\\
&\mathcal{C}_{0.n, \beta}\times_{\Mnprestable}\stablemaps\ar[dd] \ar[ur]^{ev}\ar[dd]&\\
t_0(\stablemaps)\ar[rr]^{q_0}\ar[dr]^j&& \Mnprestable\\
&\stablemaps\ar[ur]^{q}&
}
\end{equation}

\noindent where $i$ and $j$ are the truncation maps. Moreover, the face with the truncation maps is a pullback square as $\mathcal{C}_{0.n, \beta}$ is already truncated. As these are derived Deligne-Mumford stacks we can apply the base-change formulas (again, see \cite[Cor. 1.4.5]{1108.5351} or \cite[B.15]{1402.3204}) and deduce that $(\pi_0)_* (ev_0)^* \mathbb{T}_X\simeq (\pi_0)_*i^* ev^*\mathbb{T}_X\simeq j^*\pi_*ev^*\mathbb{T}_X\simeq j^*\mathbb{T}_q$. To conclude we contemplate that the natural map $j^*( \mathbb{T}_q)^{\vee}\to \mathbb{L}_{q_0}$ can be naturally identified with the map constructed in \cite[Prop 6.2]{MR1437495}.\\
\end{proof}
\end{prop}

Here's the current status of the situation. We have a classical stack $Y=t_0(\stablemaps)$ and a derived enhancement $F=\stablemaps$ which produces an obstruction theory $t$ and therefore a second derived enhancement $\R\mathrm{Obs}(t)$ of $Y$. These fit in a diagram

$$
\xymatrix{
\R\mathrm{Obs}(t)& \ar@{_{(}->}[l]_i   Y\ar@{^{(}->}[r]^j & F
}
$$

\noindent which is an isomorphism after truncation and therefore, in $\Gtheory$-theory groups.

$$
\xymatrix{
\Gtheory_0(\R\mathrm{Obs}(t))& \ar[l]^-{i_*}_-{\sim}  \Gtheory_0(Y)\ar[r]^{j_*}_{\sim} & \Gtheory_0(F)
}
$$

To complete the proof one must show that $\virtualsheaf(t):=(i_*)^{-1}(\mathcal{O}_{\R\mathrm{Obs}(t)})$ is equal to $(j_*)^{-1}(\mathcal{O}_F)$ in $\Gtheory_0(Y)$.

\begin{prop}
\label{prop-comparisonvirtualsheaves}
One has an equality of G-theory classes between $\virtualsheaf(t):=(i_*)^{-1}(\mathcal{O}_{\R\mathrm{Obs}(t)})$ and $(j_*)^{-1}(\mathcal{O}_F)$ in $\Gtheory_0(Y)$.
\begin{proof}
In fact, this identification has already been established in the case our derived stack is assumed to be embedded in a smooth stack: in \cite[Proof of Thm 3.3]{MR2496057} Kapranov-Fontanine identified the two classes in the case of quasi-smooth dg-manifolds (a possible incarnation of derived schemes) and more recently in \cite[Section 7.2.2]{1208.6325} the authors explain how the identification of the two classes for quasi-smooth derived schemes follows from: deformation to the normal cone  together with $\mathbb{A}^1$-invariance of $\Gtheory$-theory. 
To conclude, we remark that our situation is known to admit a global resolution in the sense required: see the discussion in \cite[Section 2.3]{MR2040281} and \cite[Appendix A]{MR1666787}. In the genus zero case the situation becomes simpler as this global resolution is given by the closed embedding of $\stablemaps$ in $\mathbb{R}\overline{\mathcal{M}}_{0,n}(\mathbb{P}^n, \beta)$ given by the fact $X$ is projective. The last derived stack is known to be smooth because the projective space is convex.
\end{proof}
\end{prop}

%

\vspace{0.5cm}

Finally, with the two virtual sheaves identified, it is clear that 

\begin{cor}
\label{cor-comparisonwithLee}
Our lax action produces the same  $\Ktheory$-theoretic classes of Lee.
\begin{proof}
Given $E_1,.., E_n \in \Ktheory_0(X)=\Gtheory_0(X)$ and $P\in \Ktheory_0(\Mnstable)=\Gtheory_0(\Mnstable)$, the $\Ktheory$-invariants of Lee are defined by the Euler characteristics

\begin{equation}
\label{Kinvariantsformula}
\chi_{0,n,\beta}(\virtualsheaf\otimes j^*(ev^*(E_1,..., E_n)\otimes\mathrm{Stb}^{*}(P))):= q_*(\virtualsheaf\otimes j^*(ev^*(E_1,..., E_n)\otimes\mathrm{Stb}^{*}(P)))\simeq
\end{equation}

\begin{equation}
\label{Kinvariantsformula3}
\simeq p_*j_*(\virtualsheaf\otimes j^*(ev^*(E_1,..., E_n)\otimes\mathrm{Stb}^{*}(P)))
\end{equation}

\noindent which by the projection formula for $j$ are equivalent to

\begin{eqnarray}
\label{Kinvariantsformula2}
\simeq p_*(j_*(\virtualsheaf)\otimes ev^*(E_1,..., E_n)\otimes\mathrm{Stb}^{*}(P))
\end{eqnarray}

\noindent and finally, by the comparison arguments above,

$$
\simeq p_*(\mathcal{O}_{\stablemaps}\otimes ev^*(E_1,..., E_n)\otimes \mathrm{Stb}^{*}(P))
$$

\noindent which, by definition, are the $\Ktheory$-theoretic Gromov-Witten numbers produced from our lax action.\\
\end{proof}
\end{cor}

\bibliographystyle{alpha}	
\bibliography{biblio}

\end{document}